\documentclass{article}
\usepackage[utf8]{inputenc}

\usepackage{amssymb}
\usepackage{amsthm}
\usepackage{amsmath,amscd}
\usepackage[mathscr]{euscript}
\usepackage[all]{xy}
\usepackage{lmodern}
\usepackage[T1]{fontenc}
\usepackage[textwidth=14cm,hcentering]{geometry}
\usepackage[colorlinks=true,linkcolor=red,citecolor=blue]{hyperref}
\usepackage{cleveref}
\usepackage{mathtools}
\usepackage{tikz}
\usetikzlibrary{cd}
\usetikzlibrary{arrows,positioning}

\usepackage{aliascnt}

\title{Locally rigid $\infty$-categories}
\author{Maxime Ramzi}
\date{}

\newtheorem{thm}{Theorem}[section]
\newaliascnt{lm}{thm}  
\newtheorem{lm}[lm]{Lemma}
\aliascntresetthe{lm}
\Crefname{lm}{Lemma}{Lemmas}
\newaliascnt{prop}{thm}  
\newtheorem{prop}[prop]{Proposition}
\aliascntresetthe{prop}
\Crefname{prop}{Proposition}{Propositions}
\newaliascnt{cor}{thm}  
\newtheorem{cor}[cor]{Corollary}
\aliascntresetthe{cor}
\Crefname{cor}{Corollary}{Corollaries}
\newaliascnt{add}{thm}  
\newtheorem{add}[add]{Addendum}
\aliascntresetthe{add}
\Crefname{add}{Addendum}{Addenda}

\newtheorem*{thm*}{Theorem}
\newtheorem*{cor*}{Corollary}
\newtheorem{thmx}{Theorem}

\theoremstyle{definition}
\newaliascnt{defn}{thm}  
\newtheorem{defn}[defn]{Definition}
\aliascntresetthe{defn}
\Crefname{defn}{Definition}{Definitions}
\newaliascnt{cons}{thm}  
\newtheorem{cons}[cons]{Construction}
\aliascntresetthe{cons}
\Crefname{cons}{Construction}{Constructions}
\newaliascnt{nota}{thm}  
\newtheorem{nota}[nota]{Notation}
\aliascntresetthe{nota}
\Crefname{nota}{Notation}{Notations}
\newaliascnt{conv}{thm}  

\aliascntresetthe{conv}
\Crefname{conv}{Convention}{Conventions}
\newaliascnt{ex}{thm}  
\newtheorem{ex}[ex]{Example}
\aliascntresetthe{ex}
\Crefname{ex}{Example}{Examples}
\newaliascnt{rmk}{thm}  
\newtheorem{rmk}[rmk]{Remark}
\aliascntresetthe{rmk}
\Crefname{rmk}{Remark}{Remarks}
\newaliascnt{ques}{thm}  
\newtheorem{ques}[ques]{Question}
\aliascntresetthe{ques}
\Crefname{ques}{Question}{Questions}
\newaliascnt{conj}{thm}  
\newtheorem{conj}[conj]{Conjecture}
\aliascntresetthe{conj}
\Crefname{conj}{Conjecture}{Conjectures}
\newaliascnt{warn}{thm}  
\newtheorem{warn}[warn]{Warning}
\aliascntresetthe{warn}
\Crefname{warn}{Warning}{Warnings}
\newaliascnt{obs}{thm}  
\newtheorem{obs}[obs]{Observation}
\aliascntresetthe{obs}
\Crefname{obs}{Observation}{Observations}
\newtheorem*{ques*}{Question}
\newtheorem*{rmk*}{Remark}
\newtheorem*{ex*}{Example}
\newtheorem*{defn*}{Definition}
\newaliascnt{recoll}{thm}  

\aliascntresetthe{recoll}
\Crefname{recoll}{Recollection}{Recollections}

\AtBeginEnvironment{defn}{\pushQED{\qed}}
\AtEndEnvironment{defn}{\popQED\enddefn}
\AtBeginEnvironment{cons}{\pushQED{\qed}}
\AtEndEnvironment{cons}{\popQED\endcons}
\AtBeginEnvironment{assu}{\pushQED{\qed}}
\AtEndEnvironment{assu}{\popQED\endassu}
\AtBeginEnvironment{nota}{\pushQED{\qed}}
\AtEndEnvironment{nota}{\popQED\endnota}
\AtBeginEnvironment{conv}{\pushQED{\qed}}
\AtEndEnvironment{conv}{\popQED\enddefn}\AtBeginEnvironment{ex}{\pushQED{\qed}}
\AtEndEnvironment{ex}{\popQED\endex}
\AtBeginEnvironment{exs}{\pushQED{\qed}}
\AtEndEnvironment{exs}{\popQED\endexs}
\AtBeginEnvironment{exc}{\pushQED{\qed}}
\AtEndEnvironment{exc}{\popQED\endexc}
\AtBeginEnvironment{rmk}{\pushQED{\qed}}
\AtEndEnvironment{rmk}{\popQED\endrmk}
\AtBeginEnvironment{ques}{\pushQED{\qed}}
\AtEndEnvironment{ques}{\popQED\endques}
\AtBeginEnvironment{conj}{\pushQED{\qed}}
\AtEndEnvironment{conj}{\popQED\endconj}
\AtBeginEnvironment{warn}{\pushQED{\qed}}
\AtEndEnvironment{warn}{\popQED\endwarn}
\AtBeginEnvironment{obs}{\pushQED{\qed}}
\AtEndEnvironment{obs}{\popQED\endobs}

\newcommand{\op}{^{\mathrm{op}}}
\newcommand{\cat}{\mathbf}

\newcommand{\trcl}{\mathrm{trcl}}
\newcommand{\Cat}{\cat{Cat}}
\newcommand{\on}{\operatorname}
\newcommand{\id}{\mathrm{id}}
\newcommand{\Fun}{\on{Fun}}
\newcommand{\map}{\on{map}}
\newcommand{\Map}{\on{Map}}

\newcommand{\NN}{\mathbb N}

\newcommand{\Sph}{\mathbb S}

\newcommand{\Ss}{\cat S}
\newcommand{\Sp}{\cat{Sp}}

\newcommand{\PrL}{\cat{Pr}^\mathrm{L} }
\newcommand{\PrR}{\cat{Pr}^R}
\newcommand{\Alg}{\mathrm{Alg}}
\newcommand{\CAlg}{\mathrm{CAlg}}
\newcommand{\LMod}{\cat{LMod}}
\newcommand{\RMod}{\cat{RMod}}
\newcommand{\Mod}{\cat{Mod}}

\newcommand{\Ind}{\mathrm{Ind}}

\newcommand{\pt}{\mathrm{pt}}
\newcommand{\colim}{\mathrm{colim}}
\newcommand{\fib}{\mathrm{fib}}

\newcommand{\Sh}{\cat{Sh}}

\newcommand{\dbl}{{\mathrm{dbl}}}
\newcommand{\at}{{\mathrm{at}}}
\newcommand{\one}{\mathbf{1}}

\newcommand{\V}{\mathcal V}
\newcommand{\W}{\mathcal W}

\newcommand{\st}{\mathrm{st}}
\newcommand{\Dbl}[1]{\Mod(#1)^\dbl}
\newcommand{\Prdbl}{(\PrL_{\st})^\dbl}

\newcommand{\M}{\mathcal M}
\newcommand{\N}{\mathcal N}

\newcommand{\C}{\mathcal C}

\newcommand{\PP}{\mathcal P}
\newcommand{\Rig}{\mathrm{Rig}}
\newcommand{\rig}{^{\mathrm{rig}}}

\DeclareFontFamily{U}{min}{}
\DeclareFontShape{U}{min}{m}{n}{<-> udmj30}{}

\newcommand{\category}{$\infty$-category}
\newcommand{\categories}{$\infty$-categories}

\begin{document}
\maketitle
\begin{abstract}
    We develop the theory of locally rigid and rigid symmetric monoidal $\infty$-categories over an arbitrary base $\mathcal{V}\in\mathrm{CAlg}(\mathbf{Pr}^\mathrm{L})$. Among other things, we prove that every locally rigid commutative $\mathcal{V}$-algebra arises as a ``completion'' of a rigid commutative $\mathcal{V}$-algebra. 
    
    Along the way, we introduce and study ``$\mathcal{V}$-atomic morphisms'', which are analogues of compact morphisms over an arbitrary base $\mathcal{V}$. 
\end{abstract}
\tableofcontents

\setcounter{secnumdepth}{0}
\section{Introduction}
Recent advances in algebraic $K$-theory and in condensed mathematics have seen the rise of the study of \emph{dualizable} presentable \categories{} and related objects, as generalizations of the classical compactly generated \categories. 

Typically, one studies not only compactly generated \categories, but also (symmetric) monoidal versions of them. Among those, the really well behaved ones are the \emph{compactly rigidly generated} \categories, namely those for which compact objects coincide with dualizable objects. Typical examples include the \category{} of spectra, or more generally the \category{} of $R$-module spectra for any commutative ring spectrum $R$, but also slightly more exotic objects such as the genuine equivariant stable homotopy category $\Sp_G$. They are the main objects of study of classical ``tt-geometry'', up to the difference between tensor-triangulated categories and stably symmetric monoidal $\infty$-categories. 

The rise of dualizable \categories{} therefore requires the rise of a notion of symmetric monoidal dualizable category analogous to this one, without requiring a sufficient supply of compact/dualizable objects. This notion is that of a \emph{rigid} \category, as introduced by Gaitsgory and Rozenblyum, see e.g. \cite[Chapter 1, §9]{GR} or \cite[Section 2.2]{HSSS} for accounts. We view rigid \categories{} as presentably symmetric monoidal \categories{} that are as close as they could be to being ``rigidly-compactly generated'' without quite being compactly generated. A good example to keep in mind is the \category{} $\Sh(X)$ of sheaves of spectra on a compact Hausdorff space $X$. 

The main goal of this paper is to survey in some depth the foundations of the notion of rigid \categories, and how they relate to dualizable \categories, and to develop their theory relative to an arbitrary base. As in \cite{maindbl}\footnote{Of which the present paper can be viewed as a multiplicative companion}, our first result is a presentability result: \begin{thmx}\label{thmA}
    Let $\V\in\CAlg(\PrL)$. The full subcategory $\CAlg\rig_\V\subset \CAlg(\Mod_\V(\PrL))$ spanned by rigid $\V$-algebras is presentable. 
\end{thmx}

Along the way, we are naturally led to the study of \emph{atomic morphisms}, which are to atomic objects what compact morphisms are to compact objects  - or maybe they are to compact morphisms what atomic objects are to compact objects (see \cite{maindbl} for the relevant notions of atomic objects and compact morphisms). In any case, they are a relative notion of compact morphisms. A second goal of this paper is also to exposit the basic properties of atomic morphisms, and try to explain how one might manipulate them - they are slightly more finnicky than compact morphisms in the stable setting, though they have the advantage of working in large generality. To illustrate the analogy, we prove the following theorem, which is meant to generalize \cite[Theorem 2.39]{maindbl} - see \Cref{defn:atpres} for precise definitions: 
\begin{thmx}\label{thmB}
    Let $\V\in\CAlg(\PrL)$. A $\V$-module $\M$ is dualizable if and only if every object of $\M$ is atomically presentable; if and only if sequential colimits along atomically presentable morphisms generate $\M$ as a $\V$-module. 
\end{thmx}

Finally, we point out that in fact we strive for a bit more generality. Indeed, even in the compactly generated setting, the setting of compactly rigidly generated \categories{}  is often too restrictive and does not account for behaviours of ``completions'' such as the $p$-complete category, or maybe the $T(n)$-local category, where there is in fact a sufficient supply of dualizable objects, but they are not all compact. In the compactly generated setting, this can be fixed by only stipulating that all compacts are dualizable, but not the converse, and in the dualizable setting, this is fixed by introducing the notion of \emph{locally rigid} symmetric monoidal \category\footnote{Called semi-rigid in \cite[Appendix C]{kazhdan}.}. Beyond the compactly generated examples we just mentioned, this allows for example to include the examples of \categories{} of sheaves on \emph{locally} compact Hausdorff spaces, as opposed to only the compact Hausdorff ones.

The precise goal of this paper is therefore to study the theory of locally rigid \categories, and relate it to that of rigid \categories{} and atomic morphisms. 

Our main result in the realm of locally rigid \categories{} (and, arguably, in this paper) is the following result which precisely relates locally rigid and rigid \categories - it can be viewed as stating that locally rigid \categories{} all arise as ``completions'' of rigid ones:
\begin{thmx}\label{thmC}
    Let $\V\in\CAlg(\PrL)$. For any locally rigid commutative $\V$-algebra $\W$, there exists a canonical rigid commutative $\V$-algebra $\Rig_\V(\W)$ with a localization $\Rig_\V(\W)\to \W$ admitting a fully faithful, $\V$-linear \emph{left} adjoint. 
\end{thmx}
This rigidification exists more generally, though the existence and fully faithfulness of the left adjoint are specific to locally rigid algebras. In the case where $\V=\Sp$ or $\Mod_R$ for some commutative ring spectrum, this rigidification is closely related to Clausen and Scholze's categories of nuclear modules, though we will not give a precise statement here. 

\subsection{Prerequisites}
Much of this paper relies on and assumes the basics of the theory of $\V$-modules for $\V\in~\CAlg(\PrL)$, and more specifically the theory of dualizable $\V$-modules. What is needed is either covered in the companion paper \cite{maindbl}, or is covered independently in the present paper. The (strongly) recommended order of reading is to begin with \cite{maindbl} to at least extract some of the main results such as \cite[Theorem 1.49,Corollary 1.59, Proposition 1.62, and Corollary 2.36]{maindbl}, which will be used routinely in this paper; but also to understand in what sense the development of atomic morphisms here is parallel to the development of compact maps there. 
\subsection{Outline}
In \Cref{section:prelims}, we review preliminaries concerning weighted colimits in enriched \categories, which will be useful later on to set up the notion of ``atomic presentation'', which is our relative analogue of the compact exhaustions of \cite[Section 2.2]{maindbl}. 

In \Cref{section:atomic}, we dive into the heart of the matter by introducing atomic morphisms and studying their basic properties, which are meant to be analogous to those of compact morphisms. This is where we prove \Cref{thm:dblexhaustionV}, namely \Cref{thmB} from the introduction. 

In \Cref{section:traceclass}, we recall the definition and basic properties of trace-class morphisms which will be crucial to the study of rigidity. 

\Cref{section:rigandlocrig} is the main section of the paper, this is where we introduce and study local rigidity and rigidity. We extend and give new proofs of several results from \cite[Appendix C]{kazhdan}, and give an analogue of \Cref{thm:dblexhaustionV} for rigid categories, namely \Cref{lm:rigiditythroughpres}. Throughout this section, we explain how the story simplifies if we fix $\V=\Sp$, and prove the relevant analogous results. Finally, it is in this section that we prove the existence of, and give a ``concrete'' construction of rigidifications, allowing us to prove \Cref{thm:locrigrigleft}, namely \Cref{thmC} from the introduction. 

Finally, in \Cref{section:pres}, we do some set theory and prove \Cref{thm:rigprl}, namely \Cref{thmA}.
\subsection{Relation to other work}
As mentioned in the introduction, the notion of rigid \category{} comes from the work of Gaitsgory and Rozenblyum \cite{GR}, and was further studied in \cite{HSSS}; and that of locally rigid \category{} comes from \cite{kazhdan}. Some of the results we prove here can be found there, at least in the special case $\V$ is itself rigid over $\Sp$.  

Some of it is also inspired by the work of Clausen and Scholze on nuclear categories, and some of it also comes directly from the horse's mouth, such as \Cref{cor:rigtrcl} and the corresponding \Cref{cor:trclexQ}, that I learned directly from Dustin Clausen. 
\subsection{Conventions and notations}
We work throughout with the theory of \categories{} as developped thoroughly by Lurie in \cite{HTT,HA}, although we try to take a mostly model-independent approach. We also use the results and notations from \cite{maindbl}. 
\begin{itemize}
\item Unless there is a specific need for it, we do not adress set-theoretic issues. $\widehat{\Cat}$ denotes the \category{} of large \categories.
\item We use $\Ss$ for the \category{} of spaces/anima and $\Sp$ for the \category{} of spectra. 
    \item We use $\PrL$ to denote the (very large) \category{} of presentable \categories, and $\PrL_{\st}$ the full subcategory spanned by the stable ones; we often view it as a symmetric monoidal \category{} through the Lurie tensor product - this is the only one we consider, so we keep it implicit throughout. We use $\Fun^L$ to denote the \category{} of colimit-preserving functors. For a regular cardinal $\kappa$, $\PrL_\kappa$ denotes the non-full subcategory of $\PrL$ spanned by $\kappa$-compactly generated \categories{} and $\kappa$-compact preserving left adjoints (equivalently, left adjoints whose right adjoint preserves $\kappa$-filtered colimits).
    \item For a symmetric monoidal \category{} $\M$, we let $\CAlg(\M)$ denote the \category{} of commutative algebras in $\M$, and for $A\in\CAlg(\M)$, we let $\Mod_A(\M)$ denote its \category{} of modules\footnote{Canonically, we mean $\mathbb E_\infty$-modules, although they are equivalent to left modules.}. For a non-necessarily commutative algebra $A$, we write $\LMod$ or $\RMod$ to indicate whether we consider left or right modules. 
    \item We use $\Map$ for mapping spaces, $\map$ for mapping spectra in stable \categories, and $\hom$ for internal homs, or enriched homs. 
    \item  As in \cite{maindbl}, our main reference for enriched categories is \cite{heine}. We use $\Cat_\V$ to denote the category of $\V$-enriched categories\footnote{$\widehat{\Cat}_\V$ if we want to allow large ones.}, $\Fun_\V$ to denote enriched functor (ordinary) categories, and $\Fun^L_\V$ to denote categories of functors in $\Mod_\V(\PrL)$, and $\PP_\V$ to denote $\V$-enriched presheaves.  
    \item As in \cite{maindbl}, we will often tacitly consider stable \categories{} to be special cases of $\Sp$-enriched \categories, cf. \cite[Example 1.5]{maindbl}
    \item We use both $j$ and $y$ for the Yoneda embedding. 
\end{itemize}

From now on, the word ``category'' means \category{} by default, and we write $1$-category when we want to single out ordinary categories. 
\subsection{Acknowledgements}
The acknowledgements for this paper are essentially the same as for the companion \cite{maindbl}. I specifically need to thank Robert Burklund for asking specifically about locally rigid categories rather than restricting to rigid ones. I furthermore wish to thank Jiacheng Liang for his many helpful comments, and particularly for spotting the mistake recorded in \Cref{rmk:mistake}. 

Other than that, I wish again to thank Dustin Clausen, Sasha Efimov, Thomas Nikolaus, Achim Krause, Lior Yanovski and Shay Ben Moshe for helpful conversations and feedback on the subject at hand, particularly Dustin and Sasha for explaining many of the statements that appear here. 

This work was supported by the Danish National Research Foundation through the
Copenhagen Centre for Geometry and Topology (DNRF151). Part of the results presented here were obtained while I was visiting Harvard and part of them while I was in Münster, and so I wish to thank them for their hospitality, as well as Mike Hopkins and Thomas Nikolaus for inviting me there. 

\setcounter{secnumdepth}{3}
\section{Preliminaries: weighted colimits}\label{section:prelims}
The goal of this section is to gather basic preliminaries about weighted colimits in the context of enriched categories, which we will need later on.  Throughout this section, we fix $\V\in\CAlg(\PrL)$, to be our ``enriching base''. 

This section is mostly intended to set up notation and verify that basic properties from ordinary enriched category theory carry over to categories. We recommend the reader have a look at \cite[Section 1]{maindbl} for reminders on $\V$-enriched categories and their relation to $\V$-modules - most of the material there is extracted from \cite{heine}. The main result which we will use here and was recalled there is :
\begin{thm}[{\cite[Theorem 1.13]{heine}}]\label{thm: UPVPsh}
    Let $\M_0$ be a small $\V$-category, and $\N$ a presentable $\V$-module. Restriction along the Yoneda embedding $\M_0\to \PP_\V(\M_0)$ induces an equivalence $$\Fun^L_\V(\PP_\V(\M_0), \N)\simeq \Fun_\V(\M_0,\N)$$
\end{thm}
With this in mind, we can move on to weighted colimits. 

The following is a standard name in classical enriched category theory for $\V$-presheaves:
\begin{defn}
    Let $I$ be a $\V$-enriched category. A \emph{weight} on $I$ is a $\V$-functor $I\op\to \V$.
\end{defn}
The point of this name is the following:
\begin{defn}
    Let $I$ be a small $\V$-enriched category and $\M\in\Mod_\V(\PrL)$. We define the weighted colimit functor $\colim_I^{(-)} (-)$ as follows: 
    $$\Fun_\V(I,\M)\times \Fun_\V(I\op,\V)\simeq \Fun^L_\V(\PP_\V(I),\M)\times \PP_\V(I)\to \M$$
    where the first equivalence comes from \Cref{thm: UPVPsh}. 

    Given a weight $W$ and a functor $f: I\to \M$, we denote the resulting object of $\M$ by $\colim_I^Wf$; we call $f$ the diagram and $W$ the weight.
\end{defn}
\begin{rmk}
    Given a functor $f:I\to \M$, the induced $\V$-linear left adjoint $f_! : \PP_\V(I)\to \M$ admits a right adjoint $f^*: m\mapsto \hom(f(-),m)$. Thus, we have an equivalence $$\hom(\colim_I^W f, m)\simeq \hom(W,\hom(f(-),m))$$
    This allows to define weighted colimits more generally, as objects satisfying this equivalence if they exist\footnote{This is beyond the scope of this paper, but this allows to also define ``absolute weights'', which are weights $W$ such that $\colim_I^W(-)$ is preserved whenever it exists. Once this notion is set up, one may try to compare Karoubi complete $\V$-categories, that is, those $\V$-categories that admit all absolute weighted colimits, and atomically generated $\V$-modules. They should be equivalent along the $\PP_\V\dashv (-)^\at$ adjunction from \cite{Shay}.}. 
\end{rmk}
\begin{obs}
    As is clear from the definition, weighted colimits are functorial in the weight $W$ and in the diagram $f$. They preserve colimits in both variables. 
\end{obs}
\begin{cons}\label{cons:compweight}
Let $L:\M\to \N$ be a lax $\V$-linear map between $\V$-modules.

For a small $\V$-enriched category $I$ and weight $W$, we obtain a natural transformation $\colim_I^W L\circ (-)\to L(\colim_I^W -)$ between functors $\Fun_\V(I,\M)\to \N$. To construct it, we note that the functor $\Fun_\V(I,\M)\simeq \Fun_\V^L(\PP_\V(I),\M)\subset \Fun_\V(\PP_\V(I),\M)$ is left adjoint to restriction, and thus the commutative diagram: 
\[\begin{tikzcd}
	{\Fun_\V(\PP_\V(I),\M)} & {\Fun_\V(I,\M)} \\
	{\Fun_\V(\PP_\V(I),\N)} & {\Fun_\V(I,\N)}
	\arrow[from=1-1, to=1-2]
	\arrow[from=1-1, to=2-1]
	\arrow[from=1-2, to=2-2]
	\arrow[from=2-1, to=2-2]
\end{tikzcd}\]
induces a mate of the form: \[\begin{tikzcd}
	{\Fun_\V(\PP_\V(I),\M)} & {\Fun_\V(I,\M)} \\
	{\Fun_\V(\PP_\V(I),\N)} & {\Fun_\V(I,\N)}
	\arrow[from=1-1, to=2-1]
	\arrow[from=1-2, to=1-1]
	\arrow[from=1-2, to=2-2]
	\arrow[shorten <=8pt, shorten >=8pt, Rightarrow, from=2-2, to=1-1]
	\arrow[from=2-2, to=2-1]
\end{tikzcd}\] which, after evaluation at $W$, induces the desired map. 

In particular it is not difficult to see that the square is adjointable when $L$ is $\V$-linear and colimit-preserving, so that in this case, the canonical map is an equivalence: $\V$-linear left adjoints \emph{preserve weighted colimits}. 
\end{cons}
We now give some simple examples of weighted colimits.
\begin{ex}
    Let $I$ be a non-enriched category, and $\one_\V\otimes I$ the corresponding $\V$-category where the hom objects are base-changed from $I$. 

    Let $\underline\one: (\one_\V\otimes I)\op \to \V$ be the constant weight with value $\one_\V$. In this case, combining universal properties we find that $\colim_{\one_\V\otimes I}^{\underline\one} (-): \Fun(I,\M)\simeq \Fun_\V(\one_\V\otimes I,\M)\to \M$ is the ordinary $I$-indexed colimit functor. 
    
Thus ordinary colimits are special cases of weighted colimits - they are called ``conical colimits''. 
\end{ex}
\begin{ex}
    Let $I=\{\one_\V\}$ be the $\V$-category with one object with endomorphism object $\one_\V$ (this is alternatively $\one_\V\otimes\pt$ with the notation from the previous example). 

    A weight on $I$ is equivalently an object $v\in\V$, and a diagram on $I$ with values in $\M$ is equivalently an object $m\in \M$. We find that $\colim_{\one_\V}^v m = v\otimes m$. So tensors with objects of $\V$ are special cases of weights. 
\end{ex}
Combining the two previous examples, we find:
\begin{cor}\label{cor:Vliniffweight}
    Suppose $L:\M\to \N$ is a $\V$-enriched functor between presentable $\V$-modules. $L$ is $\V$-linear and colimit-preserving if and only if it preserves weighted colimits. 
\end{cor}

The last preliminary we need concerns the interaction of weighted colimits with change of base. 
\begin{nota}
    For $f:\V\to\W$ a lax symmetric monoidal functor, we let $f_\#$ denote the functor $\Cat_\V\to\Cat_\W$ that simply applies $f$ to the hom-objects. 

    If $f$ is strong symmetric monoidal and has a right adjoint $f^R$, $f_\#$ is left adjoint to $(f^R)_\#$, and we let $f^*$ denote the latter. 
\end{nota}
\begin{obs}
    Let $\M$ be a presentable $\W$-module. Then $f^*\M$ is exactly $\M$ viewed as $\V$-module by restriction of scalars.
\end{obs}
The compatibility we need is the following:
\begin{prop}\label{prop:compatibilityweightedcolim}
    Let $f:\V\to\W\in\CAlg(\PrL)$, and let $\M$ be a presentable $\W$-module. For a small $\V$-category $I$, let $W: I\op\to \V$ be a $\V$-weight, and $x_\bullet: I\to f^*\M$ be a $\V$-diagram. 

    Consider $fW: I\op\to \V\to \W$ and the $\W$-weight $f_!W: f_\#I\op\to \W$ it induces by adjunction, as well as $\tilde x_\bullet: f_\# I\to \M$ the induced $\W$-enriched functor. 

    There is a canonical equivalence $\colim_I^W x_\bullet\simeq \colim_{f_\# I}^{f_!W}\tilde x_\bullet$
\end{prop}
\begin{proof}
 We have the following natural equivalences, natural in $y\in\M$: $$\Map_\M(\colim_I^W x, y)\simeq \Map_{\Fun_\V(I\op, \V)}(W,\hom_\M^\V(x_\bullet,y)) \simeq \Map_{\Fun_\V(I\op, \V)}(W,f^R\hom_\M^\W(x_\bullet,y))$$
$$\simeq \Map_{\Fun_\V(I\op, \W)}(fW,\hom_\M^\W(x_\bullet,y)) \simeq \Map_{\Fun_\W(f_\#I\op, \W)}(f_!W,\hom_\M^\W(\tilde x_\bullet,y))  \simeq  \Map_\M(\colim_{f_\# I}^{f_!W}\tilde x_\bullet ,y)$$
so the claim follows from the Yoneda lemma.
\end{proof}
One can also give a proof at the level of categories that the resulting functors are $\V$-naturally equivalent. 
\section{Atomic maps}\label{section:atomic}
In this section, we introduce and study $\V$-\emph{atomic maps}. Just as $\V$-atomic objects are the $\V$-linear analogue of compact objects, \cite[Definition 1.22, Example 1.24]{maindbl}, $\V$-atomic maps will be the $\V$-linear analogue of strongly compact maps as introduced in \cite[Definition 2.5]{maindbl}. 

In particular, when $\V=\Sp$, we will see that $\Sp$-atomic maps are closely related to compact maps (see \Cref{cor:atimpcomp} for a precise statement). 

\begin{rmk}
    It could make sense to have introduced atomic maps in \cite{maindbl}, and in fact, an earlier version of \cite{maindbl} contained the contents of the present article. However, I do not know precise applications of atomic maps except for the ones we describe here in relation to (local) rigidity. It only seems reasonable to expect further applications, but in the current situation, this presentation seems to make more sense. 

     I hope that the ideas developed here can be helpful elsewhere too. 
\end{rmk}
 Throughout this section, we fix $\V\in\CAlg(\PrL)$. We will have to care slightly about set theory, so for simplicity of statements we define: 
\begin{defn}
    A regular cardinal $\kappa$ is called \emph{good} if it is uncountable and $\V\in\CAlg(\PrL_\kappa)$. 

    If $\M$ is a presentable $\V$-module, a cardinal $\kappa$ is called \emph{$\M$-good} if it is good and $\M\in~\Mod_\V(\PrL_\kappa)$. 
\end{defn}

\begin{prop}
    Let $\M,\N$ be stable presentable categories. The inclusion $\Fun^L(\M,\N)\to~\Fun^{ex}(\M,\N)$ admits a right adjoint. 

    More generally, let $\M,\N$ be $\V$-modules in $\PrL$. The inclusion $\Fun^L_\V(\M,\N)\to \Fun_{\V}(\M,\N)$ admits a right adjoint. 
\end{prop}
\begin{proof}
    In the stable case, by \cite[Theorem 1.11]{heine}, $\Fun_{\Sp}(\M,\N)\simeq \Fun^{ex}(\M,\N)$ so it suffices to prove the latter claim. 

    Now if $\kappa$ is an $\M$-good cardinal (which can always be arranged up to picking a suitable $\kappa$), we have $\Fun^L_\V(\M,\N)\subset \Fun^{\kappa-filt}_{\V}(\M,\N)$ and $\Fun^{\kappa-filt}_{\V}(\M,\N)\simeq \Fun_{\V}(\M^\kappa,\N)$ under restriction/left Kan extension. 

    In particular, the inclusion $\Fun^{\kappa-filt}_{\V}(\M,\N)\subset \Fun_{\V}(\M,\N)$ does admit a right adjoint, given by restriction to $\M^\kappa$ and left Kan extension, so it suffices to prove that the inclusion $\Fun^L_\V(\M,\N)\subset \Fun^{\kappa-filt}_{\V}(\M,\N)$ admits a right adjoint. But this inclusion clearly preserves colimits, and both sides are presentable, so it admits a right adjoint. Note that $\Fun^{\kappa-filt}_\V(\M,\N)\simeq\Fun_\V(\M^\kappa,\N)\simeq \Fun^L_\V(\PP_\V(\M^\kappa),\N)$ is indeed presentable. 
\end{proof}

\begin{defn}
    Let $\M$ be a $\V$-module, and $x\in\M$. We let $\at_\M^\V(x,-)$, or simply $\at_\M(x,-)$ if $\V$ is clear from context, denote the image of $\hom_\M(x,-) : \M\to \V$ under the right adjoint to the inclusion $\Fun^L_\V(\M,\V)\subset \Fun_{\V}(\M,\V)$; i.e. $\at_\M^\V(x,-)\to \hom_\M(x,-)$ is the terminal $\V$-linear, colimit preserving functor with a lax $\V$-linear map to $\hom(x,-)$.  

    We say a morphism $f:x\to y$ in $\M$ is $\V$-atomic, or atomic if $\V$ is understood, if the classifying morphism $\one_\V\to \hom_\M(x,y)$ admits a lift to $\at_\M(x,y)$ along the canonical map $\at_\M(x,y)\to \hom_\M(x,y)$. 
\end{defn}
\begin{rmk}\label{rmk:atomicstrongcompact}
    Equivalently, a map $f:x\to y$ is $\V$-atomic if and only if there exists a $\V$-linear, colimit preserving functor $h$ with a map $h\to\hom(x,-)$ such that $f $ lifts to $h(y)$. In this sense, $\V$-atomic maps are analogues of strongly compact maps. In particular, for $\V=\Sp$, they are exactly the same as strongly compact maps in the sense of \cite[Definition 2.5]{maindbl}.
\end{rmk}
\begin{cor}\label{cor:atimpcomp}
    Let $\V=\Sp$, and $\M\in\PrL_{\st}$. A map $f:x\to y$ in $\M$ is $\Sp$-atomic if and only if it is strongly compact. If $\M$ is dualizable, then any weakly compact map is $\Sp$-atomic. 
\end{cor}
\begin{proof}
    A priori, strongly compact morphisms are defined in terms of filtered colimit preserving functors $f:\M\to \Ss$ and maps $f\to \Map(x,-)$, while $\Sp$-atomic morphisms are defined in terms of filtered colimit preserving, exact functors $g:\M\to \Sp$ and maps $g\to \map(x,-)$. Thus, it is clear that a $\Sp$-atomic map is strongly compact, since $\Omega^\infty$ preserves filtered colimits and sends $\map(x,-)$ to $\Map(x,-)$. 

    On the other hand, $\Map(x,-):\M\to\Ss$ is a $1$-excisive reduced functor, so if $f:\M\to \Ss$ is filtered colimit preserving and has a map $f\to \Map(x,-)$, then taking an excisive approximation of $f$ and using the fact that exact functors $\M\to\Sp$ are equivalent to excisive reduced functors $\M\to \Ss$, we find that $f\to \Map(x,-)$ factors through a map obtained as $\Omega^\infty$ of a map $g\to \map(x,-)$, where $g:\M\to \Sp$ is filtered colimit preserving\footnote{This uses that excisive approximations of filtered colimit-preserving functors preserve filtered colimits. } and exact. Thus, a strongly compact map is $\Sp$-atomic. 
\end{proof}
We were not able to answer the following question:
\begin{ques}\label{ques:cpctvsSpat}
If $\M$ is not dualizable, is every weakly compact map strongly compact, i.e. $\Sp$-atomic ?
\end{ques}
\begin{proof}

    The first half of the claim is \Cref{rmk:atomicstrongcompact}. The second half follows from the first together with \cite[Example 2.20]{maindbl}.
\end{proof}
\begin{obs}\label{obs:atimpliescpct}
  If the unit $\one_\V$ is $\kappa$-compact, then any $\V$-atomic map is a (strongly) $\kappa$-compact map. 
\end{obs}
\begin{ex}\label{ex:VatV}
    If $\M=\V$, the right adjoint $\Fun_{\V}(\V,\V)\to \Fun^L_\V(\V,\V)\simeq \V$ has an explicit description as $f\mapsto f(\one_\V)\otimes -$. This can be proved using an explicit unit and co-unit. 

    As a consequence, $\at_\V(x,-)\simeq \hom_\V(x,\one_\V)\otimes -$, and so $\V$-atomic maps in $\V$ are the same as \emph{trace-class maps} (see \Cref{defn:trcl} and \Cref{rmk:trcl}). 

    This is analogous to the claim that $\V$-atomic objects in $\V$ are exactly the dualizable objects, cf.\cite[Example 1.23]{maindbl}.
\end{ex}
We can describe $\at_\M(x,-)$ a bit more explicitly when $\M$ is dualizable over $\V$ - in fact, we can more generally describe the $\V$-linear colimit-preserving approximation of any $\V$-functor. To do this, we first recall the following:
\begin{thm}[{\cite[Theorem 1.49]{maindbl}}]
    Let $\M$ be a $\V$-module and $\kappa$ an $\M$-good cardinal. $\M$ is dualizable over $\V$ if and only if the $\V$-linear functor $p:\PP_\V(\M^\kappa)\to \M$ extended from the inclusion $i:\M^\kappa\to \M$ admits a $\V$-linear left adjoint $\hat y: \M\to \PP_V(\M^\kappa)$. 
\end{thm}
Note also that if $\M$ is dualizable, \cite[Theorem 3.4]{maindbl} shows that any good cardinal is $\M$-good.

With this in hand, we can use the functor $\hat y$ to describe $\V$-linear colimit-preserving approximation to functors as follows: 
\begin{prop}\label{prop:descriptionofrightadj}
Let  $\kappa$ be a good cardinal, $\M$ be a dualizable $\V$-module with $\hat y : \M\to~\PP_\V(\M^\kappa)$ the left adjoint to the canonical functor $\PP_\V(\M^\kappa)\to \M$. 

    Let $\N\in\Mod_\V(\PrL)$, and let $R$ denote the right adjoint to $\Fun^L_\V(\M,\N)\to \Fun_{\V}(\M,\N)$. Fix $f : \M\to \N$ a lax $\V$-linear functor. Let $\N_0$ be any small full sub-$\V$-category of $\N$ such that $f(\M^\kappa)\subset \N_0$, and let $\tilde f$ denote the restriction-corestriction of $f$ to a $\V$-functor $\M^\kappa\to \N_0$, and $p:\PP_\V(\N_0)\to \N$ denote the induced $\V$-linear functor. 
    
    In this case, $$R(f)\simeq p\circ \PP_\V(\tilde f)\circ \hat y$$ and the counit map $R(f)\to f$ is given by the following composite : $$p\circ\PP_\V(\tilde f)\circ \hat y\to~p\circ \PP_\V(\tilde f)\circ y\simeq p\circ \tilde f\simeq f$$
\end{prop}
\begin{rmk}
    In the case $\V=\Sp$, $\PP_\Sp$ is given by $\Ind$; in particular $\Ind(\N)$ still makes sense even as $\N$ is large. In this situation, the above description is simpler: it's simply $p\circ \Ind(f)\circ\hat y$, where $p:\Ind(\N)\to \N$ is the canonical functor. 
\end{rmk}
We note that this proposition is related/analogous to the one from \cite[Theorem 1.63]{maindbl}. 
\begin{proof}
First, we decompose the adjunction in two steps: the inclusion $\Fun^L_\V(\M,\N)\subset~\Fun_\V(\M,\N)$ factors through the subcategory $\Fun_\V^{\kappa-filt}(\M,\N)$ of $\kappa$-filtered colimit-preserving $\V$-enriched functors, and the right adjoint to the inclusion $$\Fun^{\kappa-filt}_\V(\M,\N)\subset~\Fun_\V(\M,\N)$$ 
is simply given by restriction to $\M^\kappa$ followed by left Kan extension - indeed, $\Fun^{\kappa-filt}_\V(\M,\N)\simeq~\Fun_\V(\M^\kappa,\N)$ along the restriction/left Kan extension adjunction. 

From this perspective, on a given $f$, we can realize this restriction-followed-by-left-Kan-extension as $p_\N\circ \PP_\V(\tilde f)\circ y$ - indeed this functor is $\kappa$-filtered colimit preserving, and agrees with $\tilde f$ on $\kappa$-compacts. 

Thus we are left with examining the adjunction $\Fun^L_\V(\M,\N)\rightleftarrows \Fun^{\kappa-filt}_\V(\M,\N)$. Using again the description of the latter in terms of $\M^\kappa$, and using \Cref{thm: UPVPsh}, we find that the inclusion is identified with restriction along $p:\PP_\V(\M^\kappa)\to \M$: $$\Fun^L_\V(\M,\N)\xrightarrow{p^*}\Fun^L_\V(\PP_\V(\M^\kappa),\N)$$

Since $\M$ is dualizable and $\kappa$-compactly generated, $p$ admits a left adjoint $\hat y$, and precomposition with $p$ admits precomposition with $\hat y$ as a right adjoint. 

Unwinding the constructions, we find that $R(f)\simeq p\circ \PP_\V(\tilde f)\circ \hat y$, and that the co-unit is as described.
\end{proof}
As a corollary, we obtain the promised description of atomic maps in the case where $\M$ is dualizable:
\begin{cor}\label{cor:descofat}
     Let $\kappa$ be a good cardinal, $\M$ be a dualizable $\V$-module, with $\hat y:~\M\to~\PP_\V(\M^\kappa)$ the left adjoint to the canonical functor $p:\PP_\V(\M^\kappa)\to \M$, and let $y$ be the right adjoint to this canonical functor. 

     For any $x\in \M^\kappa$, we have $$\at_\M(x,-)\simeq \hom_{\PP_\V(\M^\kappa)}(y(x),\hat y(-))$$
     and the map $\at_\M(x)\to \hom(x,-)$ is given by applying $p$ and using $p\circ y \simeq p\circ \hat y\simeq\id_\M$. 
\end{cor}
\begin{proof}
    We specialize \Cref{prop:descriptionofrightadj} to $\N=\V$. Let $\lambda$ be a cardinal such that $\hom_\M(x,-)$ restricted to $\M^\kappa$ lands in $\V^\lambda$. 

    We have $\at_\M(x,-)\simeq p\circ \PP_\V(\hom_\M(x,-))\circ \hat y_\M$. 

    But now, $\hom_{\PP_\V(\M^\kappa)}(y(x),-)$ and $p\circ \PP_\V(\hom_\M(x,-))$ by design agree on $\M^\kappa\xhookrightarrow{y}~\PP_\V(\M^\kappa)$ and are both $\V$-linear and colimit-preserving\footnote{This is where we use that $x$ is in $\M^\kappa$: otherwise $\hom(y(x),-)$ need not preserve colimits or be $\V$-linear.}. Thus they agree, and we are done. 
\end{proof}
\begin{ex}
    Let $\V=\Sp$. In this case, we find that $\at_\M(x,-) \simeq \map(y(x),\hat y(-))$. We could call this the spectrum of compact maps, or more generally, in a compactly assembled category, $\Map(y(x),\hat y(-))$ could be called the space of compacts maps. Note that its map to $\Map(x,-)$ is \emph{not} an inclusion of components. 

    Compare also with \cite[Lemma 2.21]{maindbl}.
\end{ex}

We now give examples of properties of atomic maps analogous to the basic properties of compact maps from \cite[Section 2.1]{maindbl}.
\begin{ex}\label{ex:atid}
    Let $\M$ be a $\V$-module and $x\in \M$. $x$ is atomic if and only if the identity $\id_x$ is atomic, if and only the map $\at_\M(x,-)\to\hom_\M(x,-)$ is an equivalence. 
\end{ex}
\begin{proof}
    If $x$ is atomic, $\hom_\M(x,-)$ by definition is already $\V$-linear and colimit-preserving, so that by definition, $\hom_\M(x,-)\simeq \at_\M(x,-)$. This immediately implies that $\id_x$ is atomic. 

    Conversely, suppose $\id_x$ is atomic. A lift $\tilde i \in \at_\M(x,x)$ induces a map $\hom_\M(x,-)\to~\at_\V(x,-)$ such that the composite $\hom_\M(x,-)\to~\at_\V(x,-)\to~\hom_\M(x,-)$ is the identity (by the Yoneda lemma). The other composite, $\at_\M(x,-)\to \hom_\M(x,-)\to \at_\M(x,-)$ is determined by the further composite $\at_\M(x,-)\to \hom_\M(x,-)\to \at_\M(x,-)\to \hom_\M(x,-)$ by the universal property of $\at_\M$. 

    But this composite is the canonical map, because $\hom_\M(x,-)\to \at_\M(x,-)\to \hom_\M(x,-)$ is the identity, and so the other composite is the identity of $\at_\M(x,-)$. Thus $\at_\M(x,-)\to~\hom_\M(x,-)$ is an equivalence, and $\hom_\M(x,-)$ is $\V$-linear and colimit-preserving, i.e. $x$ is atomic. 
\end{proof}
\begin{ex}\label{ex:atideal}
    The collection of atomic maps in $\M$ is a $2$-sided ideal. 
\end{ex}
\begin{proof}
    Let $f:x\to y$ be an atomic map with a lift $\tilde f\in\at_\V(x,y)$. 
    
For any $g:y\to z$, we have a commutative diagram 
\[\begin{tikzcd}
	{\at_\M(x,y)} & {\at_\M(x,z)} \\
	{\hom_\M(x,y)} & {\hom_\M(x,z)}
	\arrow[from=1-1, to=1-2]
	\arrow[from=1-1, to=2-1]
	\arrow[from=1-2, to=2-2]
	\arrow[from=2-1, to=2-2]
\end{tikzcd}\]
by naturality, so that the image of $\tilde f$ under the top horizontal map gives a witness that $gf$ is atomic. 

For $h:z\to x$, the map $\hom(x,-)\to \hom(z,-)$ induces a map $\at_\M(x,-)\to \at_\M(z,-)$ which similarly sends a witness $\tilde f$ to a witness that $fh$ is atomic.
\end{proof}
We povide a slight strenghtening of this statement in the following form:
\begin{lm}\label{lm:slightlycoherentatid}
    Let $\M$ be a $\V$-module, and $x,y,z\in \M$. There are canonical maps $$\hom_\M(x,y)\otimes \at_\M(y,z)\to \at_\M(x,z), \quad  \at_\M(x,y)\otimes \hom_\M(y,z)\to \at_\M(x,z)$$ fitting into commutative diagrams with the obvious map $\hom_\M(x,y)\otimes \hom_\M(y,z)\to \hom_\M(x,z)$. 
    
Furthermore, the two composites $\at_\M(x,y)\otimes\at_\M(y,z)\to \at_\M(x,z)$ agree. 
\end{lm}
\begin{rmk}
    It is natural to expect these maps to fit into a more coherent collection of maps and transformations, forming something like an ``ideal'' in the $\V$-category $\M$. Setting up this amount of coherence would take us too far afield, and we have not had the need for this amount of structure so far. It would certainly be interesting to have it available, to clarifiy the connection between dualizability and ideals (cf.\cite{KNSh}). 
\end{rmk}
\begin{proof}
    In the first case, we simply note that the functor  $z\mapsto \hom_\M(x,y)\otimes\at_\M(y,z)$ is $\V$-linear and colimit-preserving, so that its natural map $$\hom_\M(x,y)\otimes\at_\M(y,z)\to \hom_\M(x,y)\otimes \hom_\M(y,z)\to \hom_\M(x,z)$$ factors canonically through $\at_\M(x,z)$. 

    In the second case, we note that this is simply part of the data of an enriched functor on $\at_\M(x,-)$. 

    The ``Furthermore'' part is a consequence of the universal property of $\at_\M(x,z)$, as in the first paragraph of the proof. 
\end{proof}

\begin{ex}\label{ex:reflat}
    Let $f:\M\to \N$ be a fully faithful $\V$-linear colimit-preserving functor. The functor $f$ reflects atomic maps.
\end{ex}
\begin{proof}
    We have a map $\hom_\M(x,-)\to \hom_\N(f(x),f(-))$ which is an equivalence by fully faithfulness. We also have a map $\at_\N(f(x),f(-))\to \hom_\N(f(x),f(-))\simeq \hom_\M(x,-)$ where the source is $\V$-linear and colimit-preserving, hence it factors through $\at_\M(x,-)$. It is now a simple matter of diagram chasing to conclude. 
\end{proof}
\begin{ex}\label{ex:presat}
    Let $f:\M\to \N$ be an internal left adjoint between $\V$-modules. $f$ preserves atomic maps.
\end{ex}
\begin{proof}
    Let $R_\M$ (resp. $R_\N$) denote the right adjoint to the inclusion $i_\M: \Fun^L_\V(\M,\V)\to~\Fun_{\V}(\M,\V)$, and let $g$ denote the right adjoint to $f$. We let $f^*, g^*$ denote the functors of precomposition by $f,g$ respectively, on either $\Fun^L_\V$ or $\Fun_{lax-\V}$. The assumption that $g$ is $\V$-linear and colimit-preserving is used to guarantee that $g^*$ is well-defined on $\Fun^L_\V$. Note that we have $g^*\dashv f^*$.

    From the equivalence $i_\N\circ g^*\simeq g^*\circ i_\M$, we obtain, by passing to right adjoints, $f^*\circ~R_\N \simeq~R_\M\circ~f^*$. 

Let $x\in \M$. We apply the above to $\hom_\N(f(x),-)$ to get that $\at_\N(f(x),f(-))\to \hom_\N(f(x),f(-))$ is universal among $\V$-linear, colimit-preserving functors mapping to $\hom_\N(f(x),f(-))$. 

In particular, the composite $\at_\M(x,-)\to \hom_\M(x,-)\to \hom_\N(f(x),f(-))$ factors through $\at_\N(f(x),f(-))$. This concludes the proof. 
\end{proof}
Combining the two previous examples, we find:
\begin{cor}\label{cor:iLffat}
    Let $f:\M\to \N$ be a fully faithful internal left adjoint between $\V$-modules. There is a canonical equivalence $\at_\M(-,-)\simeq \at_\N(f(-),f(-))$. 
\end{cor}
\begin{proof}
    The universal property from the proof of \Cref{ex:presat} together with the equivalence $\hom_\M(-,-)\simeq\hom_\N(f(-),f(-))$ together give the desired equivalence.  
\end{proof}

\begin{lm}\label{lm:hominIndV}
 Let $\V\in \CAlg(\PrL_\kappa)$, and let $\M$ be a $\lambda$-compactly generated $\V$-module, $\lambda\geq \kappa$. Let $p:\PP_\V(\M^\lambda)\to \M$ be the canonical functor, and $y$ its right adjoint. 

 For $m\in\M^\lambda$, the canonical map $\at_\M(m,p(-))\to\hom_\M(m,p(-))$ factors through the canonical map $\hom_{\PP_\V(\M^\lambda)}(y(m), -)\to\hom_\M(m,p(-))$. 
 \end{lm}
 \begin{proof}
     $\at_\M(m,p(-))$ and $\hom_\M(y(m),-)$ are two $\V$-linear, colimit-preserving functors on $\PP_\V(\M^\lambda)$, so to construct a map $\at_\M(m,p(-))\to \hom_{\PP_\V(\M)}(y(m),-)$ it suffices to construct one on $\M^\lambda$. 

     But then we can simply pick the canonical map $\at_\M(m,-)\to \hom_\M(y(m),y(-))$, and then because its source is Kan extended from $\M^\lambda$, it is clear that it factors the canonical map in general. 
 \end{proof}
We now draw a table of analogies: 
\begin{center}
\begin{tabular}{|c | c | c|} 
 \hline
 Over $\Sp$ & Over $\V$ & In $\V$ \\ [0.5ex] 
 \hline\hline
 Compact object & $\V$-atomic object & Dualizable object \\ 
 \hline
 Compactly generated 
 & Atomically generated & ... \\
 stable category &  $\V$-module & \\
 \hline
 Compact map & $\V$-atomic map & Trace-class map \\
 \hline 
 Dualizable stable category & Dualizable $\V$-module & ...\\
 \hline
 \end{tabular}

\end{center}
For this analogy to be perfect, one would perhaps have to find an analogue of \cite[Theorem 2.39]{maindbl} or \cite[Theorem 2.55]{maindbl} in the general setting. There are difficulties associated with the ``Only if'' direction of either characterization - that is, while with a relatively generous definition of ``atomically exhaustible objects'' one can prove that being generated by them implies being dualizable, the converse does not seem to hold in general. The reader is invited to try and prove this slightly naive version of the ``If'' direction of either characterization, but we now embark on informally analyzing the ``Only if'' direction to correct the naive guess.

 In the stable/compactly assembled case, the idea was that, writing $\hat y(x) = \colim_I y(x_i)= \colim_I \hat y(x_i)$ forces the maps $y(x_i)\to \hat y(x)$ to factor through a map $y(x_i)\to \hat y(x_j)$ classifying a compact map $x_i\to x_j$, thus leading to the possibility of iteration, and leading quite naturally to the notion of compactly exhaustible objects. 

 However, over a general base $\V$, two problems arise: first, one cannot simply write objects of $\PP_\V(\M^\kappa)$ as colimits of objects in the image of $y$ - one also needs tensors by elements of $\V$, i.e. one essentially needs weighted colimits. 
 
 However, even taking this into account, the second problem is that $y(m)\in\PP_\V(\M^\kappa)$ being atomic does \emph{not} imply that any map $y(m)\to \colim_I^W y(x_i)$ factors through some ``finite'' stage, as in the case of $\V=\Sp$. This is among other things because the unit of $\V$ need not be compact, but also because the needed \emph{weighted} colimits behave more like classical coends, so that even if the unit is compact, one obtains at best a factorization through a finite weighted colimit $\colim_J^W y(x_j)$. 
 
 This suggests the following definition which in practice seems quite unusable, but turns out to be convenient for theoretical purposes:
 \begin{defn}\label{defn:atpres}
Let $\M$ be a $\V$-module.
    A map $\alpha:x\to y$ in $\M$ is ($\V$-)atomically presentable if $x$ can be written as a weighted colimit $\colim_I^W f$ for some small $\V$-category $I$, weight $W$ and diagram $f$ so that the map $\alpha: \colim_I^Wf\to y$ corresponds, under adjunction, to a map of weights $W\to \hom_\M(f(-),y)$ that factors through $\at_\M(f(-), y)$. 

    A map is weakly atomically presentable if it factors through an atomically presentable map. 

    An object is (weakly) atomically presentable if its identity map is. 

    In all cases, a (weakly) relatively atomic presentation is the data of such an expression (resp. factorization and expression) of $\alpha$. 
\end{defn}

The following is evident from the definitions:
\begin{lm}
    If $f:x\to y$ is (weakly)  atomically presentable, then so is $gf$ for any $g:y\to z$. 

    For the weak case, this is also true for $fh$. 
\end{lm}

We can start with a general analogue of \cite[Theorem 2.39]{maindbl}- thanks to the flexibility of the notion of atomic presentation we gave here, the sequential colimits are not really needed, but we include them for comparison:
\begin{thm}\label{thm:dblexhaustionV}
Let $\M$ be a $\V$-module. $\M$ is dualizable (0) if and only if either of the following equivalent conditions holds:
\begin{enumerate}
    \item Every object in $\M$ is atomically presentable;
    \item Every object in $\M$ is weakly atomically presentable; 
    \item Every object in $\M$ is a sequential colimit along atomically presentable maps; 
    \item Every object in $\M$ is a sequential colimit along weakly atomically presentable maps.
  
\end{enumerate}
Analogous statements where one replaces ``every object is of a certain form'' by ``objects of a certain form generate $\M$ as a $\V$-module'' are also equivalent. 
\end{thm}

To prove this, we will need an analogue of \cite[Lemma 2.41]{maindbl} that shows that our notion of atomically presented map is reasonable. 
\begin{lm}\label{lm:swapat}
    Let $\alpha: x=\colim_I^W f\to y$ be an atomic presentation of a map in a $\V$-module $\M$, with a chosen lift $\tilde\alpha: W\to \at_\M(f(-),y)$ of the corresponding map of weights $\alpha: W\to \hom_\M(f(-),y)$. Let $j:\M\to \PP_\V(\M^\kappa)$ be the restricted Yoneda embedding with left adjoint $p$, where $\kappa$ is an $\M$-good cardinal such that $f$ lands in $\M^\kappa$. Finally let $X:= \colim_I^W j\circ f$. Note that we have a map $X\to j(x)$ given by \Cref{cons:compweight}. 
    
    There is a natural dotted lift in the following diagram of $\V$-functors $\PP_\V(\M^\kappa)\to \V$:
    
\[\begin{tikzcd}
	{\hom(j(y),-)} & {\hom(y,p(-))} \\
	{\hom(X,-)} & {\hom(x,p(-))}
	\arrow[from=1-1, to=1-2]
	\arrow[from=1-1, to=2-1]
	\arrow[dashed, from=1-2, to=2-1]
	\arrow[from=1-2, to=2-2]
	\arrow[from=2-1, to=2-2]
\end{tikzcd}\]
    \end{lm}
\begin{proof}
    We construct maps and leave the task of verifying that they have the appropriate properties to the reader. 

    We have a map $$\hom(W,\at_\M(f(-),y))\otimes \hom(y,p(-))\to \hom(W,\at_\M(f(-),y)\otimes \hom(y,p(-)))\to \hom(W,\at_\M(f(-),p(-)))$$ where the second map uses \Cref{ex:atideal}. Second, we use \Cref{lm:hominIndV} to get\footnote{This is where we need $f$ to land in $\M^\kappa$.}, in total, a map $$\hom(W,\at_\M(f(-),y))\otimes \hom(y,p(-))\to \hom(W, \hom(j(f(-)),-)) \simeq \hom(\colim_I^W j(f(-)),-) = \hom(X,-)$$ 

    Using our specific lift $\one_\V\to\hom(W,\at_\M(f(-),y))$ of $\alpha$, we get the dotted arrow as desired.
\end{proof}
\begin{cor}\label{cor:swapat}
    In the notation of the previous lemma, assume $y\simeq p(Y)$ for some $Y\in~\PP_\V(\M^\kappa)$. There is then a lift $A:X\to Y$ of $\alpha:x\to y$ with a natural dotted lift as follows: 
    \[\begin{tikzcd}
	{\hom(Y,-)} & {\hom(y,p(-))} \\
	{\hom(X,-)} & {\hom(x,p(-))}
	\arrow[from=1-1, to=1-2]
	\arrow[from=1-1, to=2-1]
	\arrow[dashed, from=1-2, to=2-1]
	\arrow[from=1-2, to=2-2]
	\arrow[from=2-1, to=2-2]
\end{tikzcd}\]
\end{cor}
\begin{proof}
    Apply the lemma to get a lift $A: X\to Y$ in the first place. The bottom half of the triangle has not changed, and for the top half, consider the following commutative diagrams: 
\[\begin{tikzcd}
	{\hom_{\PP_\V(\M^\kappa)}(Y,-)} && {\hom(W,\hom_{\PP_\V(\M^\kappa)}(Y,-)\otimes \at_\M(f,y))} & {} \\
	{\hom_\M(y,p(-))} & {\hom_\M(y,p(-))\otimes \hom(W,\at_\M(f,y))} & {\hom(W,\hom_\M(y,p(-))\otimes\at_\M(f,y))}
	\arrow[from=1-1, to=1-3]
	\arrow[from=1-1, to=2-1]
	\arrow[from=1-3, to=2-3]
	\arrow[from=2-1, to=2-2]
	\arrow[from=2-2, to=2-3]
\end{tikzcd}\]

and 

\[\begin{tikzcd}
	{\hom(W,\hom_{\PP_\V(\M^\kappa)}(Y,-)\otimes \at_\M(f,y))} & {\hom(W,\hom_{\PP_\V(\M^\kappa)}(Y,-)\otimes\hom(j\circ f,Y))} \\
	{\hom(W,\hom_\M(y,p(-))\otimes\at_\M(f,y))} & {\hom(W,\hom_{\PP_\V(\M^\kappa)}(j\circ f,-))}
	\arrow[from=1-1, to=1-2]
	\arrow[from=1-1, to=2-1]
	\arrow[from=1-2, to=2-2]
	\arrow[from=2-1, to=2-2]
\end{tikzcd}\]

And now observe that putting them side by side provides exactly the desired commutation. 
\end{proof}
\begin{cor}\label{cor:atpreslocalleft}
    Let $x$ be atomically presented in a $\V$-module $\M$, and let $x=\colim_I^Wf$ be an atomic presentation of $x$. Let $\kappa$ be an $\M$-good cardinal such that $f$ lands in $\M^\kappa$,  let $j:\M\to\PP_\V(\M^\kappa)$ the corresponding Yoneda embedding, and finally $X:=\colim_I^W j\circ f$. Then for all $F\in \PP_\V(\M^\kappa)$, $p$ induces an equivalence $$\hom_{\PP_\V(\M^\kappa)}(X,F)\to \hom_\M(x,p(F))$$ 
\end{cor}
In this corollary, note that the map $\hom(X,F)\to \hom(x,p(F))$ does not depend on a witness $W\to \at_\M(f(-),x)$ that $\colim_I^Wf$ is an atomic presentation of $X$, it only depends on the map $W\to \hom(f(-),x)$ classifying the equivalence $\colim_I^Wf\to x$. Thus the corollary only depends on the \emph{property} that there exists such a witness. This will be crucial later on. 
\begin{proof}[Proof of \Cref{thm:dblexhaustionV}]
    1 implies 2, 1 implies 3, 3 implies 4 and 2 implies 4 are obvious, so it suffices to prove that 0 implies 1 and 1 implies 4. 

    (0 implies 1): Recall that essentially by definition, the canonical functor $p:\PP_\V(\M^\kappa)\to~\M$ is given as $\colim_{\M^\kappa}^{(-)} i$, where $i:\M^\kappa\to \M$ is the inclusion. 

    Now the map $\hat y\to y$ induces equivalences $p\circ \hat y\to p\circ y\to \id_\M$, and by \Cref{cor:descofat} it corresponds exactly to the map $\at_\M(-,m)\to \hom_\M(-,m)$, so the resulting expression $$\colim_\M^{\at_\M(-,m)}i \simeq \colim_\M^{\hat y(m)}i \simeq m$$ is an atomic presentation of any $m\in\M$. 

    (4 implies 0): Let $x= \colim_\NN x_n$ be a sequential colimit along weakly atomically presentable maps. Up to inserting the relevant factorizations and by cofinality, one may assume that each $x_n\to x_{n+1}$ is in fact atomically presentable. Pick presentations for each of these: $x_n = \colim_{I_n}^{W_n}x^n_\bullet$ and lifts $W_n\to \at_\M(x^n_\bullet, x_{n+1})$. 

   Let $X_n:=~\colim_{I_n}^{W_n}y(x^n_\bullet)$. By \Cref{lm:hominIndV}, we get maps $$W_n\to \at_\M(x^n_\bullet, x_{n+1})\to \hom_{\PP_\V(\M^\kappa)}(y(x_n^\bullet), X_{n+1})$$ and thus, in total, we get maps $X_n\to X_{n+1}$ lifting the maps $x_n\to x_{n+1}$. Let $X:=~\colim_\NN X_n$. We have a canonical equivalence $p(X)\simeq x$. The claim is now that $X$ is a local left adjoint of $p$ at $x$, that is, the natural map $\hom_{\PP_\V(\M^\kappa)}(X,F)\to \hom_\M(p(X),p(F))\simeq \hom_\M(x,p(F))$ is an equivalence for all $F\in\PP_\V(\M^\kappa)$. 

   Since atomic presentations are closed under tensoring with $\V$, it will follow that on sequential colimits along weakly atomically presentable maps $x$, not only does a local left adjoint $p^L(x)$ exist, but also the canonical map $v\otimes x\to p(v\otimes p^L(x))$ witnesses $v\otimes p^L(x)$ as a local left adjoint at $v\otimes x$. Since these objects generate under colimits by assumption, it will follow not only that $p$ admits a left adjoint, but also that it is strongly $\V$-linear (as opposed to only oplax $\V$-linear).
   
  To prove this claim, we now use the usual trick: using \Cref{cor:swapat},  we find dotted lifts in the diagrams: 
   \[\begin{tikzcd}
	{\hom_{\PP_V(\M^\kappa)}(X_{n+1},F)} & {\hom_\M(x_{n+1},p(F))} \\
	{\hom_{\PP_V(\M^\kappa)}(X_n,F)} & {\hom_\M(x_{n},p(F))}
	\arrow[from=1-1, to=1-2]
	\arrow[from=1-1, to=2-1]
	\arrow[dashed, from=1-2, to=2-1]
	\arrow[from=1-2, to=2-2]
	\arrow[from=2-1, to=2-2]
\end{tikzcd}\]
which together prove the claim by a simple cofinality argument as in the proof of \cite[Lemma 2.24]{maindbl}. 
\end{proof}
As a corollary, we also obtain a way of detecting internal left adjoints with dualizable source:
\begin{cor}\label{cor:iLviaexhaust}
Let $f:\M\to\N$ be a map in $\Mod_\V(\PrL)$, and assume $\M$ is dualizable. 

If there exists a factorization as follows:
\[\begin{tikzcd}
	{\at_\M(-,-)} & {\at_\N(f(-),f(-))} \\
	{\hom_\M(-,-)} & {\hom_\N(f(-),f(-))}
	\arrow[dashed, from=1-1, to=1-2]
	\arrow[from=1-1, to=2-1]
	\arrow[from=1-2, to=2-2]
	\arrow[from=2-1, to=2-2]
\end{tikzcd}\]
then $f$ is an internal left adjoint. 
\end{cor}
As in \cite[Lemma 2.41]{maindbl}, we must first use atomic presentations to detect colimits, namely: 
\begin{cor}\label{cor:colimdetectV}
   Let $x\in \M$ be atomically presentable, with $\colim_I^Wf\simeq x$ an atomic presentation, and let $X:=\colim_I^W j\circ f$. Let $V:J\op\to \V$ be a weight and $g:J\to \M$ be a diagram, with a map $W\to \hom(g(-),z)$. The functor $\hom_\M(x,-)$ applied to the induced map $\colim_J^V g\to z$  is an equivalence if and only if the induced map $$\hom(X,\colim_J^V j\circ g)\to~\hom(x,z)$$ is an equivalence. 
\end{cor}
\begin{proof}
    By \Cref{cor:atpreslocalleft}, the corresponding map $\hom(X,\colim_J^Vj\circ g)\to \hom(x,\colim_J^V g)$ is an equivalence, so the claim is immediate by 2-out-of-3. 
\end{proof}
Crucially, while it is important that $\colim_I^Wf\simeq x$ be an atomic presentation, neither the object $X$ nor the map $\hom(X,\colim_J^V j\circ g)\to \hom(x,z)$ depend on the choice of lift $W\to \at(f(-),x)$. 
\begin{proof}[Proof of \Cref{cor:iLviaexhaust}]
If there exists such a factorization, then we find that an atomic presentation of any $x\in \M$ is sent to an atomic presentation of $f(x)\in \N$. 

Let $V:J\op\to \V$ be a weight and $g:J\to \N$ a diagram. We wish to prove (cf. \Cref{cor:Vliniffweight}) that the canonical map $\colim^V_J f^R\circ g\to f^R(\colim^V_J g)$ is an equivalence. It suffices to check this after applying $\hom(x,-)$ for every $x\in \M$.  By \Cref{thm:dblexhaustionV}, these are all atomically presentable. So let $x=\colim_I^Wf$ be an atomic presentation of some $x\in\M$ and let $X=\colim_I^W j\circ f$. 

By \Cref{cor:colimdetectV}  it suffices to prove that the canonical map $\hom(X,\colim_J^V j\circ f^R\circ~g)\to~\hom(x, f^R(\colim_J^V g))$ is an equivalence. 

Now for $\kappa$ large enough $j\circ f^R\simeq \PP_\V((f^R)^\kappa)\circ j$ and so this canonical map is also $\hom(\PP_\V(f)(X),\colim_J^V j\circ g)\to \hom(f(x),\colim_J^Vg)$ and is thus an equivalence by the first paragraph of the proof and the previous corollary.
\end{proof}
The very same proof shows the following strenghtening of this result, which is unfortunately more technical to state, but is essentially a version of \cite[Addendum 2.42]{maindbl}: instead of asking for \emph{all} atomic maps to be preserved, it suffices to ask for \emph{enough} of them to be.
\begin{add}\label{add:enough}
    Let $L:\M\to \N$ be a $\V$-linear functor between $\V$-modules and assume $\M$ is dualizable. Suppose that the $x_s, s\in S$ are generators of $\M$ under colimits and $\V$-tensors, and for each $s\in S$, let $\colim_{I_s}^{W_s} f_s\to x_s$ be an atomic presentation of $x_s$, with corresponding maps of weights $W_s\to \hom(f_s, x_s)$. 

    If, for each $s$, we have factorizations of the form: 
    \[\begin{tikzcd}
	{W_s} & {\at_\N(L\circ f_s, L(x_s))} \\
	{\hom_\M(f_s,x_s)} & {\hom_\N(L\circ f_s,L(x_s))}
	\arrow[dashed, from=1-1, to=1-2]
	\arrow[from=1-1, to=2-1]
	\arrow[from=1-2, to=2-2]
	\arrow[from=2-1, to=2-2]
\end{tikzcd}\]
then $L$ is an internal left adjoint. 
\end{add}
Note that the criterion does not impose anything on the witness $W_s\to \at_\M(f_s,x_s)$ of atomicity of these presentations, all we need is for them to exist. 

As a special case, we find the following convenient criterion in the case of tensor products:
\begin{cor}
    Let $\M_1,...,\M_k$ be dualizable $\V$-modules and $f: \M_1\otimes_\V\dots\otimes_\V\M_k\to \N$ be a $\V$-linear map. It is an internal left adjoint if there exists a factorization as follows: 
    \[\begin{tikzcd}
	{\otimes_i\at_{\M_i}(-,-)} & {\at_\N(f(-),f(-))} \\
	{\hom_{\otimes_i\M_i}(-,-)\to } & {\hom_\N(f(-),f(-))}
	\arrow[dashed, from=1-1, to=1-2]
	\arrow[from=1-1, to=2-1]
	\arrow[from=1-2, to=2-2]
	\arrow[from=2-1, to=2-2]
\end{tikzcd}\]
\end{cor}

Before moving on to our study of rigidity, we rephrase the above ideas in a slightly different language, and use them (though the above ones would work just as well) to describe the ``dualizable core'' of any $\V$-module. The presentation here is heavily inspired by the work of Bunke and Dünzinger \cite{etheory}. I believe their presentation is cleaner: I have only included the previous discussion because I do not know how to give such a presentation for the case of rigidification to come, so that this previous ``hands-on'' discussion serves as a blueprint for that case. 

I also do not know of a practical use of the following construction of the dualizable core over a general base $\V$ (beyond proving that it exists), but the relevant ideas should also be seen as a simpler blueprint of our study of rigidifications later on. 

\begin{defn}
Let $\M$ be a $\V$-module and $\kappa$ an $\M$-good cardinal. A ($\V$-)\emph{shape-able} object of $\M$ is an object $x$ such that the functor $\hom(x, p(-)) : \PP_\V(\M^\kappa)\to \V$ admits a $\V$-linear left adjoint, i.e. there exists $\hat j(x)$ such that $\hom(\hat j(x), -)\simeq \hom(x,p(-))$. 
    
We call $\hat j(x)$ the \emph{shape} of $x$. 
\end{defn}
\begin{obs}
    Let $x$ be shape-able in $\M$, and let $v\in \V$. We have a map $$v\otimes x \to v\otimes p(\hat j(x))\simeq p(v\otimes \hat j(x)).$$ 

   Note that for all $v\in \V$, this map witnesses $v\otimes \hat j(x)$ as the shape of $v\otimes x$, i.e. it induces an equivalence $\hom(v\otimes \hat j(x), F)\simeq \hom(v\otimes x, p(F))$.
\end{obs}
\begin{rmk}
Alternatively, one could ask for the functor $\Map(x,p(-)): \PP_\V(\M^\kappa)\to \Ss$ to be representable. In this case, it would not imply that $v\otimes x$ has the same property, though one could explicitly require the map $v\otimes x\to p(v\otimes \hat j(x))$ constructed above to witness $v\otimes\hat j(x)$ as a representing object, and then we would reach an equivalent definition. 
\end{rmk}

If $x$ is atomically presentable, \Cref{cor:atpreslocalleft} shows that it is shape-able: if $x\simeq \colim_I^W f$ is an atomic presentation, we see that $\hat j(x) = \colim_I^W j\circ f$ is the shape of $x$. 

Conversely, it is not entirely clear to the author whether shapes are always atomic presentations, namely, whether $\hom_{\PP_\V(\M^\kappa)}(y(-),\hat y(x))\to \hom_\M(-,x)$ lifts to $\at_\M(-,x)$: by \Cref{cor:descofat}, this is true when $\M$ is dualizable, but if not, the fact that $\hat y(-)$ is not globally defined makes it difficult to answer this question. 

But morally, one should think of the two notions as capturing the same idea, though they might not be equivalent away from the dualizable setting.

We now have a nice feature of shapes:
\begin{obs}
    If $x$ is shape-able with shape $\hat j(x)$, there is a canonical equivalence $p\circ \hat j(x)\simeq x$. 
\end{obs}
It is actually a general phenomenon:
\begin{lm}
    Let $p:P\to M$ be a functor with a fully faithful right adjoint $j$. Suppose $p$ admits a local left adjoint at $x\in M$, that is, the functor $\Map(x,p(-)): P\to \Ss$ is corepresentable by some object $\hat j(x)\in P$. 

    The canonical map $x\to p\hat j(x)$ is an equivalence. 
\end{lm}
\begin{proof}
    The identity of $\hat j(x)$ classifies a map $x\to p\hat j(x)$, which we aim to prove is an equivalence. So fix $z\in M$, and consider the diagram: 
    \[\begin{tikzcd}
	{\Map(\hat j(x),j(z))} & {\Map(p\hat j(x), pj(z))} & {\Map(p\hat j(x),z)} \\
	& {\Map(x,pj(z))} & {\Map(x,z)}
	\arrow[from=1-1, to=1-2]
	\arrow[from=1-1, to=2-2]
	\arrow[from=1-2, to=1-3]
	\arrow[from=1-2, to=2-2]
	\arrow[from=1-3, to=2-3]
	\arrow[from=2-2, to=2-3]
\end{tikzcd}\]
In this diagram, the leftmost horizontal map is obtained by applying $p$, while the other maps are given by pre- or postcomposition with either the unit $x\to p\hat j(x)$ or the counit $pj(z)\to z$, the latter being an equivalence. 

Thus, the following maps are equivalences: the upper horizontal composite, by the $p\dashv y$ adjunction; the bottom horizontal map, because the counit $pj(z)\to z$ is an equivalence; the diagonal arrow, because of the universal property of $\hat j(x)$. 

It follows from the top composite and the bottom composite being equivalences that also the rightmost vertical map must be an equivalence. As $z$ was arbitrary, it follows that $x\to p\hat j(x)$ is an equivalence too, as claimed. 
\end{proof}

\begin{cons}
    Let $\kappa$ be a good cardinal and $\M$ be a $\kappa$-compactly generated $\V$-module. We let $\M_{0}\subset \M$ denote the subcategory of $\kappa$-compact shape-able objects, and $\M_{sh} \subset \M$ the full subcategory it generates under colimits and tensors with $\V$. 
\end{cons}
\begin{obs}\label{obsA}
    Every object of $\M_{sh}$ is shape-able in $\M$. Furthermore, $\hat j: \M_{sh}\to~\PP_\V(\M^\kappa)$ preserves $\kappa$-compacts. 
\end{obs}
\begin{proof}
    The class of shape-able objects is closed under colimits and $\V$-tensors. 

    For the second claim, note that it suffices to prove that $\hat j$ of a $\kappa$-compact shape-able object is $\kappa$-compact. This is clear since $p$ commutes with ($\kappa$-filtered) colimits. 
\end{proof}
\begin{rmk}
    For good $\kappa$, one can argue without too much difficulty that any atomically presentable object is a colimit of $\kappa$-compact atomically presentable objects. It is not clear to the author that the same holds for shape-able objects, but we \emph{do} need some set-theoretic control and this is why we define $\M_{sh}$ as we do and not simply as the class of shape-able objects (which would be the most natural thing to consider). 
\end{rmk}

\begin{defn}
    Let $\M$ be a $\kappa$-compactly generated $\V$-module. We define $\M^{(\alpha)}$ by induction on the ordinal $\alpha$: $\M^{(1)}= \M_{sh}$, $\M^{(\alpha+1)}$ is defined as $\PP_\V((\M^{(\alpha)})^\kappa)\times_{\PP_\V(\M^\kappa)}\M^{(1)}$ that is, the $x$'s in $\M^{(1)}$ such that $\hat j(x)$ is in $\PP_\V(\M^{(\alpha)})$, and at limit stages, $\M^{(\alpha)}$ is generated under colimits and $\V$-tensors by $\bigcap_{\beta < \alpha}(\M^{(\beta)})^\kappa$.  
\end{defn}
\begin{rmk}
    By \cite[Proposition 2.31]{maindbl}, \Cref{obsA} and by construction, $\M^{(\alpha)}$ is $\kappa$-compactly generated for all $\alpha$ and furthermore $\M^{(\alpha)}\to~\M$ preserves $\kappa$-compacts.
\end{rmk}
\begin{cor}\label{cor:stabilize}
    There is $\alpha$ such that $\M^{(\alpha+1)}=\M^{(\alpha)}$.
\end{cor}
\begin{proof}
    By the previous remark, $(\M^{(\alpha)})^\kappa$ defines an ordinal indexed sequence of subcategories of $\M^\kappa$, and thus stabilizes because this category is small. Since each $\M^{(\alpha)}$ is $\kappa$-compactly generated, the stabilization at the level of $\kappa$-compacts implies stabilization. 
\end{proof}
\begin{defn}\label{defn:dblcore}
    We let $\M_\dbl$ be $\M^{(\alpha)}$ for some $\alpha$ such that $\M^{(\alpha+1)}=\M^{(\alpha)}$, and we call it the dualizable core of $\M$. 
\end{defn}
\begin{lm}
    $\M_\dbl$ is dualizable.
\end{lm}
\begin{proof}
    Since $\M^{(\alpha)} =\M^{(\alpha+1)}$, we find that $\hat j: \M^{(\alpha)}\to \PP_\V((\M^{(\alpha)})^\kappa)$ is a left adjoint of $p$. Thus by \cite[Theorem 1.49]{maindbl}, $\M^{(\alpha)}$ is dualizable. 
\end{proof}
\begin{prop}\label{prop:dblcoreincisiL}
    The inclusion $\M_\dbl\to \M$ is an internal left adjoint. 
\end{prop}
More generally, we prove: 
\begin{lm}
    Let $f:\N\to \M$ be a map in $\Mod_\V$, and suppose $n\in\N$ is shape-able, and that $\PP_\V(f^\kappa)(\hat j(n))$ is a shape of $f(n)$. In this case, 
    \begin{itemize}
        \item For every diagram $m: I\to \M$, the canonical map $\colim_I f^R(m)\to f^R(\colim_I m)$ induces an equivalence after applying $\hom(n,-)$; 
        \item For every $v\in\V$, the canonical map $v\otimes f^R(m)\to f^R(v\otimes m)$ induces an equivalence after applying $\hom(n,-)$. 
    \end{itemize}

    In particular, if $\N$ is dualizable and $f$ preserves shapes, then $f$ is an internal left adjoint. 
\end{lm}
This is essentially the same proof as in \Cref{cor:colimdetectV} and \Cref{cor:iLviaexhaust}, so we leave it as an instructive exercise to the reader.

Conversely,
\begin{lm}
    Let $f:\N\to \M$ be an internal left adjoint in $\Mod_\V(\PrL_\kappa)$, and let $n\in \N$ be a shape-aple object. In this case, the functor $f$ preserves the shape of $n$, that is, $\PP_\V(f^\kappa)(\hat j(n))$ is a shape of $f(n)$. 
\end{lm}
\begin{proof}
    Let $F\in\PP_\V(\M^\kappa)$. We wish to prove that $\hom(\PP_\V(f^\kappa)(\hat j(n)), F)\to \hom(f(n),p(F))$ is an equivalence. For this, we simply note that this map is compatible with the respective equivalences $\hom(\PP_\V(f^\kappa)(\hat j(n)), F)\simeq \hom(\hat j(n), \PP_\V((f^R)^\kappa)) F)$ and $\hom(f(n),p(F))\simeq \hom(n,f^Rp(F))$. 

    The claim follows from the fact that $f^R p \simeq p\circ \PP_\V((f^R)^\kappa)$ (because $f^R$ is $\V$-linear). 
\end{proof}
\begin{cor}
    The inclusion $\M_\dbl\to \M$ is universal among internal left adjoints from a dualizable $\V$-module. 
\end{cor}
\begin{proof}
From the definition, it is clear that if $\N$ is dualizable, then $\N_\dbl=\N$. 

Furthermore, if $f:\N\to \M$ is an internal left adjoint, the previous lemma shows that it preserves shapes, and so it sends $\N_\dbl$ to $\M_\dbl$. Since $\M_\dbl\to \M$ also preserves shapes by design, it follows that $\N=\N_\dbl\to \M_\dbl$ does too, and hence is an internal left adjoint. 

Conversely, \Cref{prop:dblcoreincisiL} shows that if $\N\to\M_\dbl$ is an internal left adjoint, so is $\N\to \M$. In total, this proves that $\Fun^{iL}_\V(\N,\M_\dbl)\to\Fun^{iL}_\V(\N,\M)$ is an equivalence. 
\end{proof}
\section{Trace-class maps}\label{section:traceclass}
In this section, we discuss trace-class maps, which are to compact maps (resp. atomic maps) what dualizable objects are to compact objects (resp. atomic objects). 
\begin{defn}\label{defn:trcl}
    Let $\C$ be a symmetric monoidal category. A map $f:x\to y$ is called \emph{trace-class} if there exists an object $d$ together with a pairing $d\otimes x\to \one_\C$ and a map $\one_\C\to y\otimes d$ such that $f$ factors as $x\to y\otimes d\otimes x\to y$. 
\end{defn}
\begin{rmk}\label{rmk:trcl}
    If $\C$ is monoidal closed, then this condition is equivalent to the requirement that the map $\one_\C\to \hom_\C(x,y)$ classifying $f$ lifts through the canonical map $$\hom_\C(x,\one_\C)\otimes y\to \hom_\C(x,y)$$ This is how this notion is often defined; we defined it the way we did to make two facts apparent: firstly, they can be defined even when $\C$ is not closed, and secondly, arbitrary symmetric monoidal functors preserve trace-class maps even if they are not closed\footnote{Of course on can prove this with the other definition too, but this way it is really clear.}.

    Finally, it makes the following property clear: suppose $\C\in\CAlg(\PrL_\kappa)$, then a map $f:x\to y$ between objects in $\C^\kappa$ is trace-class in $\C$ if and only if it is so in $\C^\kappa$. 
\end{rmk}
The above remark suggests that $\hom_\C(x,\one_\C)\otimes y$ is the ``object of trace-class morphisms'' from $x$ to $y$. In the following, we  make this relative:
\begin{defn}
    Let $\W$ be a commutative $\V$-algebra, and let $x,y\in\W$. We let $$\trcl_\W^\V(x,y):=~\hom_\W^\V(\one_\W,\hom_\W^\W(x,\one_\W)\otimes y)$$ This is the $\V$-enriched hom from $\one_\W$ to $\hom(x,\one_\W)\otimes y$, where the latter is the $\W$-internal hom. 
\end{defn}

Recall from \Cref{ex:VatV} that for $\V\in\CAlg(\PrL)$, trace-class maps are the same thing as $\V$-atomic maps in $\V$. One of the key features of trace-class maps is the following:
\begin{prop}\label{prop:trclexch}
    Let $f:\M\to \N$ be a lax-$\V$-linear functor between $\V$-modules, and $f: x\to y$ a trace-class map. In this case, for any $m\in\M$, there is a diagonal filler in the following commutative diagram:
    \[\begin{tikzcd}
	{x\otimes f(m)} & {f(x\otimes m)} \\
	{y\otimes f(m)} & {f(y\otimes m)}
	\arrow[from=1-1, to=2-1]
	\arrow[from=1-2, to=2-2]
	\arrow[from=2-1, to=2-2]
	\arrow[from=1-1, to=1-2]
	\arrow[dashed, from=1-2, to=2-1]
\end{tikzcd}\]
\end{prop}
\begin{rmk}
This is formally analogous to the diagonal filler appearing in the definition of compact maps, however if we think of tensoring as analogous to colimits, here the assumption is on ``a type of colimit diagram'', while in the definition of compact maps, the assumption is on ``$m$''. It is not clear to the author whether there is a common framework for these two statements; however they have the same kind of underlying ``meaning'' (one can commute things that ought not to commute), and the same kind of use. 
\end{rmk}
\begin{rmk}
    An important special case of this is the case where $f:x\to y$ is the identity of a dualizable object $x$. In this case, this shows that the canonical map $x\otimes f(m)\to f(x\otimes m)$ is an equivalence. 
\end{rmk}
\begin{proof}
Fix $d$, $\epsilon : d\otimes x\to \one_\V$ and $\eta: \one_\V\to y\otimes d$ as in the definition of trace-class maps. The diagonal filler $f(x\otimes m)\to y\otimes f(m)$ is defined as the composite: $$f(x\otimes m)\xrightarrow{\eta\otimes \id}y\otimes d\otimes f(x\otimes m) \to y\otimes f(d\otimes x\otimes m)\xrightarrow{y\otimes f(\epsilon\otimes \id)} y\otimes f(m)$$

The following diagram shows that it is actually a diagonal filler:
\[\begin{tikzcd}
	{x\otimes f(m)} && {f(x\otimes m)} \\
	{y\otimes d\otimes x\otimes f(m)} & {y\otimes d\otimes f(x\otimes m)} & {f(y\otimes d\otimes x\otimes m)} \\
	& {y\otimes f(d\otimes x\otimes m)} \\
	{y\otimes f(m)} && {f(y\otimes m)}
	\arrow[from=1-1, to=1-3]
	\arrow[from=1-3, to=2-3]
	\arrow[from=2-3, to=4-3]
	\arrow[from=1-3, to=2-2]
	\arrow[from=2-2, to=2-3]
	\arrow[from=1-1, to=2-1]
	\arrow[from=2-2, to=3-2]
	\arrow[from=2-1, to=2-2]
	\arrow[from=4-1, to=4-3]
	\arrow[from=3-2, to=4-1]
	\arrow[from=2-1, to=4-1]
	\arrow[from=3-2, to=2-3]
\end{tikzcd}\]
\end{proof}
\begin{cor}\label{cor:nuclax}
Let $\V\in\CAlg(\PrL), \M,\N\in\Mod_\V(\PrL)$
    Let $f:\M\to \N$ be a lax $\V$-linear map, e.g. the right adjoint of a morphism in $\Mod_\V(\PrL)$. Suppose $f$ preserves sequential colimits. 

    Let $x_\bullet:\mathbb N\to \V$ be a sequential diagram where each transition map is trace class, and let $x_\infty$ denote its colimit. In this case, for any $m\in\M$, the natural map $x_\infty\otimes f(m)\to f(x_\infty\otimes m)$ is an equivalence. 
\end{cor}
We spell out this corollary in the case of a dualizable object, but it is a strict special case of the above: 
\begin{cor}\label{cor:projdbl}
Let $\V\in\CAlg(\PrL), \M,\N\in\Mod_\V(\PrL)$
    Let $f:\M\to \N$ be a lax $\V$-linear map, e.g. the right adjoint of a morphism in $\Mod_\V(\PrL)$.

    Let $x\in\V$ be dualizable. In this case, for any $m\in\M$, the natural map $x\otimes~f(m)\to~f(x\otimes~m)$ is an equivalence. 
\end{cor}

Finally, as an analogue of \cite[Example 1.23]{maindbl}, we prove:
\begin{lm}\label{lm:trclat}    
Let $f:x\to y$ be a trace-class map in $\V$, and $g: m\to n$ a $\V$-atomic map in a $\V$-module $\M$. The tensor product $f\otimes g :x\otimes m\to y\otimes n$ is $\V$-atomic. 
\end{lm}
\begin{proof}
    We produce a factorization $$\hom(x,\one_\V)\otimes y\otimes \at_\M(m,n)\to\at_\M(x\otimes m,y\otimes n)\to \hom_\M(x\otimes m, y\otimes n)$$ of the canonical map $$\hom(x,\one_\V)\otimes y\otimes \at_\M(m,n)\to \hom(x,y)\otimes \hom(m,n)\to \hom(x\otimes m, n\otimes y)$$ 

    We do this in two steps. First, using (lax) $\V$-linearity of $\at_\M(m,-)$, we reduce to the case $y=\one_\V$. In this case, we note that $$\hom(x,\one_\V)\otimes \at_\M(m,-) \to \hom(x,\one_\V)\otimes \hom_\M(m,-)\to \hom_\M(x\otimes m,-)$$ is a natural transformation from a colimit-preserving $\V$-linear functor, and hence it factors canonically through $\at_\M(x\otimes m,-)$, as was needed. 
\end{proof}
\section{Local rigidity and rigidification}\label{section:rigandlocrig}
 The goal of this section is to study how the notion of dualizable $\V$-module varies, as $\V$ itself varies. We introduce the notions of rigid and more generally locally rigid $\V$-algebras, and prove basic properties about them.  Finally, our last goal is to describe rigidification and to use it to say a bit more about locally rigid categories. 

\begin{obs}
    For a morphism $f:\V\to \W$ in $\CAlg(\PrL)$, the induced basechange functor $\W\otimes_\V - : \Mod_\V\to \Mod_\W$ can be upgraded to a symmetric monoidal $2$-functor, hence it preserves dualizability and internal left adjoints. In particular it induces a functor $\Dbl{\V}\to \Dbl{\W}$. This functor also has a right adjoint, the restriction of scalars functor $\Mod_\W\to \Mod_\V$ which is a lax symmetric monoidal $2$-functor. Note that its $2$-functoriality guarantees that it preserves internal left adjoints. 
\end{obs}
Here are a number of questions one can ask about a given $f$ (all of them have negative answers in full generality, but one can ask \emph{when} they have positive answers): 
\begin{ques}
    When does restriction of scalars preserve dualizability ? When does it reflect dualizability ? 

    When does it \emph{reflect} internal left adjoints, i.e. if a map of $\W$-modules $\M\to \N$ is a $\V$-internal left adjoint, when can we conclude that it is a $\W$-internal left adjoint ?
\end{ques}
One can also ask about atomic maps:
\begin{ques}
    Let $f:x\to y$ be a map in a $\W$-module $\M$. When does $f$ being $\V$-atomic imply that it is $\W$-atomic ? What about the converse ?

    How precisely can we compare $\at_\M^\V(x,y)$ and $\at_\M^\W(x,y)$ ? 
\end{ques}
The first questions have general answers, namely: 
\begin{prop}
The forgetful functor $f^*:\Mod_\W\to \Mod_\V$ induced by $f:\V\to \W$ preserves dualizability if and only if $\W$ is dualizable as a $\V$-module. 

It \emph{reflects} dualizability if $\W$ is \emph{smooth} as a $\V$-module, i.e. $\W$ is dualizable as a $\W\otimes_\V\W$-module. 
\end{prop}
\begin{proof}
This is \cite[Propositions 4.6.4.4., 4.6.4.12.]{HA}. 
\end{proof}
It turns out that these two conditions (being smooth and proper) are implied by the stronger condition of being \emph{rigid}. This notion will also be of help to study our other questions; as well as its variant, that of a \emph{locally rigid} $\V$-algebra. We point out that these notions have been introduced by Gaitsory and Rozenblyum and further studied in \cite{gaitsroz,HSS} and \cite{kazhdan}.

For conceptual clarity, we first study locally rigid algebras, and later rigid algebras, but let us note that most examples of locally rigid algebras are localizations of rigid algebras, and the latter are more ``standard''. 
\subsection{Local rigidity}
\begin{defn}
    Let $f:\V\to\W$ be a morphism in $\CAlg(\PrL)$. It is said to be locally rigid, or $\W$ is said to be locally rigid over $\V$ if the multiplication map $\W\otimes_\V\W\to \W$ is an internal left adjoint in $\Mod_{\W\otimes_\V\W}(\PrL)$ and $\W$ is dualizable over $\V$. 
\end{defn}
We begin by studying in detail the atomically generated situation.
\begin{ex}\label{ex:locrigat}
    Let $f:\V\to\W$ be a morphism in $\CAlg(\PrL)$. Assume $\W$ is $\V$-atomically generated. In this case, $\W$ is locally rigid if and only if every $\V$-atomic object in $\W$ is dualizable. 
\end{ex}
\begin{proof}
    Since $\W$ is $\V$-atomically generated, it is dualizable over $\V$ by \cite[Proposition 1.40]{maindbl}. 

    Assume now that every $\V$-atomic object in $\W$ is dualizable.  In this case, the multiplication map sends pairs of $\V$-atomics to $\V$-atomics (because a $\V$-atomic tensored with a $\W$-dualizable is $\V$-atomic, by an argument similar to \cite[Example 1.23]{maindbl}), and thus is an internal left adjoint in $\Mod_\V(\PrL)$. We are left with proving that its right adjoint $m^R : \W\to\W\otimes_\V\W$ is strictly $\W$-bilinear. Since it is $\V$-linear and colimit-preserving, and since $\W$ is $\V$-atomically generated, it suffices to check that the projection formula $w\otimes m^R(w')\to m^R(w\otimes w')$ is an equivalence for $\V$-atomic $w$. But now by assumption, $w$ is dualizable and so this map is an equivalence by \Cref{cor:projdbl}.  

    Conversely, assume that $\W$ is locally rigid, and let $x\in\W$ be $\V$-atomic. It is classified by a $\V$-linear functor $\V\to \W$ which, by assumption, is a $\V$-linear left adjoint. The $\W$-linear functor $x\otimes -: \W\to \W$ is obtained by composing $\W\otimes (\V\xrightarrow{x\otimes -}\W)$ and the multiplication $\W\otimes_\V\W\to \W$. They are both $\W$-internal left adjoints : the first one by assumption on $x$, the second one by local rigidity. Thus the composite is as well, and so by \cite[Example 1.23]{maindbl} $x$ is dualizable. 
\end{proof}
\begin{ex}
    $\Sp_{T(n)},\Sp_{K(n)},\Sp_p$ are locally rigid over $\Sp$ because they are compactly generated (by $T(n)/K(n)$-localizations of compact type $n$ spectra for the first two, and by $\Sph/p$ for the second) and compacts therein are dualizable.
\end{ex}
In the compactly generated case, the typical situation is the following: 
\begin{ex}
    Let $\V\in\CAlg(\PrL_{\st})$, and let $K\subset\V$ be a (small) set of dualizable objects and let $\V_{K-tors}$ be the localizing tensor ideal it generates. The inclusion $\V_{K-tors}\subset V$ is an internal left adjoint in $\Mod_\V$, and one can prove that its right adjoint witnesses $\V_{K-tors}$ as a symmetric monoidal localization of $\V$. 

    With this structure, $\V_{K-tors}$ is a locally rigid $\V$-algebra. 
\end{ex}
\begin{ex}\label{ex:locrigloc}
    Let $f:\V\to \W$ be a locally rigid morphism, and $L:\W\to \W_L$ a symmetric monoidal localization which is also a $\V$-internal left adjoint. In this case, $\W_L$ is $\V$-locally rigid. 
\end{ex}
\begin{proof}
    The fact that $L$ is a $\V$-internal left adjoint and a localization proves that $\W_L$ is a $\V$-linear retract of $\W$, hence it is dualizable. 

    Now the multiplication map $\W_L\otimes_\V\W_L\to \W_L$ is a $\V$-internal left adjoint by \cite[Lemma 1.30]{maindbl} and the fact that $\W\otimes_\V\W\to \W_L\otimes_\V\W_L\to \W_L$ is equivalently $\W\otimes_\V~\W\to~\W\to~\W_L$, both of which are $\V$-internal left adjoints. To obtain $\W_L$-linearity, we note that the map $\W\xrightarrow{m^R}\W\otimes_\V\W\to \W_L\otimes_\V\W_L$ is (say) left $\W$-linear, with values in a left $\W_L$-module, and hence factors through $\W\to \W_L$. It follows that the following diagram 
    \[\begin{tikzcd}
	{\W\otimes_\V\W} & \W \\
	{\W_L\otimes_\V\W_L} & {\W_L}
	\arrow[from=1-1, to=2-1]
	\arrow[from=1-1, to=1-2]
	\arrow[from=2-1, to=2-2]
	\arrow[from=1-2, to=2-2]
\end{tikzcd}\]
is horizontally right adjointable. Contemplating this gives $\W_L$-linearity (on either side) of $\W_L\xrightarrow{m^R}\W_L\otimes_\V\W_L$.
\end{proof}
The following is a key example outside of the compactly generated world:
\begin{ex}
    It is proved in \cite[Proposition 2.54]{maindbl} that whenever $X$ is a locally compact Hausdorff topological space, $\Sh(X)$ was dualizable. We claim that it is in fact locally rigid. Indeed, the multiplication map is given by $\Delta^*: \Sh(X)\otimes \Sh(X)\simeq \Sh(X\times X)\to \Sh(X)$, so its right adjoint is $\Delta_*$, which preserves colimits as $\Delta$ is proper, cf. \cite[7.3.1.5, 7.3.1.13, 7.3.1.15,7.3.1.16]{HTT}.

    We now have to argue that $\Delta_*$ is $\Sh(X\times X)$-linear. This is the content of e.g. \cite[Proposition 6.12]{volpe} ($\Delta$ is proper so $\Delta_! \simeq \Delta_*$). 
\end{ex}
We now study locally rigid $\V$-algebras in more generality. 
\begin{prop}\label{prop:actintladj}
Let $f:\V\to \W \in\CAlg(\PrL)$. 
    The multiplication map $\W\otimes_\V\W\to \W$ is a $\W\otimes_\V\W$-internal left adjoint if and only if for all $\W$-modules $\M$, the counit/action map $\W\otimes_\V\M\to \M$ is a $\W$-internal left adjoint.
\end{prop}
\begin{proof}
    First, note that it suffices for the multiplication map to be an internal left adjoint as a left (say) $\W$-module, by symmetry. Thus, clearly the given condition implies that the multiplication is an internal $\W\otimes_\V\W$-left adjoint by plugging in $\M=\W$.
    
For the converse, note that $-\otimes_\W \M: \Mod_{\W\otimes_\V\W}(\PrL)\to \Mod_\W(\PrL)$ is a $2$-functor, so it preserves internal left adjoints, which allows us to conclude. 
\end{proof}
\begin{ex}\label{ex:complocrig}
  Let $f:\V\to \W, g:\W\to \mathcal U$ be two locally rigid morphisms in $\CAlg(\PrL)$. The composite $gf$ is also locally rigid. 
\end{ex}
\begin{proof}
  Firstly, $\mathcal U$ is dualizable over $\W$ and $\W$ over $\V$, so $\mathcal U$ is dualizable over $\V$ by \cite[Proposition 4.6.4.4.]{HA}. 

  Second, by the above, we need to argue that $\mathcal U\otimes_\V\M \to \M$ is a $\mathcal U$-internal left adjoint for all $\mathcal U$-modules $\M$. We write it as a composite $\mathcal U\otimes_\W (\W\otimes_\V\M)\to \mathcal U\otimes_\W\M\to \M$, where it is now clear that both arrows are internal left adjoints. 
\end{proof}

\begin{lm}\label{lm:tensoroneat}
    Let $\M$ be a dualizable $\V$-module, and $\alpha: x\to y$ a $\V$-atomic map in $\M$. In this case, $\one_\W\boxtimes \alpha \in \W\otimes_\V\M$ is $\W$-atomic. 

    In fact, there is a natural map $\one_\W\otimes\at_\M^\V(x,y)\to \at_{\W\otimes_\V\M}^\W(\one_\W\boxtimes x,\one_\W\boxtimes y)$ compatible with the projections to $\hom_\M^\V$ and $\hom_{\W\otimes_\V\M}^\W$. 
\end{lm}
\begin{proof}
Fix $\lambda$ large enough so that $x$ is $\lambda$-compact.

The map $\PP_\V(\M^\lambda)\to \W\otimes_\V\PP_\V(\M^\lambda)$ induces a map $$\one_\W\otimes\at_\M^\V(x,y)= \one_\W\otimes \hom^\V_{\PP_\V(\M^\lambda)}(j(x),\hat j(y))\to\hom_{\W\otimes_\V\PP_\V(\M^\lambda)}(\one_\W\otimes j(x), \one_\W\otimes \hat j(y))$$
where we have used \Cref{cor:descofat}. 

Now $\PP_\V(\M^\lambda)\to \W\otimes_\V\PP_\V(\M^\lambda)$ sends $\V$-atomics to $\W$-atomics, and therefore the latter is equivalent to $\at^\W_{\W\otimes_\V\PP_\V(\M^\lambda)}(\one_\W\otimes j(x), \one_\W\otimes \hat j(y))$. 

Precomposing by $\hat j(x)\to j(x)$ thus yields in total, a map $$\one_\W\otimes\at_\M^\V(x,y)\to \at^\W_{\W\otimes_\V\PP_\V(\M^\lambda)}(\one_\W\otimes \hat j(x),\one_\W\otimes \hat j(y))$$
Finally, using \Cref{cor:iLffat} (since the functor $\one_\W\otimes \hat j$ is a fully faithful $\W$-internal left adjoint), the latter is equivalent to $\at^\W_{\W\otimes_\W\M}(\one_\W\otimes x,\one_\W\otimes y)$, and it is easy to check that all of this is happening over the relevant $\hom$-objects, so that this proves the claim.  
\end{proof}
\begin{cor}\label{cor:locrigimplat}
    Let $\W$ be a locally rigid $\V$-algebra. In any dualizable $\W$-module $\M$, any $\V$-atomic map is $\W$-atomic.

    More precisely, there is a natural map $\one_\W\otimes\at_\M^\V(x,y)\to \at_\M^\W(x,y)$ compatible with the projections to $\hom_\M^\V$ and $\hom_\M^\W$. 
\end{cor}
\begin{proof}
    Combine the previous lemma  with \Cref{prop:actintladj} and \Cref{ex:presat}. 
\end{proof}
This allows us to reformulate the definition of local rigidity in the following way for $\V=\Sp$: 
\begin{prop}\label{prop:locrigcharac}
    $\W$ is locally rigid over $\Sp$ if and only if it is dualizable and any $\Sp$-atomic map in $\W$ is trace class; if and only if any dualizable $\W$-module is $\Sp$-dualizable and any $\Sp$-atomic map in a dualizable $\W$-module is $\W$-atomic. 
\end{prop}
Rephrasing this in terms of compact maps, we find that $\W$ is locally rigid over $\Sp$ if and only if it is dualizable and any compact map is trace-class. 
\begin{proof}
The middle claim is clearly a special case of the last claim, so it suffices to prove that the first one implies that last one, and that the middle one implies the first one. 

So first assume that $\W$ is locally rigid. As it is dualizable, any dualizable $\W$-module is $\V$-dualizable \cite[Proposition 4.6.4.4.]{HA}. Now if $\M$ is dualizable over $\W$ and $\alpha$ is $\Sp$-atomic in $\M$, then by \Cref{lm:tensoroneat}, $\one_\W\boxtimes\alpha$ is $\W$-atomic in $\W\otimes\M$. By \Cref{prop:actintladj}, the counit $\W\otimes\M\to\M$ is an internal left adjoint and so it sends $\W$-atomic maps to $\W$-atomic maps, so that $\alpha$ is $\W$-atomic too. 

Conversely, suppose $\W$ is $\Sp$-dualizable, and $\Sp$-atomic maps in $\W$ are trace-class. We have to prove that to $\W\otimes\W\to \W$ is a $\W$-internal left adjoint. First note that it sends pairs $f\boxtimes g$ of $\Sp$-atomic maps to the tensor product of a $\Sp$-atomic map and a trace-class map (by assumption) which is thus $\Sp$-atomic by \Cref{lm:trclat}. Thus it sends enough $\Sp$-atomic maps to $\Sp$-atomic maps, and is thus a $\Sp$-internal left adjoint by \cite[Addendum 2.42]{maindbl}.

We now prove that the right adjoint $\mu^R$ is actually $\W$-linear on (say) the left. For this, we use \Cref{cor:nuclax}. Writing any object of $\W$ as a colimit of compactly exhaustible objects, we see that $f^R$ is $\W$-linear. 
\end{proof}
\begin{rmk}
    It is clear that to prove that $\W$ is locally rigid, it suffices to produce a collection of compact maps which are all trace-class and which are ``enough'' compact maps; the exact same proof works as having enough compact maps certainly implies being dualizable. 
\end{rmk}

\begin{prop}\label{prop:locrigdbl}
    Let $\W$ be locally rigid, and $\M$ a $\W$-module which is dualizable over $\V$. In this case, $\M$ is dualizable over $\W$. 
\end{prop}
This proposition is true in general. We give a first proof in the case $\V=\Sp$, based on the technology of compact maps, and later give a proof for general $\V$-the latter proof is not particularly more complicated, but we feel the former showcases how the technology of compact maps allows us to reason in ways very similar to the compactly generated situation. 

We note that this statement is already contained in \cite[Appendix C]{kazhdan}, although with a different proof (even in the general case). 
\begin{proof}[Proof for $\V=\Sp$]
    Piecing things together, we find that $\Sp$-atomic maps in $\M$ are $\W$-atomic (for this in the previous proof we only need the underlying $\Sp$-module to be dualizable !). 
    
Thus the colimit-closure of $\W$-atomic telescopes in $\M$ is the whole of $\M$, and thus by \Cref{thm:dblexhaustionV} $\M$ is dualizable over $\W$. 
\end{proof}
In the general case, we propose two distinct proofs, one to illustrate some general $2$-categorical methods that might be relevant other places, and one to illustrate that one can simply copy the $\Sp$-proofs by replacing the word ``compact map'' with the word ``atomic map''. 
\begin{proof}[First proof in the general case]
Note that $\W\otimes_\V\M$ is dualizable over $\W$, and by \Cref{prop:actintladj}, the action map $\W\otimes_\V\M\to \M$ is an internal left adjoint. Its right adjoint is clearly conservative, so we may conclude by \cite[Proposition 1.41]{maindbl}. 
\end{proof}
\begin{proof}[Second proof in the general case]
    We have maps $\one_\W\otimes\at_\M^\V(-,y)\to \at_\M^\W(-,y)$ for all $y\in \M$ by \Cref{cor:locrigimplat}. It follows from this together with \Cref{prop:compatibilityweightedcolim} that $\V$-atomic presentations in $\M$ are also $\W$-atomic presentations, and so by \Cref{thm:dblexhaustionV}, $\M$ is dualizable over $\W$. 
\end{proof}
\begin{prop}\label{prop:intladjlocrig}
    Let $\W$ be locally rigid over $\V$, and let $f:\M\to \N$ be a $\W$-linear functor between $\W$-modules. If $f$ is a $\V$-internal left adjoint, then it is a $ \W$-internal left adjoint. 
\end{prop}
Similarly to before, we can give a compact-map based proof for $\V=\Sp$, and we will later give two proofs in general: a categorical proof in general and an atomic-map based proof. As before, a different proof appears in \cite[Appendix C]{kazhdan}. 
\begin{proof}[Proof for $\V=\Sp$]
    This is the same argument as the end of the proof of \Cref{prop:locrigcharac}.
\end{proof}
\begin{proof}[First proof in general]
The same argument as in \Cref{prop:locrigdbl} works, as the action map $\W\otimes_\V\M\to \M$ is natural in the input, and we can use \cite[Lemma 1.30]{maindbl} in place of \cite[Proposition 1.41]{maindbl}. 
\end{proof}
\begin{proof}[Second proof in general]
    This is essentially the same argument as in the case of $\Sp$. More precisely, using \Cref{cor:locrigimplat} and \Cref{prop:compatibilityweightedcolim}, we may use \Cref{add:enough} with the weights $W_s$ being of the form $\one_\W\otimes\at_\M^\V(-,x_s)$, together with the fact that these are preserved by \Cref{ex:presat}. 
\end{proof}

We can use these results to prove a converse to \Cref{ex:complocrig}:
\begin{cor}\label{cor:2outof3locrig}
    Let $f:\V\to\W, g:\W\to \mathcal U$ two morphisms in $\CAlg(\PrL)$. If $f$ and $gf$ are locally rigid, then $g$ is also locally rigid. 
\end{cor}
\begin{rmk}
    We do not have a full $2$-out-of-$3$ for local rigidity, as the example $\Sp\to \Fun(\mathbb N,\Sp)\to \Sp$ shows. 
\end{rmk}
\begin{proof}
    By \Cref{prop:locrigdbl}, $\mathcal U$ is dualizable over $\W$. Consider now the multiplication map $\mathcal U\otimes_\W\mathcal U\to\mathcal U$, which we need to show is a $\mathcal U$-internal left adjoint on either side. Note that the map $\mathcal U\otimes_\V\mathcal U\to \mathcal U\otimes_\W\mathcal U$ is a map of the form $\M\to \M\otimes_{\W\otimes_\V\W}\W$ and so by local rigidity of $\W$, it is a $\W$-internal left adjoint, and hence a $\V$-internal left adjoint. Furthermore, its image generates the target under colimits. Thus, by \cite[Lemma 1.30]{maindbl} and local rigidity of $gf$, the map $\mathcal U\otimes_\W\mathcal U\to \mathcal U$ is a $\V$-internal left adjoint. By \Cref{prop:intladjlocrig}, it is a $\mathcal U$-internal left adjoint, as was to be shown. 
\end{proof}
We now prove a second stability property of locally rigid categories. 
\begin{prop}\label{prop:colimlocrig}
    Any colimit of commutative locally rigid  $\V$-algebras along internal left adjoints is locally rigid.  
\end{prop}
\begin{proof}
    We prove that this is so for finite coproducts and for sifted colimits, from which the general case follows. 

    The former are simply tensor products, and it is clear from the definition that a tensor product of locally rigid $\V$-algebras is locally rigid. 

    Now let $\W_\bullet: I\to \CAlg(\Mod_\V(\PrL))$ be a sifted diagram with locally rigid values, and internal left adjoints as transition maps. As $I$ is sifted, the colimit $\W$ of this diagram exists and is computed in $\Mod_\V(\PrL)$. 

    Fix $\M$ a $\W$-module, by \Cref{prop:actintladj} it suffices to prove that the action map $\W\otimes_\V\M\to \M$ is a $\W$-internal left adjoint. $\W$ is generated under colimits by the images of the $\W_i$'s, so to check that it is an internal left adjoint, it suffices to check that it is a $\V$-internal left adjoint, and a $\W_i$-internal left adjoint for all $i$. 

    For the former, we note that it is a colimit along internal left adjoints of $\W_i\otimes_\V \M\to \M$, each of which is an internal left adjoint as $\W_i$ is locally rigid. 

The latter works the same, noting now that for each $i$, the forgetful map $I_{i/}\to I$ is cofinal by siftedness of $I$, so that we can instead assume we had a diagram of locally rigid $\W_i$-algebras along $\W_i$-internal left adjoints by \Cref{prop:intladjlocrig}. 
\end{proof}
\begin{rmk}
This is not true if the functors are not internal left adjoints. This will be a crucial difference with rigid algebras: we will see that in this case, any symmetric monoidal colimit-preserving functor \emph{is} an internal left adjoint. 

A counterexample can be constructed from locally rigid $\Sp$-algebras of the form $\prod_I \Sp$.
\end{rmk}
We will study further categorical properties of the category of locally rigid categories in \Cref{section:rigprlV}. 

We now give an \textit{a posteriori} description of the duality datum of a locally rigid category. Recall that dualizability is assumed in the \emph{definition} of local rigidity\footnote{In contrast to the case of rigidity where it could be omitted, cf. \Cref{lm:rigimpldbl}.}. However, once we have assumed it, the duality datum has a simple description :
\begin{prop}\label{prop:exceptional}
    Let $\W$ be a commutative locally rigid $\V$-algebra. There is an ``exceptional'' $\V$-linear colimit-preserving functor $\Gamma_! : \W\to \V$ such that the following datum exhibits $\W$ as a self-dual over $\V$: $$\V\xrightarrow{\eta}\W\xrightarrow{\mu^R}\W\otimes_\V\W, \qquad \W\otimes_\V\W\xrightarrow{\mu}\W\xrightarrow{\Gamma_!}\V $$
\end{prop}
\begin{proof}
see \cite[Lemma C.3.7]{kazhdan} for a discussion.

One can use the rigid case (see \Cref{lm:rigimpldbl}) as well as the theory of rigidification (see \Cref{lm:rigfprops}) to give an alternative proof. We will describe this alternative proof later in \Cref{section:rigidification}. We will also see there that $\Gamma_!$ is quite explicit: it is given by $\at_\W^\V(\one_\W,-)$; see \Cref{thm:locrigrigleft}, \Cref{prop:exceptional2} and \Cref{cor:exceptional=at}. 
\end{proof}
We give two examples:
\begin{ex}
For $\W= \Sp_p$, $\Gamma_!$ is left adjoint to the localization functor $\Sp\to \Sp_p$. Namely, there is an equivalence between $p$-complete spectra and locally $p$-nilpotent spectra, $\Gamma_!$ is that equivalence followed by the inclusion of the latter inside spectra.  

A similar description exists for $\Sp_{T(n)/K(n)}$, using the fact that the localization $L_n\Sp\to~\Sp_{K(n)}$ (resp. $L_n^f\Sp\to\Sp_{T(n)}$) admits a left adjoint.
\end{ex}
\begin{ex}
    If $X$ is a locally compact Hausdorff space, and $p:X\to \pt$ is the unique map, then $\Gamma_! = p_! = \Gamma_c$ is the compactly supported sections functor, which explains the notation. 
\end{ex}
In particular, $\W$ is self-dual over $\V$, which gives the following:
\begin{cor}
    Let $\W$ be a commutative locally rigid $\V$-algebra, and $\M$ a $\W$-module. Post-composition with $\Gamma_!$ induces an equivalence $$\Fun^L_\W(\M,\W)\to \Fun^L_\V(\M,\W)\to \Fun^L_\V(\M,\V)$$
\end{cor}
\begin{proof}
    In general, evaluation at $\one_\W$ gives an equivalence\footnote{This is a special case of the co-extension of scalars adjunction, $\hom_A(M,\hom(A,N))\simeq \hom(M,N)$ where $M$ is an $A$-module.} $$\Fun^L_\W(\M,\Fun_\V^L(\W,\V))\to \Fun^L_\V(\M,\Fun_\V^L(\W,\V))\to \Fun_\V^L(\M,\V)$$
    In the locally rigid case, the functor $\Gamma_!:\W\to \V$ is identified with $\mathrm{ev}_{\one_\W}:~\Fun^L_\V(\W,\V)\to~\V$.
\end{proof}

Local rigidity has a (one-directional) implication about atomic maps (and thereby, objects, as well as internal left adjoints). Rigidity, which we study in more detail in the next subsection, is about reversing this one-directionality.

\begin{ques}[Local compactness]\label{quest:loccpt}
Let $\W$ be locally rigid over $\Sp$ and compactly generated, and $f:x\to y$ a map in a (dualizable) $\W$-module. Suppose that for any compact $K\in\W$, $K\otimes f$ is a compact map\footnote{It makes sense to call these maps ``locally compact maps'', see the following definition for a slightly more general notion.}. Does it follow that $f$ is $\W$-atomic ? 
\end{ques}
We are not able to fully answer this question, though we provide the following convenient alternative, which we will use later to understand limits of rigid categories:
\begin{defn}
    Let $\W$ be a locally rigid commutative $\Sp$-algebra, $\M$ a dualizable $\W$-module and let $f:x\to y$ in $\M$. Say $f$ is \emph{locally compact} if for any compact map in $\W$, $g: k\to k'$, the tensor product $f\otimes g: k\otimes  x\to k'\otimes y$ is compact in $\M$.  
\end{defn}
\begin{lm}\label{lm:2foldloccpct}
Let $\W$ be locally rigid over $\Sp$ and $\M$ a dualizable $\W$-module. Assume $\one_\W$ is $\omega_1$-compact. 

    Let $x\xrightarrow{f}y\xrightarrow{g}z$ be two locally compact maps in $\M$.  

    In this case, the composite $g\circ f: x\to z$ is $\W$-atomic in $\M$. 
\end{lm}
\begin{rmk}\label{rmk:2foldtrcl}
    It is worth examining the case $\W= \M$. In this case, the conclusion is that the composite is actually trace-class, thus we are to produce a \emph{lift} $\one_\W\to \hom(x,\one_W)\otimes z$ from the \emph{property} that $g,f$ are locally compact. 
\end{rmk}
The proof will use the following lemma, the trace-class version of which (cf. the previous remark) is probably folklore:
\begin{lm}
Let $\W$ be locally rigid over $\Sp$ and $\M$ a dualizable $\W$-module. 
    Let $f:~x\to~y, g:~y\to~z$ be two maps in $\M$, and assume $f,g$ are both $\W$-atomic. 

    Let $f_0,f_1: \one_\W\to \at_\M(x,y)$ be two lifts of the map $\tilde f: \one_\W\to \hom_\M(x,y)$ classifying $f$. The lifts $g\circ f_0, g\circ f_1$ of $g\circ f$ in $\at_\M(x,z)$ obtained by functoriality of $\at_\M(x,-)$ are homotopic \emph{as lifts}. 
\end{lm}
The homotopy between $g\circ f_0$ and $g\circ f_1$ will \emph{depend} on a witness of atomicity of $g$, but the fact that they are homotopic does not depend on that. 
\begin{proof}
Fix a lift $\hat g:\one_\W\to\at_\M(y,z)$ of $\tilde g:\one_\W\to \hom_\M(y,z)$, and consider the following diagram: 
\[\begin{tikzcd}
	{\at_\M(x,y)} & {\at_\M(x,z)} \\
	{\hom_\M(x,y)} & {\hom_\M(x,z)}
	\arrow["{g\circ -}", from=1-1, to=1-2]
	\arrow[from=1-1, to=2-1]
	\arrow[from=1-2, to=2-2]
	\arrow[dashed, from=2-1, to=1-2]
	\arrow["{g\circ-}"', from=2-1, to=2-2]
\end{tikzcd}\]
The outer diagram is simply the naturality diagram for $\at_\M(x,-)\to \hom_\M(x,-)$. 

We claim that there is a dashed arrow filling this diagram, namely the composite $\hom_\M(x,y)= \hom_\M(x,y)\otimes \one_\W\xrightarrow{\id\otimes\hat g}\hom_\M(x,y)\otimes\at_\M(y,z)\to \at_\M(x,z)$, where the last map is from \Cref{lm:slightlycoherentatid}. That the inner triangles commute also follows from \textit{loc. cit.}. 

Thus for $ h\in\at_\M(x,y)$, $g\circ h$ only depends on the image of $h$ in $\hom_\M(x,y)$, naturally. By definition, $f_0,f_1$ have the same image in $\hom_\M(x,y)$, and so we are done. 
    \end{proof}
    Note that the exact same proof yields:
\begin{add}\label{add:doubletrcl}
    The version of this lemma where $\M=\W$, ``atomic'' is replaced with ``trace-class'' and $\W$ is allowed to be arbitrary with internal homs (that is, not necessarily presentable) is also valid.
\end{add}

\begin{proof}[Proof of \Cref{lm:2foldloccpct}]
We have to produce a lift $\one_\W\to \at_\M(x,z)$ of $\widetilde{gf}: \one_\W\to \hom_\M(x,z)$. 

Using \cite[Corollary 2.34]{maindbl} and the assumption that $\one_\W$ is $\omega_1$-compact, we may write it as $\colim_n w_n$ with each $j_n: w_n\to w_{n+1}$ being compact. We let $\iota_n: w_n\to\one_\W$ denote the 
cocone maps. 

The strategy is now as follows: \begin{enumerate}
    \item For each $w_n$, produce a map $w_n\to \at_\M(x,y)$ fitting in a diagram: \[\begin{tikzcd}
	{w_n} & {\at_\M(x,y)} \\
	{w_{n+1}} & {\hom_\M(x,y)} 
	\arrow[dashed, from=1-1, to=1-2]
	\arrow["{j_n}"', from=1-1, to=2-1]
	\arrow[from=1-2, to=2-2]
	\arrow["{\tilde f\circ \iota_n}"', from=2-1, to=2-2]
\end{tikzcd}\] 
\item Observe that, while the lifts $w_n\to\at_\M(x,y)$ are not necessarily compatible as $n$ varies, their composition to $w_n\to \at_\M(x,z)$ \emph{are} by the previous lemma; 
\item Glue the pieces together via $\one_\W= \colim_nw_n$
\end{enumerate}
The second and third step are immediate, so we focus on the first, which is where we use the locally compact assumption. We simply point out that the second step is why we need there to be two maps, $f$ and $g$.

Before doing so, we point out the following natural map, constructed as usual using the universal property of $\at_\M(x,-)$: it is a map $\at_\M(v\otimes x, y)\to \hom(v,\at_\M(x,y))$, which is the mate of $v\otimes \at_\M(v\otimes x,y)\to \at_\M(x,y)$. The latter is constructed using the universal property of $\at_\M(x,-)$ and the natural map $v\otimes\hom(v\otimes x,y)\to \hom(x,y)$.

With this map in hand, here is the lift: by local compactness of $f$, we find that the composite $w_n\otimes x\to w_{n+1}\otimes y\to y$ is compact, and hence $\W$-atomic. It thus corresponds to a map $\one_\W\to \at_\M(w_n\otimes x,y)$ which, by the previous map, gives us a map $w_n\to \at_\M(x,y)$. Chasing through the definitions, we find that it fits as a dotted lift in the diagram, as claimed.  
\end{proof}
\begin{rmk}
    It is unclear to the author how optimal this is. It is very likely that with some good enough understanding of the posets $\omega_n$, as $n$ varies, one can prove with similar techniques a similar lemma involving composites of $n+1$ locally compact maps. But it is also possible that all of this is inessential, and that the answer to \Cref{quest:loccpt} is simply ``yes'', or maybe yes with composites of two locally compact maps with no set-theoretic assumptions. 
\end{rmk}
\begin{rmk}
    The condition that $\one_\W$ be $\omega_1$-compact corresponds, in topology, to something like ``countable at infinity''. 
\end{rmk}

We conclude this section with a remarkable property of locally rigid categories: 
\begin{prop}
    Let $\V\in\CAlg(\PrL_{\st})$ be locally rigid\footnote{As we will see in the proof, we only need $\V$ to be generated under colimits by trace-class exhaustible objects.}, and $f: \V\to \W$ any morphism in $\CAlg(\PrL_{\st})$. If for all $v\in\V$, the induced map $\map(\one_\V,v)\to \map(\one_\W,f(v))$ is an equivalence, then $f$ is fully faithful. 
\end{prop}
\begin{proof}
Fix $v_0,v_1\in\V$. We wish to prove that the induced map $\map(v_0,v_1)\to \map(f(v_0),f(v_1))$ is an equivalence. 

As $\V$ is dualizable, it is generated under colimits by compactly-exhaustible objects, and as $\V$ is locally rigid, these are trace-class exhaustible. So we may reduce to the situation where $v_0= \colim_\mathbb N x_n$ where each transition map $x_n\to x_{n+1}$ is trace-class. 

The goal is now to rewrite $\map(v_0,v_1) = \map(\one_\V, \hom(v_0,v_1))$ and $\hom(v_0,v_1) = \lim_\mathbb N\hom(x_n, v_1)$. 

We note that for any $y\in\V$, the trace-class-ness of $x_n\to x_{n+1}$ guarantees that the map $f(\hom(x_n,y))\to \hom(f(x_n),f(y))$ is a pro-equivalence: this can be checked by constructing the following map :$$\hom(f(x_{n+1}),f(y))\to f(\hom(x_n,\one_\V))\otimes f(x_{n+1}) \otimes \hom(f(x_{n+1}), f(y)) \to f(\hom(x_n,\one_\V))\otimes f(y)\to f(\hom(x_n, y))$$

Combining all these pro-isomorphisms with our assumption, we find, along the canonical maps: 
$$\map(\one_\V,\hom(v_0,v_1))\simeq \lim_\mathbb N\map(\one_\V, \hom(x_n, v_1)) \simeq \lim_\mathbb N\map(\one_\W, f(\hom(x_n,v_1))\simeq \lim_\mathbb N\map(\one_\W, \hom(f(x_n),f(v_1)) $$
$$\simeq \map(\one_\W, \hom(f(v_0),f(v_1))$$
\end{proof}
\subsection{Rigidity}
\begin{defn}
    Let $f:\V\to \W$ be a morphism in $\CAlg(\PrL)$. It is rigid, or $\W$ is a rigid $\V$-algebra, if it is locally rigid and the unit of $\W$ is $\V$-atomic. 
\end{defn}
\begin{ex}\label{ex:atgenrig}
    Let $f:\V\to\W$ be a morphism in $\CAlg(\PrL)$. Assume $\W$ is $\V$-atomically generated, that the unit $\one_\W$ is $\V$-atomic, and that every $\V$-atomic is dualizable. 

    In this case, $\W$ is rigid over $\V$.
\end{ex}
\begin{proof}
    This follows from the definition and from \Cref{ex:locrigat}. 
\end{proof}
\begin{ex}
    For any commutative algebra $R\in\CAlg(\V)$, $\Mod_R(\V)$ is a rigid $\V$-algebra. 
\end{ex}
\begin{proof}
    It is clear that it is $\V$-atomically generated and that its unit ($R$) is $\V$-atomic. Let $M\in~\Mod_R(\V)$ be $\V$-atomic. We note that the composite $\Mod_R(\V)\xrightarrow{\hom(M,-)}\Mod_R(\V)\to \V$, where the second map is the forgetful functor, is simply $\hom_R(M,-)$. The forgetful functor is colimit-preserving and conservative, thus one can check the fact that $\hom(M,-)$ is $\V$-linear and colimit-preserving by observing that $\hom_R(M,-)$ is (because $M$ is $\V$-atomic). 

    Now to check $\Mod_R(\V)$-linearity, we observe that $\Mod_R(\V)$ is generated under ($\Delta\op$-)colimits by the image of $\V\to \Mod_R(\V)$. 
\end{proof}
\begin{ex}
    Let $C$ be a small symmetric monoidal category in which every object is dualizable. The commutative $\V$-algebra $\Fun(C,\V)$ equipped with Day convolution is rigid over $\V$. 
\end{ex}
\begin{proof}
    $\Fun(C,\V)\simeq \Fun(C,\Ss)\otimes \V$ so it suffices to prove that $\Fun(C,\Ss)$ is rigid over $\Ss$. For this, we note that it is atomically generated by the image of the Yoneda embedding, that its unit, $\Map(\one_C,-)$ is atomic, and finally that its $\Ss$-atomics are exactly the retracts of the image of the Yoneda embedding, so that they are all dualizable by assumption (and the fact that retracts of dualizables are dualizable in idempotent-complete categories). 
\end{proof}
\begin{rmk}
    In the above proof, we reduced to $\Ss$ because for a general $\V$, $\Fun(C,\V)$ will have more atomics than simply the image of the Yoneda embedding. Using the concept of absolute colimits, one can make do without this reduction, and in fact prove the results for small symmetric monoidal $\V$-enriched categories all of whose objects are dualizable. We will not do this here. 
\end{rmk}
\begin{ex}
    Let $f:\V\to \W$ be a rigid morphism, and $L:\W\to \W_L$ a symmetric monoidal localization which is also a $\V$-internal left adjoint. In this case, $\W_L$ is $\V$-rigid. 
\end{ex}
\begin{proof}
    We have seen in \Cref{ex:locrigloc} that it is locally rigid. Since $L:\W\to \W_L$ is a $\V$-internal left adjoint, it sends the $\V$-atomic object $\one_\W$ to a $\V$-atomic object, namely $\one_{\W_L}$. 
\end{proof}
\begin{ex}
    Combining the two previous examples, we find that if $\V$ is semi-additive and $C$ is a small symmetric monoidal category with finite products that are compatible with the tensor product, and in which every object is dualizable, $\Fun^\times(C,\V)$ is rigid over $\V$, as a $\V$-linear localization of $\Fun(C,\V)$. 
    
For example, $\Sp_G$ is rigid over $\Sp$ whenever $G$ is a finite group\footnote{This is also true for compact Lie groups, but there, $\Sp_G$ is no longer has a ``Mackey functor''-description.}. 
\end{ex}
\begin{ex}
    We saw that if $X$ is locally compact Hausdorff, then $\Sh(X)$ is locally rigid. If $X$ is in fact compact, then its unit is compact, so that $\Sh(X)$ is rigid. 
\end{ex}
\begin{prop}\label{prop:unitatimplat}
    Let $f:\V\to\W$ be a morphism in $\CAlg(\PrL)$ with $\one_\W$ being $\V$-atomic. In this case, $\W$-atomic maps are $\V$-atomic. 
\end{prop}
\begin{proof}
Let $\M$ be a $\W$-module, and $x\in\M$, and consider the composite $$f^R\at^\W_\M(x,-)\to f^R\hom_\M^\W(x,-)\simeq \hom_\M^\V(x,-)$$ As $\one_\W$ is $\V$-atomic, $f^R$ is colimit-preserving and $\V$-linear, and so $f^R\at_\M^\W(x,-)$ also is. It follows that this map factors through $\at_\M^\V(x,-)$ and thus the claim follows. 
\end{proof}
\begin{ex}
  Let $f:\V\to \W, g:\W\to \mathcal U$ be two rigid morphisms in $\CAlg(\PrL)$. The composite $gf$ is also rigid. 
\end{ex}
\begin{proof}
    By \Cref{ex:complocrig}, it is locally rigid. Furthermore, $\V\to \W$ is a $\V$-internal left adjoint as $\one_\W$ is $\V$-atomic; $\W\to \mathcal U$ is a $\W$-internal left adjoint as $\one_\mathcal U$ is $\W$-atomic, but now this immediately implies that it is als a $\V$-internal left adjoint, so the composite $\V\to \mathcal U$ also is, thus proving that $\one_\mathcal U$ is $\V$-atomic.

    Alternatively, $\one_\mathcal U$ is $\W$-atomic by assumption, and thus $\V$-atomic by the above proposition. 
\end{proof}
Here, we can directly prove a weak converse:
\begin{cor}
    Let $f:\V\to\W, g:\W\to \mathcal U$ be two morphisms in $\CAlg(\PrL)$. If $f$ and $gf$ are rigid, so is $g$. 
\end{cor}
\begin{proof}
    By \Cref{cor:2outof3locrig}, $g$ is locally rigid, so it suffices to prove that $\one_{\mathcal U}$ is $\W$-atomic. Since it is $\V$-atomic, this follows from \Cref{cor:locrigimplat}. 
\end{proof}
\begin{lm}
    The unit $\one_\W$ is $\V$-atomic if and only if for all $\V$-modules $\M$, the unit map $\M\to \W\otimes_\V\M$ is a $\V$-internal left adjoint.
\end{lm}
\begin{proof}
Specializing to $\M=\V$ itself, we see that this condition implies that the unit is $\V$-atomic. 

Conversely, tensoring a $\V$-internal left adjoint such as $\V\to \W$ with any $\V$-module yields a $\V$-internal left adjoint.
\end{proof}
\begin{cor}
    $\W$ is rigid over $\V$ if and only if the unit and counit of the extension-restriction of scalars adjunction are (pointwise\footnote{It follows that they are actually also adjointable.}) internal left adjoints. 
\end{cor}
\begin{cor}\label{cor:rigadj}
    If $\W$ is rigid over $\V$, extension-restriction of scalars induces an adjunction $\Dbl{\V}\rightleftarrows \Dbl{\W}$. This adjunction is in fact monadic: there is a canonical equivalence $\Dbl{\W}\simeq \Mod_\W(\Dbl{\V})$. 
\end{cor}
\begin{rmk}
    From this, it follows that \cite[Theorem A]{maindbl} in the case where $\V$ is rigid over $\Sp$, is a consequence of the case where $\V=\Sp$ which is a bit more elementary than in full generality. 
\end{rmk}
\begin{obs}\label{obs:morrig}
    Let $f:\V\to \W$ be a symmetric monoidal colimit-preserving functor between commutative $\Sp$-algebras, where $\V$ is locally rigid and $\one_\W$ is compact. In this case, $f$ sends trace-class maps to trace-class maps, i.e. $\V$-atomic maps to $\W$-atomic maps, in particular (by local rigidity of $\V$) compact maps in $\V$ to $\W$-atomic maps, in particular (because $\one_\W$ is compact) compact maps in $\V$ to compact maps in $\W$. Because $\V$ is dualizable, this implies that $f$ is an internal left adjoint\footnote{And because $\V$ is locally rigid, this further implies that $f$ is a $\V$-internal left adjoint, i.e. $f^R$ satisfies the projection formula.}. 

    More generally, assume $\W_0\to \W_1$ is a symmetric-monoidal colimit-preserving functor between commutative $\V$-algebras, where $\W_0$ is locally rigid and the unit in $\W_1$ is $\V$-atomic. The latter condition implies, by tensoring with $\W_0$, that $\W_0\to \W_0\otimes_\V\W_1$ is a $\W_0$-internal left adjoint; and \Cref{prop:actintladj} implies that $\W_0\otimes_\V\W_1\to \W_1$ is also a $\W_0$-internal left adjoint. It follows that the map $\W_0\to \W_1$ is a $\W_0$-internal left adjoint, and thus by restricting scalars, a $\V$-internal left adjoint. 
\end{obs}
\begin{rmk}\label{rmk:kappacpctunit}
    The proof of this last paragraph works to show that for a good cardinal $\kappa$, if $\W_0$ is locally rigid and the unit of $\W_1$ is $\kappa$-compact, then any symmetric monoidal colimit-preserving functor $\W_0\to \W_1$ has a $\kappa$-filtered colimit-preserving right adjoint. 
\end{rmk}
\begin{lm}
    Suppose $\W$ is a commutative $\V$-algebra, whose underlying $\V$-module is atomically generated. In this case, $\W$ is locally rigid if and only all its $\V$-atomics are dualizable, and it is rigid if and only if $\V$-atomic and dualizables coincide.
\end{lm}
\begin{proof}
    The locally rigid case is \Cref{ex:atgenrig}. 

 Now if $\W$ is rigid then by the above, all $\V$-atomics are dualizable. The converse of course holds if $\one_\W$ is atomic.

    Conversely, assume atomics and dualizables agree. It follows again from the above that $\W$ is locally rigid, and because $\one_\W$ is atomic, it follows that $\W$ is rigid. 
\end{proof}
We now prove a more precise claim regarding the comparison of $\V$ and $\W$-atomic maps. First, we need a lemma. 
\begin{prop}
    Let $\V\to \W$ be a morphism in $\CAlg(\PrL)$, and $C$ a small $\W$-enriched category. 
    \begin{itemize}
        \item If the unit of $\W$ is $\V$-atomic, then the induced functor $j:\PP_\V(C)\to\PP_\W(C)$ is fully faithful; 
        \item If furthermore $\W$ is rigid over $\V$, and $C= \M^\kappa$ for some $\W$-module $\M\in \Mod_\V(\PrL_\kappa)$ and a good cardinal $\kappa$, then $j:\PP_\V(\M^\kappa)\to \PP_\W(\M^\kappa)$ is an equivalence. 
        
    \end{itemize}
\end{prop}
\begin{proof}
    1. We note that the source $\PP_\V(C)$ is $\V$-atomically generated, and $j$ sends $C$ to $\W$-atomics, and hence, by the assumption on $\one_\W$, to $\V$-atomics. 
    
    It follows that the right adjoint $j^R$ is $\V$-linear by \cite[Corollary 1.31]{maindbl}, and so to check that the unit map $\id\to j^Rj$ is an equivalence, it suffices to do so on the atomic generators. Since they generate, one can check that this map is an equivalence by mapping in from an atomic generator, so, ultimately it suffices to check that $j_{\mid C}$ is fully faithful, and this is clear by design.

    2. We prove that $\PP_\W(\M^\kappa)$ has the same universal mapping in property as $\PP_\V(\M^\kappa)$ in $\Dbl{\V}$. First, $\PP_\W(\M^\kappa)$ is dualizable over $\W$ and hence over $\V$. 

    Second, for $\N$ any dualizable $\V$-module, we have $$\Fun^{iL}_\V(\N,\PP_\W(\M^\kappa))\simeq \Fun^{iL}_\W(\W\otimes_\V\N,\PP_\W(\M^\kappa))$$
    by \Cref{cor:rigadj}. 

    And then, $\Fun^{iL}_\W(\W\otimes_\V\N,\PP_\W(\M^\kappa))\simeq \Fun^{L,\kappa}(\W\otimes_\V\N,\M)$ by \cite[Theorem 1.63]{maindbl}, and $\Fun^{L,\kappa}_\W(\W\otimes_\V\N,\M)\simeq \Fun^{L,\kappa}_\V(\N,\M)$, and using \cite[Theorem 1.63]{maindbl} again, the latter is $\Fun^{iL}_\V(\N,\PP_\V(\M^\kappa))$. 

    Unwinding the equivalences, one finds that the corresponding map $\PP_\V(\M^\kappa)\to \PP_\W(\M^\kappa)$ is the one we described. 
\end{proof}
\begin{cor}\label{cor:rigat=at}
    Let $f:\V\to\W$ be a rigid morphism. In this case, in $\W$-modules, $\W$-atomic maps and $\V$-atomic maps coincide. 
    
    In fact, for any $\W$-module $\M$ and $x\in\M$, the canonical map $\at_\M^\V(x,-)\to f^R\at_\M^\W(x,-)$ is an equivalence. 
\end{cor}
\begin{proof}
In \Cref{prop:unitatimplat} and \Cref{cor:locrigimplat}, we have constructed maps in both directions lying over $\hom_\M^\V(x,-)$. Note that this is already sufficient to prove that the notions of $\V$-atomic and $\W$-atomic maps coincide. Now we prove the more precise statement. Fix $\kappa$ large enough. 

The universal property of $\at_\M^\V(x,-)$ guarantees that the composite $$\at_\M^\V(x,-)\to f^R\at^\W_\M(x,-)\to \at_\M^\V(x,-)$$ is the identity, since this is the only self-map lying over $\hom^\V_\M(x,-)$. 

For the self-map of $f^R\at_\M^\W(x,-)$, we have to work a little bit more. Since $\hom_\M^\V(x,-)\simeq f^R\hom_\M^\W(x,-)$, what we have is a commutative triangle: 
\[\begin{tikzcd}
	{f^R\at_\M^\W(x,-)} & {f^R\at_\M^\W(x,-)} \\
	{} & {f^R\hom_\M^\W(x,-)}
	\arrow[from=1-1, to=1-2]
	\arrow[from=1-1, to=2-2]
	\arrow[from=1-2, to=2-2]
\end{tikzcd}\]
and if we could remove the $f^R$'s, we would know that the horizontal map would be an equivalence. 

The functor $\at_\M^\W(x,-)$ is $\W$-linear, so by \Cref{cor:rigselfdual} the top map is of the form $f^R(\alpha)$ for some $\alpha : \at^\W_\M(x,-)\to \at^\W_\M(x,-)$. The point now is to extend the conclusion \Cref{cor:rigselfdual} to $\V$-enriched functors.

More precisely, we claim that postcomposing with $f^R$ induces an equivalence $\Fun_\W(\M,\W)\simeq \Fun_\V(\M,\V)$.

Taking this for granted, we can conclude by observing that our triangle thus lifts to a triangle with $f^R$, where the same argument guarantees that the endomorphism $\at_\M^\W(x,-)\to~\at^\W_\M(x,-)$ is the identity. 

As is clear from the proof, we only have to prove this equivalence on accessible functors, in fact, $\kappa$-accessible for some fixed $\kappa$ is sufficient (just fix $x$, and pick $\kappa$ so that $\hom_\M^\W(x,-)$ and $f^R$ are $\kappa$-accessible). Now we note that we have the following equivalences: $$\Fun_\W^{\kappa-filt}(\M,\W)\simeq \Fun_\W(\M^\kappa,\W)\simeq \Fun^L_\W(\PP_\W(\M^\kappa),\W)\simeq \Fun^L_\V(\PP_\W(\M^\kappa),\V)$$
where the second equivalence comes from \Cref{thm: UPVPsh} and the third from \Cref{cor:rigselfdual}. 

On the other hand, $\Fun^{\kappa-filt}_\V(\M,\V)\simeq \Fun_\V(\M^\kappa,\V)\simeq \Fun^L_\V(\PP_\V(\M^\kappa),\V)$.

From this perspective, unwinding the previous equivalences, we find that the map given by postcomposing with $f^R$ is identified with $j^*:\Fun^L_\V(\PP_\W(\M^\kappa),\V)\to \Fun^L_\V(\PP_\V(\M^\kappa),\V)$, where $j:\PP_\V(\M^\kappa)\to \PP_\W(\M^\kappa)$ is the $\V$-linear map induced by the embedding $\V$-enriched $\M^\kappa\to \PP_\W(\M^\kappa)$, using \Cref{thm: UPVPsh}. 

Now by the previous lemma, $j$ is an equivalence, and therefore so is $j^*$, which proves our claim. 
\end{proof}
Similar to \Cref{prop:colimlocrig}, we have:
\begin{prop}\label{prop:colimrig}
    The full subcategory of $\CAlg(\Mod_\V(\PrL))$ spanned by rigid $\V$-algebras is closed under colimits.
\end{prop}
\begin{proof}
    By \Cref{obs:morrig}, any map between rigid $\V$-algebras is an internal left adjoint. Thus \Cref{prop:colimlocrig} already implies that any such colimit is locally rigid. 

    Furthermore, in the sifted case \cite[Corollary 1.32]{maindbl} implies that all the maps in the cocone diagram are internal left adjoints, and they send the unit to the unit, and hence the unit in the colimit is also $\V$-atomic. 

    In the finite coproduct case, the claim follows from the fact that a tensor product of $\V$-internal left adjoints remains a $ \V$-internal left adjoint. 
\end{proof}

In the locally rigid case, we have seen in \Cref{prop:exceptional} (although not proved yet) that the duality data of $\W$ can be described in terms of the multiplication map and an ``exceptional'' functor. It turns out that in the rigid case, this exceptional functor is simply the global sections functor, so that one can define the duality data ``ahead of time''. In particular, one may \emph{deduce} dualizability without assuming it. More precisely:
\begin{lm}\label{lm:rigimpldbl}
    Let $\V\to\W$ be a symmetric monoidal colimit-preserving functor, and assume $\one_\W$ is $\V$-atomic, and that the multiplication map $\mu:\W\otimes_\V\W\to\W$ is a $\W$-internal left adjoint. In this case, $\W$ is rigid over $\V$, namely, it is dualizable.
\end{lm}
\begin{proof}
See \cite[Proposition 2.17]{HSSS}. The point is that the following functors provide a (self-)duality datum for $\W$: $\V\xrightarrow{\eta} \W\xrightarrow{\mu^R}\W\otimes_\V\W$ and $\W\otimes_\V\W\xrightarrow{\mu}\W\xrightarrow{\eta^R}\V$. Checking the triangle identities will simply use the $\W$-linearity of $\mu^R$ and the $\V$-linearity of $\eta^R$ - the details are in the cited reference. 
\end{proof}
\begin{cor}\label{cor:rigselfdual}
    Let $\W$ be a rigid commutative $\V$-algebra. The functor $\W\to \Fun^L_\V(\W,\V)$ given as the mate of $\W\otimes_\V\W\xrightarrow{\mu}\W\xrightarrow{\eta^R}\V$ is an equivalence. In particular, any $\V$-linear functor $\W\to \V$ is of the form $\hom_\W(\one_\W,x\otimes -)$ for a uniquely determined $x$. 

    More generally, for any $\W$-module $\M$, postcomposition with $\hom_\W(\one_W,-)$ induces an equivalence $\Fun^L_\W(\M,\W)\simeq \Fun^L_\V(\M,\V)$. 
\end{cor}
In a similar vein to the criterion for dualizability from \cite[Theorem 2.55]{maindbl}, at least over $\Sp$, we obtain the following characterization of rigidity:
\begin{cor}\label{cor:rigtrcl}
    Let $\W\in\CAlg(\PrL_{\st})$. $\W$ is rigid over $\Sp$ if and only if its unit is compact, and there exists a set $S$ of trace-class maps such that  such that for any $x\in \W$, if $x$ is nonzero, then there exists some nonzero map to $x$ that factors through $S$, and such that either of the following two conditions hold:
    \begin{enumerate}
        \item Any map in $S$ factors as a composite of two maps in $S$; 
        \item All the objects appearing as sources of maps in $S$ are in the colimit-completion of $S$-telescopes in $\M$, that is, sequential colimits where all the transition maps are in $S$. 
    \end{enumerate}
\end{cor}
\begin{proof}
    That the unit is compact implies that trace-class maps are compact. This also holds in $\W\otimes\W$, and it follows from this that multiplication map $\W\otimes\W\to \W$ sends ``enough'' compact maps (in the sense of \cite[Addendum 2.42]{maindbl}) to compact maps - namely, it sends all pure tensors of trace-class maps to trace-class maps. 

    Thus, by \cite[Addendum 2.42]{maindbl}, the multiplication map is an internal left adjoint in $\PrL_\st$. If we know that $S$-telescopes generate $\W$ under colimits, \Cref{cor:nuclax} then directly implies that the right adjoint to the multiplication is $\W$-linear on either side, and thus $\W$ is rigid .

    We thus simply need to prove that condition 1. or 2. implies that $S$-telescopes generate $\W$ under colimits - this, however, is the exact same proof as in \cite[Theorem 2.55]{maindbl}.  
\end{proof}

In his proof that the category of localizing motives is rigidity, Efimov uses the following weaker variant: 
\begin{cor}\label{cor:Efimovrig}
    Let $\W\in\CAlg(\PrL_{\st})$. $\W$ is rigid over $\Sp$ if and only if its unit is compact, and there exists a set $S$ of trace-class maps such that  the $S$-telescopes generate $\W$ under colimits. 
\end{cor}

From all of these niceness properties, one finds that when working over a rigid base, most of the theory developed over $\Sp$ works essentially the same. In particular, the theory of localizing invariants (which can be defined over any base) behaves quite closely to the classical theory over rigid bases. For example, a consequence of \Cref{cor:rigadj} is that ordinary $K$-theory of the underlying (dualizable) $\V$-module \emph{is} the universal localizing invariant on $\V$-modules. Note that this is \emph{not} true over a general base. In fact, a question worth investigating could be:
\begin{ques}\label{ques:relK}
Let $\V\in\CAlg(\PrL_{\st})$. What is the analogue of $K$-theory on $\Dbl{\V}$ ?
\end{ques}
In the locally rigid, compactly generated case (and potentially some $\omega_1$-compactness assumption) one can give some answer to that question in terms of $\Ind(\V^\dbl)$. It seems likely that in the general locally rigid case, the answer will involve the \emph{rigidification} of $\V$. 

Before introducing rigidifications, let us describe a natural source of locally rigid categories:
\begin{prop}\label{prop:compriglocrig}
    Let $\V\to\W$ be a rigid map in $\CAlg(\PrL)$, and let $L:\W\to \W_L$ be a colimit-preserving, symmetric monoidal localization, and suppose $L$ admits a $\V$-linear \emph{left} adjoint $\iota$\footnote{A left adjoint will automatically be oplax $\V$-linear, we ask that this structure makes it strict $\V$-linear. In particular, $\iota$ is a $\V$-internal left adjoint.}. 

    In this case, $\W_L$ is locally rigid. 
\end{prop}
\begin{proof}
As $\W$ is rigid over $\V$, $\iota$ is also a $\W$-linear left adjoint (we explain this in the lemma following the proof). Since $\W$ is rigid over $\V$, it suffices to prove that $\W_L$ is locally rigid over $\W$ by \Cref{ex:complocrig}, and thus we are reduced without loss of generality to the case $\V=\W$. 

But now $\W_L$ is an idempotent $\W$-algebra and thus it suffices to prove that $\W_L$ is dualizable over $\W$. This is clear as the pair $(\iota,L)$ witnesses it as a $\W$-linear retract of $\W$. 
\end{proof}
\begin{lm}\label{lm:leftadjlinear}
    Let $\W$ be a commutative locally rigid $\V$-algebra, and $f:\M\to \N$ a $\W$-linear functor. Suppose $f$ admits a $\V$-linear left adjoint $f^L$. In this case, the canonical oplax $\W$-linear structure on $f^L$ is strict $\W$-linear. 
\end{lm}
\begin{proof}
    $\W$-linearity of $f^L$ amounts to the horizontal left adjointability of this square: 
    \[\begin{tikzcd}
	{\W\otimes_\V\M} & {\W\otimes_\V\N} \\
	\M & \N
	\arrow[from=1-1, to=2-1]
	\arrow[from=1-1, to=1-2]
	\arrow[from=2-1, to=2-2]
	\arrow[from=1-2, to=2-2]
\end{tikzcd}\]
Because the vertical functors have right adjoints, this amounts to the vertical right adjointability of the same square. But in turn, this follows from the fact that the vertical right adjoints are given by tensoring $\M$ and $\N$ with $\W\to \W\otimes_\V\W$. 
\end{proof}

We claim in fact that every locally rigid $\Sp$-algebra arises this way, i.e. as a completion of a rigid $\Sp$-algebra. For this, we introduce ``rigidification'', which is a right adjoint to the inclusion of rigid algebras into all $\CAlg(\PrL_{\st})$. This construction will be a special case of the ``S-construction'' from \cite[Section 4.2]{maindbl}. 

We will also see later that this ``rigidification'' exists in general, but in the case of $\Sp$ it is a bit more hands on and explicit, and so we single out this special case, for concreteness and convenience.
\subsection{Rigidification}\label{section:rigidification}
Our goal in this section is to introduce \emph{rigidifications}. Beyond the fact that this is a very natural notion, we are mainly after the following fact: every locally rigid $\V$-algebra is the ``completion'' of a rigid $\V$-algebra, namely its rigidification. This is convenient, as it allows us to reduce many properties of locally rigid categories to analogous ones for rigid categories - particularly, it will allow us to give an alternative proof of \Cref{prop:exceptional} and to get a more precise understanding of the exceptional functor $\Gamma_!$ appearing there.

Before we state the main results of this section, let us give a definition and basic properties thereof.

\begin{defn}\label{defn:rigidification}
    Let $\overline\W, \W\in\CAlg(\Mod_\V)$ and $f:\overline \W\to \W$ a map between them. We say $f$ witnesses $\overline \W$ as a rigidification of $\W$ if $\overline \W$ is a rigid $\V$-algebra and for any rigid commutative $\V$-algebra $\W_0$, postcomposition with $f$ induces an equivalence: $$\Fun^{L,\otimes}_\V(\W_0,\overline \W)\xrightarrow{\simeq}\Fun^{L,\otimes}_\V(\W_0,\W)$$

\end{defn}
This clearly characterizes rigidifications canonically if they exist, and so:
\begin{nota}
    If it exists, we let $\Rig_\V(\W)$ denote the rigidification of $\W$.
\end{nota}
The following is an obvious observation, but worth making nonetheless:
\begin{obs}
    Let $\V_0\to\V_1\to\W$ be maps in $\CAlg(\PrL)$, and suppose $\V_1$ is $\V_0$-rigid. In this case, $\W$ admits a $\V_0$-rigidification if and only if it admits a $\V_1$-rigidification, and in case they exist, the canonical map between them is an equivalence. 
\end{obs}
\begin{proof}
    This follows by considering the respective universal properties, and using \Cref{ex:complocrig} and \Cref{cor:2outof3locrig}. 
\end{proof}
We will see soon that rigidifications always exist\footnote{It is a not-so-difficult consequence of our presentability theorem, \Cref{thm:rigprl}, but we will also construct them by hand.} but we may also reason about them prior to knowing that they exist. Before we do so, let us state our goal:

\begin{thm}\label{thm:locrigrigleft}
     Suppose $\W$ is a commutative locally rigid $\V$-algebra. In this case, $\W$ admits a rigidification, and the canonical map $\Rig_\V(\W)\to \W$ admits a fully faithful left adjoint.
\end{thm}
Recall that by \Cref{prop:compriglocrig}, if this canonical map admits a fully faithful left adjoint, then $\W$ is locally rigid. 

In this sense, we may see $\W$ as a completion of $\Rig_\V(\W)$, where the fully faithful left adjoint embeds it as a category of ``locally nilpotent'' objects, and the (automatically fully faithful) right adjoint embeds it as a category of ``complete'' objects.

We will prove this theorem by ``describing'' the rigidification in a similar fashion to the dualizable core from \Cref{defn:dblcore}. Before doing so, let us explain the main consequence, that is, the promised alternative proof of \Cref{prop:exceptional}:
\begin{prop}\label{prop:exceptional2}
Let $\W$ be a commutative $\V$-alebra, and suppose there exists a rigid $\V$-algebra $\overline\W$ with a morphism $p:\overline\W\to \W$ admitting a fully faithful $\V$-linear left adjoint $p^L$. Let $\Gamma_! : \W\xrightarrow{p^L}\overline\W\xrightarrow{\hom(\one_{\overline\W},-)}\V$. The following is a self-duality datum for $\W$: 
$$\V\xrightarrow{\eta}\W\xrightarrow{\mu^R}\W\otimes_\V\W, \qquad \W\otimes_\V\W\xrightarrow{\mu}\W\xrightarrow{\Gamma_!}\V $$
\end{prop}
\begin{proof}
By \Cref{lm:leftadjlinear}, $p^L$ is $\overline \W$-linear, so that the canonical map $p^L(w\otimes p(\overline w))\to p^L(w)\otimes \overline w$ is an equivalence for all $w\in\W,\overline w\in\overline \W$. 

Plugging in $\overline w= p^L(w_1)$ and using that the unit $\id_\W\to pp^L$ is an equivalence, we find that the canonical map $p^L(w\otimes w_1)\to p^L(w)\otimes p^L(w_1)$ is an equivalence. In other words, the following diagram is vertically left adjointable: 
\[\begin{tikzcd}
	{\overline \W\otimes\overline\W} & {\overline \W} \\
	\W\otimes\W & \W
	\arrow[from=1-1, to=2-1]
	\arrow[from=1-1, to=1-2]
	\arrow[from=1-2, to=2-2]
	\arrow[from=2-1, to=2-2]
\end{tikzcd}\]
Passing to right adjoints everywhere, we find that it is horizontally right adjointable, namely, the following square (canonically) commutes: 
\[\begin{tikzcd}
	{\overline \W} & \overline\W\otimes\overline\W \\
	\W & \W\otimes\W
	\arrow["{\mu^R_\W}"', from=2-1, to=2-2]
	\arrow["{p\otimes p}", from=1-2, to=2-2]
	\arrow["{\mu^R_{\overline\W}}", from=1-1, to=1-2]
	\arrow["p"', from=1-1, to=2-1]
\end{tikzcd}\]

Altogether, this simply proves that the pre-duality datum we provided is exactly the one induced by the one of $\overline\W$ from \Cref{lm:rigimpldbl} using the retraction $(p^L,p)$, and hence is itself a duality datum. 
\end{proof}

\begin{cor}\label{cor:exceptional=at}
  Let $\W$ be a commutative $\V$-alebra, and suppose there exists a rigid $\V$-algebra $\overline\W$ with a morphism $p:\overline\W\to \W$ admitting a fully faithful $\V$-linear left adjoint $p^L$. 

  In this case, $\W$ is locally rigid by \Cref{prop:compriglocrig} and the exceptional functor $\Gamma_! : \W\to \V$ is given by $\at_\W(\one_\W,-)$. 
\end{cor}
\begin{proof}
    Using the previous proposition, we see that $\Gamma_! = \hom(\one_{\overline\W}, p^L(-))$ where $p^L:\W\to \overline\W$ is left adjoint to $p$. 

    In particular, $p^L$ is a fully faithful internal left adjoint, so that by \Cref{cor:iLffat} we have an equivalence $\at_\W(x,y)\simeq \at_{\overline\W}(p^L(x),p^L(y))$. Since $\overline\W$ is rigid, the latter is $$\hom_{\overline\W}(\one_{\overline \W}, \hom^{\overline \W}_{\overline \W}(p^L(x),\one_{\overline\W})\otimes p^L(y))$$

    Now, $p$ makes $\W$ into a $\overline\W$-module, and since $p^L$ is $\V$-linear and $\overline\W$ is rigid over $\W$, $p^L$ is $\overline\W$-linear by \Cref{lm:leftadjlinear}. It follows then that  $\hom^{\overline \W}_{\overline \W}(p^L(x),\one_{\overline\W})\simeq \hom^{\overline\W}_\W(x,p(\one_{\overline \W})) = \hom^{\overline \W}_\W(x,\one_\W) = p^R \hom_\W^\W(x,\one_\W)$ and so with $x=\one_\W$, we find $\hom_{\overline\W}^{\overline\W}(p^L(\one_W),\one_{\overline\W}) = p^R(\one_\W)$. 

    Using $\overline\W$-linearity of $p^L$ , we find $$\at_\W(\one_\W,y)= \hom_{\overline\W}(\one_{\overline\W}, p^R(\one_\W)\otimes p^L(y))\simeq \hom_{\overline\W}(\one_{\overline\W}, p^L(pp^R(\one_\W)\otimes y))\simeq \hom_{\overline\W}(\one_{\overline \W},p^L(y))$$ 
    where we used that $pp^R\simeq \id_\W$ by fully faithfulness. 
\end{proof}
\begin{ex}
    In the case of $\V=\Sp$, this says that for locally rigid $\W$, the exceptional functor $\Gamma_! : \W\to \Sp$ is given by the spectrum of compact maps from the unit $\one_\W$. For example this gives a completely categorical construction of the compactly-supported-global-sections functor on sheaves on a locally compact Hausdorff space. 
\end{ex}
The final thing we do before describing rigidifications in order to prove \Cref{thm:locrigrigleft} is to prove that they exist. 

First, we note the following easy lemma:
\begin{lm}\label{lm:rig=rigP}
    Let $\kappa$ be a good cardinal, and let $\W\in\CAlg(\Mod_\V(\PrL_\kappa))$. In this case, for any rigid $\V$-algebra $\W_0$, the symmetric monoidal functor $p:\PP_\V(\W^\kappa)\to\W$ induces an equivalence $$\Fun^{L,\otimes}_\V(\W_0,\PP_\V(\W^\kappa))\to \Fun^{L,\otimes}_\V(\W_0,\W)$$

    In particular, $\Rig_\V(\W)$ exists if and only if $\Rig_\V(\PP_\V(\W^\kappa))$ exists and if they do, the canonical map between them is an equivalence. 
\end{lm}
\begin{proof}
    The ``in particular'' part is a formal consequence of the first half of the claim, so we focus on that. 

    For this, we simply note that since the unit of $\W$ is $\kappa$-compact, any symmetric monoidal functor $\W_0\to \W$ has a $\kappa$-filtered colimit-preserving right adjoint by \Cref{rmk:kappacpctunit}. Thus we may replace the right hand side by $\Fun^{L,\otimes,\kappa}_\V(\W_0,\W)$. 

    Similarly, the unit of $\PP_\V(\W^\kappa)$ is atomic, and so by \Cref{obs:morrig}, we may replace the left hand side by $\Fun^{iL,\otimes}_\V(\W_0,\PP_\V(\W^\kappa))$. 

    The result now follows from \cite[Theorem 1.63, Remark 1.64]{maindbl}.
\end{proof}
We now essentially need to prove that there exist rigidifications for $\V$-algebras with atomic units (we will also use this when we describe rigidifications). For this, we will use the following lemma:
\begin{lm}\label{lm:imrig}
    Let $\W_0,\W_1\in\CAlg(\Mod_\V(\PrL))$, and let $f:\W_0\to \W_1$ be a internally left adjoint $\V$-linear symmetric monoidal functor. The full subcategory of $\W_1$ generated under colimits and $\V$-tensors by the image of $f$ is closed under tensor products in $\W_1$, and is thus canonically symmetric monoidal. 

    If $\W_0$ is locally rigid over $\V$, then so is this image. In particular if the unit of $\W_1$ is atomic, the image is rigid. 
\end{lm}
\begin{proof}
As $\W_0$ is locally rigid, all the hypotheses and conclusions are unchanged if work over $\W_0$ instead, i.e. we may without loss of generality assume that $\V=\W_0$, which is what we do from now on. 

    Let us call $\mathcal I$ the image of $f$ which, by \cite[Lemma 2.61]{maindbl} is dualizable. The corestriction of $f$, $\tilde f: \V\to \mathcal I$, is an internally left adjoint symmetric monoidal functor which generates $\mathcal I$ under colimits. In other words, up to replacing $\W_1$ by $\mathcal I$, we are essentially in the situation where $\V\to \W$ is a $\V$-internal left adjoint that generates $\W$ under colimits. 

   Now, using \cite[Lemma 1.30]{maindbl}, we find that $\W\otimes_\V\W\to \W$ is a $\V$-internal left adjoint, because the composition $\V\to\W\otimes_\V\W\to \W$ is. In particular, to prove that it is also a $\W$-internal left adjoint, it suffices now to prove that the right adjoint is linear with respect to a class of generators of $\W$ under $\V$-tensors and colimits - for example, $\one_\W$. But it clearly is linear with respect to $\one_\W$, and so we are done. 
\end{proof}
\begin{cor}
    Any $\W\in\CAlg(\Mod_\V(\PrL))$ admits a rigidification. 
\end{cor}
\begin{proof}
    Fix $\kappa$ a good cardinal such that $\W\in\CAlg(\Mod_\V(\PrL_\kappa))$. 

    By \Cref{lm:rig=rigP}, it suffices to prove that $\PP_\V(\W^\kappa)$ has a rigidification. 

    Consider the poset $P$ of full-sub-$\V$-algebras of $\PP_\V(\W^\kappa)$ that are rigid. By \Cref{obs:morrig}, for $\W_0\in P$, the inclusion $\W_0\to \PP_\V(\W^\kappa)$ is an internal left adjoint and in particular preserves $\kappa$-compacts, and $\W_0$ is $\kappa$-compactly generated - thus $P$ is small : it is a subposet of the poset of subcategories of $\PP_\V(\W^\kappa)^\kappa$. 

    Furthermore, one can deduce from the previous lemma that $\PP_\V(\W^\kappa)$ has a rigidification if and only if this poset has a maximal element. 

    Let $\W_0 := \colim_{C\in P}C$, the colimit being taken in $\CAlg(\Mod_\V(\PrL))$. By \Cref{prop:colimrig}, $\W_0$ is rigid, and it certainly has a map to $\PP_\V(\W^\kappa)$. Its image in the sense of \Cref{lm:imrig} is rigid (by that lemma) and contains $C$ for every $C\in P$, thus proving that it has a largest element, thereby concluding the proof.  
\end{proof}
Note that the proof in particular shows: 
 \begin{prop}\label{prop:rigdesc}
Let $\kappa$ be a good cardinal, and let $\W\in\CAlg(\Mod_\V(\PrL_\kappa))$. 
    There is a canonical fully faithful symmetric monoidal embedding $\Rig_\V(\W)\to \PP_\V(\W^\kappa)$. 
\end{prop}
In fact, more generally, it is easy to adapt the proof to show:
    \begin{cor}\label{cor:cpctunitrigff}
        If $\one_\W$ is $\V$-atomic, then the functor $\Rig_\V(\W)\to \W$ is fully faithful. 
    \end{cor}
    The main theorem of this section is a study of the somewhat ``opposite'' situation where the unit is very much not atomic, but the category is locally rigid. 

Now that we have established this, we get, assuming the main theorem:
\begin{cor}
    Let $\W$ be a locally rigid $\V$-algebra and let $\Gamma_! := \at_\W(\one_\W,-): \W\to \V$. $\Gamma_!$ participates in a duality datum for $\W$ with evaluation:
    \[\begin{tikzcd}
	{\W\otimes_\V\W} & \W & \V
	\arrow["\mu", from=1-1, to=1-2]
	\arrow["{\Gamma_!}", from=1-2, to=1-3]
\end{tikzcd}\]
and coevaluation:  \[\begin{tikzcd}
	\V & \W & {\W\otimes_\V\W}
	\arrow["\eta", from=1-1, to=1-2]
	\arrow["{\mu^R}", from=1-2, to=1-3]
\end{tikzcd}\]
\end{cor}

With all this said and done, we may embark on a slightly more precise description of rigidifications.

\subsubsection{Over $\Sp$}
In the same vein as for \cite[Section 4]{maindbl}, while the existence of rigidification is good to know, for certain purposes it can be useful to have a more concrete description (for example, to prove \Cref{thm:locrigrigleft}). Our goal here is to draw inspiration from \Cref{cor:trclexQ} to construct the rigidification by hand in terms of trace-class maps. 

We recall the following result:
\begin{cor}[{\cite[Corollary 2.57]{maindbl}}]\label{cor:cpctexhaustQ}
    Let $\M$ be a compactly assembled category, and $x$ a compactly exhaustible object therein. There exists a $\mathbb Q_{\geq 0}$-indexed diagram with colimit $x$ such that each of the transition maps is compact. 
\end{cor}
Recall also that the idea here of using $\mathbb Q_{\geq 0}$ instead of the cofinal poset $\mathbb N$ was to encode, \emph{in the diagram}, the fact that the colimit could be written as a colimit, not only along maps of a certain type, but maps which could be split into maps of the same type, which could themselves be split into maps of the same type, etc. For our purposes, the relevant corollary is: 
\begin{cor}\label{cor:trclexQ}
    Let $\V\in\CAlg(\PrL_{\st})$ be locally rigid, and let $x\in\V$ be compactly exhaustible. There exists a $\mathbb Q_{\geq 0}$-indexed diagram with colimit $x$ and such that each of the transition maps is trace-class. 
\end{cor}
\begin{proof}
    This follows from \Cref{cor:cpctexhaustQ}: in a locally rigid category, every compact map is trace-class by \Cref{prop:locrigcharac}. 
\end{proof}
We learned the following construction and results from Dustin Clausen:
\begin{cons}
    Let $\V\in\CAlg(\PrL_{\st,\kappa})$. The symmetric monoidal structure on $\V$ induces one on $\Ind(\V^\kappa)$, for which the colimit functor $\Ind(\V^\kappa)\to\V$ is symmetric monoidal. 

    Consider the full subcategory $\Rig_\kappa(\V)$ of $\Ind(\V^\kappa)$ generated under colimits by $\mathbb Q_{\geq 0}$-telecopes along trace-class maps in $y(\V^\kappa)$.  
\end{cons}
\begin{obs}
 As trace-class maps in $\V$ are closed under tensor products and $\mathbb Q_{\geq 0}\to\mathbb Q_{\geq 0}\times\mathbb Q_{\geq 0}$ is cofinal, $\Rig_\kappa(\V)\subset \Ind(\V^\kappa)$ is closed under tensor products. 
\end{obs}
\begin{thm}\label{thm:rigidification}
    Let $\V\in\CAlg(\PrL_{\st,\kappa})$ for some uncountable cardinal $\kappa$. 
    \begin{itemize}
        \item $\Rig_\kappa(\V)$ is a rigid $\Sp$-algebra. 
        \item The symmetric monoidal functor $\Rig_\kappa(\V)\to\V$ obtained by restricting the functor $\Ind(\V^\kappa)\to \V$ to $\Rig_\kappa(\V)$ has the following universal property: for any rigid commutative $\Sp$-algebra $\M$, the canonical functor is an equivalence : 
    $$\Fun^{L,\otimes}(\M,\Rig_\kappa(\V))\to \Fun^{L,\otimes}(\M,\V)$$
    \end{itemize}
    In particular, $\Rig_\kappa(\V)\simeq \Rig_\Sp(\V)$ and it does not depend on $\kappa$ so long as $\V\in\CAlg(\PrL_{\st,\kappa})$. 
\end{thm}
\begin{proof}
(Rigidity) We first prove that $\Rig_\kappa(\V)$ is rigid. Our goal is to apply \Cref{cor:rigtrcl}. First, the unit is $y(\one_\V)$ and hence compact.
    
    Now, the set of maps $S$ we consider is the set of maps $f:x\to z$ such that there exist presentations $x\simeq\colim_{\mathbb Q_{\geq 0}}y(x_r), z\simeq\colim_{\mathbb Q_{\geq 0}}y(z_r)$ along trace-class maps, where each $x_r,z_r\in \V^\kappa$, and such that $f$ factors through some $z_r, r\in\mathbb Q_{\geq 0}$. 

    Firstly, any such map factors as a composite of two such maps: indeed, say $f$ factors through $z_r$. We then observe that $f$ factors as $x\to \colim_{s<r+1}y(z_r) \to z$, where both maps are now clearly of the desired form, using that $\mathbb Q_{\geq 0}^{<r+1}\cong \mathbb Q_{\geq 0}$ as posets. 

    Second, let $D\in\Rig_\kappa(\V)$ be nonzero. We aim to prove that $D$ receives some nonzero map $x\to z\to D$, where $x\to z$ is as above. We first note that each $\mathbb Q_{\geq 0}$-telescope along trace-class maps in $y(\V^\kappa)$ is $\omega_1$-compact: indeed, they are countable colimits of compacts in $\Ind(\V^\kappa)$, so they are $\omega_1$-compact therein, and the inclusion $\Rig_\kappa(\V)\subset \Ind(\V^\kappa)$ is fully faithful and colimit-preserving. 

    In particular, it follows that $\Rig_\kappa(\V)$ is $\omega_1$-compactly generated by countable colimits of $\mathbb Q_{\geq 0}$-telescopes along trace class maps. It follows that if $D\neq 0$, there exists some such telescope $x$ with $\map(x,D)\neq 0$. Up to applying $\Sigma^k$ for some $k\in\mathbb Z$ (and observing that trace-class maps are closed under this operation), we find some such telescope $x\simeq\colim_{\mathbb Q_{\geq 0}}y(x_r)$ with a nonzero map $x\to D$.

    We now observe that $\pi_0\map(x,D)$ fits in a Milnor short exact sequence. For it to be nonzero, we need either $\lim^1\pi_1\map(x_n,D)$ or $\lim\pi_0\map(x_n,D)$ to be nonzero. Up to taking $\Sigma x$, we may assume that the pro-object $\pi_0\map(x_n,D)$ is nonzero, so there must be some $n$ with a nonzero map $x_{n+1}\to D$ which restricts to a nonzero map $x_n\to x_{n+1}\to D$. Finding two rational numbers $n<r<s<n+1$, we may now factor the above as $$x_n\to \colim_{t\in \mathbb Q_{\geq n}^{<r}}y(x_t)\to \colim_{t\in\mathbb Q_{\geq n}^{<s}}y(x_t)\to x_{n+1}\to D$$ which provides the desired maps. 

    It follows from \Cref{cor:rigtrcl} that $\Rig_\kappa(\V)$ is rigid. 
\newline 

(Universal property)
    We move on to the second part of the claim. Consider the full subcategory of $\Fun^L(\M,\V)$ spanned by those functors that \emph{admit} symmetric monoidal structures\footnote{This is a rather useless object. We recommend the reader not consider it in their free time.}, denoted $\Fun^L_{(\otimes)}(\M,\V)$. 

    By \cite[Theorem 1.63, Remark 1.64]{maindbl}, we have an equivalence $$\Map^{L,\otimes}_\kappa(\M, \V)\simeq \Map^{iL,\otimes}(\M,\Ind(\V^\kappa))$$ 

    We make two claims: firstly, the inclusion $\Map^{L,\otimes}_\kappa(\M, \V)\subset \Map^{L,\otimes}(\M,\V)$ is an equivalence. This is because a symmetric monoidal functor $\M\to \V$ sends trace-class maps to trace-class maps, hence, using the rigidity of $\M$ and the fact that $\one_\V$ is $\kappa$-compact, it must send compact maps to $\kappa$-compact maps, which proves the claim by \cite[Corollary 2.40]{maindbl}.

    Secondly, the inclusions $$\Map^{L,\otimes}(\M,\Rig_\kappa(\V)) \supset \Map^{iL,\otimes}(\M,\Rig_\kappa(\V))\subset \Map^{iL,\otimes}(\M,\Ind(\V^\kappa))\subset \Map^{L,\otimes}(\M,\Ind(\V^\kappa))$$ are all equivalences. 
The first and last one are equivalences because $\M$ is rigid and the unit of both $\Rig_\kappa(\V), \Ind(\V^\kappa)$ are compact, so that \Cref{obs:morrig} applies. For the middle one, we note that by \Cref{cor:cpctexhaustQ} and \Cref{prop:locrigcharac}, $\M$ is generated under colimits by $\mathbb Q_{\geq 0}$-indexed diagrams along trace-class maps, so it suffices to show that these land in $\Rig_\kappa(\V)$ for any symmetric monoidal internal left adjoint $f:\M\to \Ind(\V^\kappa)$.

We start with the following observation: let $x\to z$ be a trace-class map in $\Ind(\V^\kappa)$. Compactness of the unit guarantees that it factors as $x\to y(z_0)\to z$ for some trace-class map $x\to y(z_0), z_0\in \V^\kappa$. 

Now let $m\in \M$ be compactly exhaustible, so we write $m\simeq\colim_\mathbb N m_n$ with trace-class transition maps by \Cref{prop:locrigcharac}. Each of the $f(m_n)\to f(m_{n+1})$ factors through a $y(z_n), z_n \in\V^\kappa$ by the above observation, where $f(m_n)\to y(z_n)$ is trace-class. Thus $f(m)\simeq\colim_\mathbb N y(z_n)$ with each $z_n\to z_{n+1}$ trace-class. 

The proper proof is a slight poset-theoretic elaboration on this idea. We use facts about posets that we justify after the proof. First, write $m\simeq \colim_{\mathbb Q_{\geq 0}}m_r$ with each transition map trace-class. Fix a cofinal embedding of $\mathbb Q_{\geq 0}\times[1]\to \mathbb Q_{\geq 0}$ where the source has the lexicographic ordering. This allows us to write $m\simeq \colim_{(r,i)\in \mathbb Q_{\geq 0}\times[1]}m_{(r,i)}$. Now for each $r\in\mathbb Q_{\geq 0}$, we have a trace-class map $f(m_{(r,0)})\to f(m_{(r,1)})$ which therefore factors through  some $y(z_r)$. 

We will prove below that the commutative square below is cocartesian : 
\[\begin{tikzcd}
	{\mathbb Q_{\geq 0}^\delta\times[1]} & {\mathbb Q_{\geq 0}\times [1]} \\
	{\mathbb Q_{\geq 0}^\delta\times [2]} & {\mathbb Q_{\geq 0}\times [2]}
	\arrow[from=2-1, to=2-2]
	\arrow[from=1-1, to=2-1]
	\arrow[from=1-2, to=2-2]
	\arrow[from=1-1, to=1-2]
\end{tikzcd}\]
where the inclusion $[1]\to [2]$ is the long edge, and where $\mathbb Q_{\geq 0}^\delta$ denotes the set $\mathbb Q_{\geq 0}$ viewed as a discrete category. 

Thus we can use the above $r$-wise factorizations to get an extension of our diagram $(r,i)\mapsto f(m_{(r,i)})$ on $\mathbb Q_{\geq 0}\times [1]$ to a diagram $(r,i)\mapsto \tilde m_{(r,i)}$ on $\mathbb Q_{\geq 0}\times [2]$ where for each nonnegative rational $r$, the corresponding $[2]$-indexed diagram is $f(m_{(r,0)})\to y(z_r)\to f(m_{(r,1)})$. Now the inclusion of the vertex $1$ induces a cofinal map $\mathbb Q_{\geq 0}\to \mathbb Q_{\geq 0}\times [2]$, so that $$f(m)\simeq\colim_{\mathbb Q_{\geq 0}\times [1]}f(m_{(r,i)}) \simeq \colim_{\mathbb Q_{\geq 0}\times [2]} \tilde m_{(r,i)}\simeq \colim_{\mathbb Q_{\geq 0}}y(z_r)$$ 

Thus in total, $f(m)$ is in $\Rig_\kappa(\V)$. 

\end{proof}
We used two facts about posets in the above proof. The first is :
\begin{lm}
    There exists a cofinal functor $\mathbb Q_{\geq 0}\times [1]\to \mathbb Q_{\geq 0}$. 
\end{lm}
\begin{proof}
    $\mathbb Q_{\geq 0}\times\mathbb Q_{\geq 0}$ is in fact isomorphic to $\mathbb Q_{\geq 0}$ by Cantor's theorem. Thus we get an embedding $\mathbb Q_{\geq 0}\times [1]\to \mathbb Q_{\geq 0}$. This embedding is clearly cofinal. 
\end{proof}
And now the second one, we phrase more generally - we have applied it to $P=\mathbb Q_{\geq 0}, Q_0=~[1], Q_1=~[2]$: 
\begin{lm}
    Let $P,Q_0,Q_1$ be totally ordered posets, with an embedding $Q_0\to Q_1$. Suppose that $Q_0\to Q_1$ is both cofinal and initial. In this case, the following square is cocartesian: 
    \[\begin{tikzcd}
	{P^\delta\times Q_0} & {P\times Q_0} \\
	{P^\delta\times Q_1} & {P\times Q_1}
	\arrow[from=2-1, to=2-2]
	\arrow[from=1-1, to=2-1]
	\arrow[from=1-2, to=2-2]
	\arrow[from=1-1, to=1-2]
\end{tikzcd}\]
where all the products are meant in the lexicographic sense. 
\end{lm}
\begin{proof}
As $P$ is linearly ordered, it is a filtered colimit of posets of the form $[n]$, and so we may assume without loss of generality that it is of this form. Since the pushouts $[n]\coprod_{n=0}[1]\simeq~[n+1]$ are preserved by $(-)^\simeq$ (noting that $P^\simeq \simeq P^\delta$), we may also assume $P=[1]$. 

Similarly, $Q_1$ is a filtered colimit of its finite subposets. In fact, it is a filtered colimit of its finite subposets $F$ such that $F\cap Q_0\to F$ is both cofinal and initial. So we may assume without loss of generality that $Q_1=[n]$ and $Q_0$ is some initial and cofinal subposet thereof. In particular, $Q_0$ contains $0\leq n$, and so using a pushout pasting lemma, we may assume $Q_0= \{0\leq n\}\subset Q_1= [n]$. 

Now we note that in this special case, $\{0,1\}\times \{0\leq n\}\to \{0\leq 1\}\times \{0\leq n\}$ is pushed out from $\{(0,n),(1,0)\}\to \{(0,n)\leq (1,0)\}$ for general spine reasons, and so is $\{0,1\}\times [n]\to \{0\leq 1\}\times [n]$, so we are done. 
\end{proof}
\begin{lm}\label{lm:rigfprops}
    Let $\V\in\CAlg(\PrL_{\st})$. We have:
    \begin{enumerate}
        \item If $\one_\V$ is compact then the canonical map $\Rig(\V)\to \V$ is fully faithful. 
        \item $\V$ is locally rigid if and only if the map $\Rig(\V)\to \V$ admits a fully faithful left adjoint. 
    \end{enumerate}
\end{lm}
\begin{rmk}
    The converse to the first statement is not clear to the author, and we are in fact enclined to believe it is wrong in full generality. 
\end{rmk}
\begin{proof}
   1. was proved in \Cref{cor:cpctunitrigff}.
 
  2. If $\V$ is locally rigid, then compact maps are trace-class, and it follows that $\hat y :\V\to \Ind(\V^\kappa)$ lands in $\Rig(\V)$ by \Cref{cor:cpctexhaustQ} : indeed, compactly exhaustible objects generate $\V$ under colimits, and if $x=\colim_{\mathbb Q_{\geq 0}}x_r$ with compact transition maps, then $\hat y(x) = \colim_{\mathbb Q_{\geq 0}} y(x_r)$, and each $y(x_r)\to y(x_s)$ is trace-class, as it was so in $\V$. In particular, $\hat y$ the desired fully faithful left adjoint. 

   Conversely, if there is such a left adjoint, then \Cref{prop:compriglocrig} kicks in, showing that $\V$ is locally rigid. 
\end{proof}
In particular, point 2. indicates that locally rigid categories are always (canonically) ``completions'' of rigid categories.
\begin{cons}
    Let $\V\in\CAlg(\PrL_{\st})$. By \Cref{ex:atgenrig}, $\Ind(\V^\dbl)$ is rigid and hence the canonical symmetric monoidal colimit-preserving functor $\Ind(\V^\dbl)\to \V$ factors through $\Rig(\V)$. Beware that in general, the induced map $\Ind(\V^\dbl)\to\Rig(\V)$, while always fully faithful, is not in general an equivalence. 
\end{cons}
\begin{ex}
    Let $\V\in\CAlg(\PrL_{\st})$ be such that any trace-class map factors through a dualizable. In this case, the canonical map $\Ind(\V^\dbl)\to \Rig(\V)$ is an equivalence. Indeed, for any sequential colimit along trace-class maps $\colim_\NN y(v_n)$ in $\Ind(\V)$, we can replace it by a colimit of dualizables $y(w_n)$ through which $v_n\to v_{n+1}$ factors. 
\end{ex}
Another example, related to the description of limits of dualizable categories in \cite[Section 4]{maindbl}, is the following: 
\begin{ex}\label{ex:prodrig}
    Let $\V_i$ be a family of rigid compactly generated $\Sp$-algebras, and consider their product $\prod_i \V_i$. In this case, the rigidification is particularly simple: we have $\Rig(\prod_i\V_i)\simeq\Ind(\prod_i \V_i^\omega)$. 

    Indeed consider, in $\prod_i \V_i$, a trace-class map $(v_i)\to (w_i)$. Each $v_i\to w_i$ is trace-class in $\V_i$, and hence compact ($\one_{\V_i}$ is compact) and hence factors through a compact ($\V_i$ is compactly generated) and hence through a dualizable ($\V_i$ is rigid). As dualizables in $\prod_i\V_i$ are detected pointwise, this implies that $(v_i)\to (w_i)$ factors through a dualizable object. Thus the result follows from the previous example.  
\end{ex}
More generally, as an application of this description of rigidification together with \Cref{lm:2foldloccpct}, we obtain: 
\begin{cor}\label{cor:limrig}
    The category $\CAlg\rig$ has all limits and the forgetful functor $\CAlg\rig\to~\Prdbl$ preserves them.
\end{cor}
We thank Sasha Efimov for pointing out an error in an earlier proof of this corollary and indicating how to fix it. 

Note that this forgetful functor exists by \Cref{obs:morrig}. Before proving this, we will need a preliminary which is about trace-class maps in lax pullbacks. First, we analyze internal homs in lax pullbacks, by analyzing a simpler case. 
\newcommand{\laxfib}{\overset{\rightarrow}{\fib}}
\begin{lm}\label{lm:hominlaxfib}
    Let $f:C\to D$ be a left adjoint symmetric monoidal functor, and consider the category $\laxfib(f):=(C\times D)\times_{(D\times D)}D^{\Delta^1}$ of tuples $(c, d, f(c)\to d)$ where $c\in C,d\in D$. 

    If $C,D$ have internal homs and $C$ admits pullbacks, then $\laxfib(f)$ admits internal homs, and they are given by the following formula: $$\hom((c_0,d_0), (c_1,d_1)) =(\hom(c_0,c_1)\times_{f^R\hom(f(c_0),d_1)} f^R\hom(d_0,d_1), \hom(d_0,d_1))$$ and the map 
    $f(\hom(c_0,c_1)\times_{f^R\hom(f(c_0),d_1)} f^R\hom(d_0,d_1))\to \hom(d_0,d_1))$ is given by the projection onto a term in the pullback followed by the counit map of the $f\dashv f^R$-adjunction. 

    In particular, the forgetful functor $\laxfib(f)\to D$ preserves internal homs. 
\end{lm}
\begin{proof}
This is a ``brute force'' calculation. Fix $(c,d)\in\laxfib(f)$ and we compute $$\Map_{\laxfib(f)}((c,d)\otimes (c_0,d_0), (c_1,d_1)) \simeq \Map_C(c\otimes c_0, c_1)\times_{\Map_D(f(c)\otimes f(c_0), d_1)}\Map_D(d\otimes d_0, d_1)$$ 
$$\simeq \Map_C(c,\hom(c_0,c_1))\times_{\Map_D(f(c),\hom(f(c_0),d_1)} \Map_D(d,\hom(d_0,d_1)) $$

$$\simeq \Map_C(c,\hom(c_0,c_1))\times_{\Map_D(f(c),\hom(f(c_0),d_1)}\map_D(f(c),\hom(d_0,d_1))\times_{\Map_D(f(c),\hom(d_0,d_1))}\Map_D(d,\hom(d_0,d_1)) $$

$$\simeq \Map_C(c,\hom(c_0,c_1)\times_{f^R\hom(f(c_0),d_1)} f^R\hom(d_0,d_1))\times_{\Map_D(f(c),\hom(d_0,d_1))}\Map_D(d,\hom(d_0,d_1))$$

$$\simeq \Map_C(c,H_C)\times_{\Map_D(f(c),H_D)}\Map_D(d,H_D)$$ 
where $H=(H_C,H_D)$ is the object described in the statement of the lemma. Unravelling the string of equivalences, one finds that the relevant structure maps are the ones we claimed, and therefore this proves the lemma. 
\end{proof}
We now discuss trace-class maps in pullbacks along $\hom$-preserving symmetric monoidal functors: 
\begin{cor}\label{cor:comppwtrcl}
Let $C_0\xrightarrow{p_0}C_{01}\xleftarrow{p_1}C_1$ be a span of symmetric monoidal functors between symmetric monoidal categories, and let $C$ be the pullback. Assume that $C_0,C_1$ have internal homs and that they are preserved by $p_0$ and $p_1$. Finally, let $f:x\to y, g:y\to z$ be two maps in $C$ whose projection to each $C_i$ is trace-class. 

Their composite $gf: x\to z$ is trace-class in $C$. 
\end{cor}
\begin{proof}
We first point out that internal homs in $C$ are computed pointwise because of our assumption. 

Now, if we tried to prove that $f$ is trace-class, we would need to provide a lift: 
\[\begin{tikzcd}
	& {\hom_C(x,\one_C)\otimes y} \\
	{\one_C} & {\hom_C(x,y)}
	\arrow[from=1-2, to=2-2]
	\arrow[dashed, from=2-1, to=1-2]
	\arrow[from=2-1, to=2-2]
\end{tikzcd}\]
which, since internal homs and tensors are computed pointwise, would amount to lifts: 
\[\begin{tikzcd}
	& {\hom_{C_i}(x_i,\one_{C_i})\otimes y_i} \\
	{\one_{C_i}} & {\hom_{C_i}(x_i,y_i)}
	\arrow[from=1-2, to=2-2]
	\arrow[dashed, from=2-1, to=1-2]
	\arrow[from=2-1, to=2-2]
\end{tikzcd}\]
for $i=0,1$, \emph{together with} a homotopy between the induced lifts in $C_{01}$. 
\begin{rmk}
    This is the crucial difference between trace-class maps and dualizable objects, since in the case of dualizable objects the lifts are unique!
\end{rmk}
The lifts are there by assumption, but we cannot in general get a homotopy between them. 

Since $y\to z$ is also trace-class in $C_{01}$, \Cref{add:doubletrcl} guarantees that there \emph{is} a homotopy between the induced lifts for $gf$, and thus that $gf$ admits a trace-class witness in $C$, as was needed. 
\end{proof}
\begin{cor}\label{cor:trclinlaxpb}
    Let $C_0\xrightarrow{p_0}C_{01}\xleftarrow{p_1}C_1$ be a span of left adjoint symmetric monoidal functors between symmetric monoidal categories, which admit internal homs\footnote{This time, we do \emph{not} assume that $p_i$'s preserve internal homs}. Let $C$ denote their pullback and $\overset{\rightarrow}{C}$ the following doubly lax pullback: $$\overset{\rightarrow}{C} := \laxfib(p_0)\times_{C_{01}}\laxfib(p_1)$$
    There is a canonical inclusion $C\to \overset{\rightarrow}{C}$. If $f:x\to y, g:y\to z$ are maps in $C$ whose projection to each $C_i$ is trace-class, then $gf$ is trace-class in $\overset{\rightarrow}{C}$. 
\end{cor}
\begin{proof}
    By \Cref{lm:hominlaxfib}, the span $\laxfib(p_0)\to C_{01}\leftarrow \laxfib(p_1)$ fits the hypotheses of the previous corollary. It thus suffices to show that the projection of $f$ (resp. $g$) to $\laxfib(p_i)$ is trace-class. 

    But this projection is the image under the canonical functor $C_i\to \laxfib(p_i), c\mapsto (c,p_i(c),\id: p_i(c)=p_i(c))$ of the projection of $f$ (resp. $g$) to $C_i$. Thus, since that projection is assumed trace-class, so is its image in $\laxfib(p_i)$. 
\end{proof}
Finally, we need to indicate how lax pullbacks interact with rigidification. The following is the key lemma: 
\begin{lm}
    Let $f,g: \V\to \W$ be two morphisms in $\CAlg(\PrL_\st)$, with $\V$ rigid. Any morphism of symmetric monoidal functors $\eta: f\to g$ is invertible. 
\end{lm}
This lemma is the rigid analogue of a classical result for small symmetric monoidal categories where all objects are dualizable (and hence for rigidly compactly generated categories as well), except here we do not necessarily have a sufficient supply of dualizable objects in $\V$. 
\begin{proof}
For a trace-class map $\alpha: x\to y$ in $\V$, we construct a dotted lift in the following diagram: 

\[\begin{tikzcd}
	{f(x)} & {g(x)} \\
	{f(y)} & {g(y)}
	\arrow["{\eta_x}", from=1-1, to=1-2]
	\arrow["{f(\alpha)}"', from=1-1, to=2-1]
	\arrow[dashed, from=1-2, to=2-1]
	\arrow["{g(\alpha)}", from=1-2, to=2-2]
	\arrow["{\eta_y}"', from=2-1, to=2-2]
\end{tikzcd}\]
From there, using \Cref{cor:trclexQ} and the fact that $f,g$ commute with sequential colimits, we find that $\eta$ is an equivalence on $\omega_1$-compacts and thus on all objects because $f,g$ commute with all colimits.
\end{proof}
With this in hand, we can now prove: 
\begin{cor}\label{cor:rigpb=laxpb}
    Let $\V_0\xrightarrow{p_0} \V_{01}\xleftarrow{p_1} \V_1$ be span in $\CAlg(\PrL_\st)$, and let $\V$ be their pullback, $\overset{\rightarrow}{\V}$ be their lax pullback as in \Cref{cor:trclinlaxpb}. The canonical inclusion $\V\to \overset{\rightarrow}{\V}$ induces an equivalence on rigidification. 
\end{cor}
\begin{proof}
    This follows at once from the previous lemma: a symmetric monoidal functor $f:\W\to \overset{\rightarrow}{\V}$ is equivalently three functors $f_i:\W\to \V_i,i\in\{0,1,01\}$ with transformations $p_if_i\to f_{01}$ of symmetric monoidal functors $\W\to \V_{01}$. 

    If $\W$ is rigid, these are therefore equivalences, and so the functor $f$ lands in $\V$ (note that $\V\to \overset{\rightarrow}{\V}$ is fully faithful). 
\end{proof}
We are finally equipped to give a proof of \Cref{cor:limrig}: 
\begin{proof}[Proof of \Cref{cor:limrig}]
By definition of rigidity, this forgetful functor factors as a composite $\CAlg\rig\to \CAlg(\Prdbl)\to \Prdbl$ where the second functor preserves limits for general reasons; thus we are really considering the first functor $\CAlg\rig\to \CAlg(\Prdbl)$. Now, again by \Cref{obs:morrig} this is a full subcategory inclusion, so it suffices to prove that it is closed under limits. 

We do it for products and finite limits separately. 

\textbf{Products} : Let $(\V_i)_i$ be a family of rigid categories. Recall that $\prod_i^\dbl \V_i \subset \Ind(\prod_i \V_i)$ is generated under $\mathbb Q_{\geq 0}$-indexed colimits of diagrams of the form $(y((x_i^r)_i))_{r\in\mathbb Q_{\geq 0}}$ where for all $r<s$ in $\mathbb Q_{\geq 0}$ and all $i\in I$, $x_i^r\to x_i^s$ is compact. 

Since $\V_i$ is rigid, this implies that it is trace-class, and thus the whole of $(x_i^r)_i \to (x_i^s)_i$ is trace-class in $\prod_i\V_i$. By the usual trick, it implies that $\colim_{r\in\mathbb Q_{\geq 0}} y((x_i^r)_i)$ is also a $\mathbb Q_{\geq 0}$-indexed colimit of things in $\prod_i^\dbl\V_i$ along trace-class maps, and thus it follows that $\prod_i^\dbl\V_i$ is rigid by \Cref{cor:Efimovrig}.  

\textbf{Finite limits:} The empty limit is the terminal object, i.e. $0$, and this is clearly rigid. So we are left with pullbacks: let $\V_0\xrightarrow{p_0} \V_{01}\xleftarrow{p_1} \V_1$ be a span in $\CAlg\rig$, and $\V$ be its pullback in $\CAlg(\Prdbl)$, $\V'$ its pullback in $\CAlg(\PrL_\st)$, and $\overset{\rightarrow}{\V}$ the doubly lax pullback as in \Cref{cor:trclinlaxpb}.  

We have canonical maps $\V\to \V'\to \overset{\rightarrow}{\V}$ and therefore a diagram: 
\[\begin{tikzcd}
	{\Rig(\V)} & {\Ind(\V)} \\
	{\Rig(\V')} & {\Ind(\V')} \\
	{\Rig(\overset{\rightarrow}{\V})} & {\Ind(\overset{\rightarrow}{\V})}
	\arrow[from=1-1, to=1-2]
	\arrow[from=1-1, to=2-1]
	\arrow[from=1-2, to=2-2]
	\arrow[from=2-1, to=2-2]
	\arrow[from=2-1, to=3-1]
	\arrow[from=2-2, to=3-2]
	\arrow[from=3-1, to=3-2]
\end{tikzcd}\]

The unit in $\V$ is compact, so $\Rig(\V)\to \V$ is fully faithful, and it suffices to prove that it is surjective. 

Since pullbacks are finite limits, $\V\to \V'$ is fully faithful, and therefore this is also the case on $\Ind$; $\V'\to \overset{\rightarrow}{\V}$ is also clearly fully faithful. 

Furthermore, we have observed in \Cref{cor:rigpb=laxpb} that the bottom left vertical map is an isomorphism, and it is clear from \Cref{obs:morrig} that the top left vertical map also is. Thus it suffices to prove that the image of $\hat y: \V\to \Ind(\V)\to \Ind(\overset{\rightarrow}{\V})$ is contained in the image of $\Rig(\overset{\rightarrow}{\V})$. 

However we already know that the image of $\V$ in $\Ind(\V')$ is generated under colimits by $\mathbb Q$-indexed colimits of maps between representables whose projections to each $\V_i$ are trace-class. Since each map in $\mathbb Q$ is a composite of two maps, it follows from \Cref{cor:trclinlaxpb} that each of the maps in question is trace-class in $\overset{\rightarrow}{\V}$, and hence, that the relevant $\mathbb Q$-colimits are in $\Rig(\overset{\rightarrow}{\V})$, as was to be shown. 
\end{proof}

\subsubsection{In general}
We will prove in \Cref{section:pres} that the category of rigid $\V$-algebras is presentable, and explain how to obtain from this a different proof of the existence of rigidifications. However, it is still not explicit enough to obtain \Cref{thm:locrigrigleft} - the goal of this subsection is to provide an actual \emph{construction} of the rigidification for which we have enough control to prove this result. This construction is motivated by the construction over $\Sp$, and just as in this case, we first need an ``intrinsic'' description of rigid categories, obtained as \Cref{lm:rigiditythroughpres}.

Our goal is ultimately to prove that for a locally rigid $\V$-algebra $\W$, the canonical functor $\Rig_\V(\W)\to \W$ admits a \emph{fully faithful} left adjoint. We have already noted that $\Rig_\V(\W)\to \W$ factors in a fully faithful way through $\PP_\V(\W^\kappa)$, so we split the assignment in two halves: a- for a $\tilde{\W}$ with $\V$-atomic unit (to be thought of as $\PP_\V(\W^\kappa)$, so in particular, not assumed locally rigid), fully describe $\Rig_\V(\tilde \W)$ as a full subcategory of $\tilde \W$; b- Use this description to prove that for $\tilde\W := \PP_\V(\W^\kappa)$, $\hat y:\W\to\PP_\V(\W^\kappa)$ factors through $\Rig_\V(\PP_\V(\W^\kappa)) =\Rig_\V(\W)$. The second step will be rather easy once the first is established, so we now focus on the first. 

We already noted in \Cref{cor:cpctunitrigff} that if $\one_\W$ is atomic, then $\Rig_\V(\W)\to \W$ is fully faithful, so our goal is to describe the essential image. 

The following will be our motivation for the next construction: 
\begin{defn}
    Let $\W$ be a commutative $\V$-algebra. A map $f:x\to y$ is (weakly) trace-class presentable if it (factors through a map that) can be written as $\colim_I^W f \to y$ with $W$ a weight where the map $W\to \hom_\W^\V(f,y)$ factors through $\trcl_\W^\V(f,y)$ as a map of weights. 

    An object is (weakly) trace-class presentable if its identity map is (weakly) trace-class presentable, and, as usual, a trace-class presentation is a witness thereof. 
\end{defn}
\begin{rmk}
    The notion of trace-class presentability in $\W$ a priori depends on the base $\V$, even though it is not apparent in the name. 
\end{rmk}
\begin{thm}\label{lm:rigiditythroughpres}
    Suppose $\W$ is a commutative $\V$-algebra with $\V$-atomic unit. $\W$ is rigid if and only if it satisfies an analogue of \Cref{thm:dblexhaustionV} with trace-class presentability, that is, if and only if: \begin{enumerate}
    \item Every object in $\W$ is trace-class presentable;
    \item Every object in $\W$ is weakly trace-class presentable; 
    \item Every object in $\W$ is a sequential colimit along trace-class presentable maps; 
    \item Every object in $\W$ is a sequential colimit along weakly trace-class presentable maps.
  
    \end{enumerate}
    Analogous statements where one replaces ``every object is of a certain form'' by ``objects of a certain form generate $\W$ as a $\V$-module'' are also equivalent. 
\end{thm}
As in the case of \Cref{thm:dblexhaustionV}, it is worth starting with a preparatory lemma: 
\begin{lm}\label{lm:trclinterchange}
    Let $\W$ be a $\V$-algebra and $\alpha: x\to y$ be a trace-class presentable map in $\W$. Let $f:\M\to \N$ be a lax $\W$-linear functor whose restriction of scalars to a lax $\V$-linear functor is strong $\V$-linear and colimit-preserving. 

    For any $m\in \M$, there is a dotted lift in the commutative diagram below: 
    \[\begin{tikzcd}
	{x\otimes f(m)} & {f(x\otimes m)} \\
	{y\otimes f(m)} & {f(y\otimes m)}
	\arrow[from=1-1, to=1-2]
	\arrow[from=1-1, to=2-1]
	\arrow[dashed, from=1-2, to=2-1]
	\arrow[from=1-2, to=2-2]
	\arrow[from=2-1, to=2-2]
\end{tikzcd}\]
\end{lm}
\begin{proof}
    We construct a $\V$-natural map $\trcl_\W^\V(x,y) \to \hom_\M^\V(f(x\otimes m), y\otimes f(m))$ together with commutative diagrams: 
    \[\begin{tikzcd}
	{\trcl^\V_\W(x,y)} & {\hom_\M^\V(f(x\otimes m),y\otimes f(m))} \\
	{\hom_\W^\V(x,y)} & {\hom_\M^\V(x\otimes f(m),y\otimes f(m))}
	\arrow[from=1-1, to=1-2]
	\arrow[from=1-1, to=2-1]
	\arrow[from=1-2, to=2-2]
	\arrow[from=2-1, to=2-2]
\end{tikzcd}\]
and 
\[\begin{tikzcd}
	{\trcl^\V_\W(x,y)} & {\hom_\M^\V(f(x\otimes m),y\otimes f(m))} \\
	{\hom_\W^\V(x,y)} & {\hom_\M^\V(f(x\otimes m),f(y\otimes m))}
	\arrow[from=1-1, to=1-2]
	\arrow[from=1-1, to=2-1]
	\arrow[from=1-2, to=2-2]
	\arrow[from=2-1, to=2-2]
\end{tikzcd}\]

Once we have this lift, the proof becomes easy: write $f(x\otimes m)= f(\colim_I^W x_i\otimes m)$ for some trace-class presentation $\colim_I^W x_i\to y $ of $\alpha: x\to y$, use the $\V$-linearity of $f$ to rewrite this as $\colim^W_I f(x_i\otimes m)$ and then use the maps $W\to\trcl_\W^\V(x_\bullet,y)\to \hom(f(x_\bullet \otimes m), y\otimes f(m))$ to get a map $\colim_I^W f(x_\bullet\otimes m)\to y\otimes f(m)$. The verification of the triangles commuting comes down to the two claimed commuting diagrams. 

Now to construct the map in question, use the fact that we have a natural map $\hom_\W^\V(\one_\W, w)\otimes~n\to~w\otimes~n$ by definition of restriction of scalars, so that we have a composite : $$\trcl_\W^\V(x,y)\otimes f(x\otimes m)\to \hom_\W^\W(x,\one_\W)\otimes y\otimes f(x\otimes m)\to y\otimes f(\hom_\W^\W(x,\one_\W)\otimes x\otimes m)\to y\otimes f(m)$$
The mate of this composite is the desired transformation. 

The verification that the diagrams commute is immediate if tedious. 
    
\end{proof}

\begin{proof}[Proof of \Cref{lm:rigiditythroughpres}]
First, we check that any of the assumptions implies that $\W$ is rigid. 

\Cref{lm:trclinterchange} shows that if $f:\M\to \N$ is a lax $\W$-linear functor which is strong $\V$-linear and colimit-preserving, then any of the assumptions 1.-4. or their ``generated by'' analogues imply that $f$ is strong $\W$-linear. 

Thus, by considering their right adjoints, $\W$-linear functors that are $\V$-internal left adjoints are also $\W$-internal left adjoints. It therefore suffices to prove that $\W\otimes_\V \W\to \W$ is a $\V$-internal left adjoint to prove rigidity. 

First we observe that there is a map $\trcl_\W^\V\to \at_\W^\V$ since $\one_\W$ is atomic. Therefore, our assumptions imply the assumptions of \Cref{thm:dblexhaustionV}, so that $\W$ is dualizable over $\V$. We may therefore apply \Cref{add:enough} to check whether this is a $\V$-internal left adjoints. 

Now we note that tensor products of trace-class maps are trace-class - a refinement of this is the existence of a transformation $\trcl_\W^\V(w_0,w_1)\otimes \trcl_\W^\V(w_0',w_1')\to \trcl_\W^\V(w_0\otimes w_0', w_1\otimes w_1')$. Thus $\mu: \W\otimes_\V\W\to \W$ satisfies the assumptions of \Cref{add:enough}, so $\mu$ is a $\V$-internal left adjoint, as was to be shown. 

Conversely, suppose $\W$ is rigid. In this case, it is dualizable (by definition or by \Cref{cor:rigselfdual}) and the canonical map $\trcl_\W^\V\to \at_\W^\V$ is an equivalence by \Cref{cor:rigat=at} and \Cref{ex:VatV}. Therefore \Cref{thm:dblexhaustionV} shows that every object in $\W$ is trace-class presentable. 
\end{proof}
\begin{cons}
    Let $\W$ be a commutative $\V$-algebra with $\V$-atomic unit. In this case we have a canonical comparison map $\trcl^\V_\W(x,y):=\hom^\V_\W(\one_\W, \hom_\W^\W(x,\one)\otimes y)\to \at_\W^\V(x,y)$. 

    We define $\W^{(\alpha)}$ by induction as follows. $\W^{(0)}=\W$, $\W^{(\alpha+1)}$ is generated under colimits and tensor products by objects $x\in\W^{(\alpha)}$ that can be written as $\colim_I^W x_i$ with $x_i\in (\W^{(\alpha)})^\kappa$ and $W\to \trcl^\V_{\W^{(\alpha)}}(x_\bullet, x)$, $W$ a $\kappa$-compact weight. 

    At limit stages, $(\W^{(\alpha)})^\kappa$ is defined as $\bigcap_{\beta<\alpha}(\W^{(\beta)})^\kappa$. 

For the same reason as in \Cref{cor:stabilize}, for $\alpha$ large enough $\W^{(\alpha+1)}= \W^{(\alpha)}$. 
\end{cons}

\begin{lm}
    If $\W^{(\alpha+1)}= \W^{(\alpha)}$, then $\W^{(\alpha)}$ is rigid. 
\end{lm}
\begin{proof}
    This follows at once from \Cref{lm:rigiditythroughpres}.
\end{proof}
\begin{obs}
    If $\W_0\to \W$ is a map in $\CAlg_\V$, $\W_0$ rigid, then for all $\alpha$, it factors through $\W^{(\alpha)}$. 
\end{obs}
\begin{cor}
    If If $\W^{(\alpha+1)}= \W^{(\alpha)}$, then $\W^{(\alpha)}$ is the rigidification of $\W$. 
\end{cor}
\begin{cor}
    Let $\W$ be a locally rigid $\V$-algebra. Then $\hat y: \W\to \PP_\V(\W^\kappa)$ factors through the rigidification of $\PP_\V(\W^\kappa)$. In particular, the functor $\Rig_\V(\W)\to \W$ admits a fully faithful left adjoint. 
\end{cor}
\begin{proof}
    The ``in particular'' part follows from the fact that the functor $\PP_\V(\W^\kappa)\to \W$ induces an equivalence $\Rig_\V(\PP_\V(\W^\kappa))\xrightarrow{\simeq}\Rig_\V(\W)$ by \Cref{lm:rig=rigP}.

For the first part of the claim, we proceed as follows. Let $x\in \W$. We note that $\hat y(x) = \colim^{\at_\W^\V(-,x)}_{\W^\kappa}y(m) \simeq \colim^{\at_\W^\V(-,x)}_{\W^\kappa}\hat y(m)$.

 Now, we have a map $\hom(y(m),\hat y(x))\simeq \at(m,x)\to \trcl(m,x)$.

 Note that by comparing their restriction to $\W^\kappa$, we find that $p_\V \circ \PP_\V(\trcl(m,-)) : \PP_\V(\W^\kappa)\to \V$ is exactly $\trcl(y(m),-)$ so that $\trcl(y(m),\hat y(-))$ is the universal $\V$-linear approximation to $\trcl(m,-)$, by \Cref{prop:descriptionofrightadj}. 

 Thus we obtain a map $\at(m,x)\to \trcl(y(m),\hat y(x))$. Using the natural map $\hat y(m)\to y(m)$ we get in total a map $\at(m,x)\to \trcl(\hat y(m),\hat y(x))$. This map clearly lies above $\hom(m,x)$, and since $\hom(\hat y(m),\hat y(x))\to \hom(m,x)$ is an equivalence, it also lies above $\hom(\hat y(m),\hat y(x))$. This shows by induction that $\hat y$ lands in $\PP_\V(\W^\kappa)^{(\alpha)}$ for all $\alpha$, as was needed.  

\end{proof}

\section{Presentability results}\label{section:pres}
In this section, we prove an analogue of our main theorem for rigid categories, namely we prove that the category of rigid $\V$-algebras is presentable. In the case of $\V=\Sp$, we have an explicit description of the right adjoint $\CAlg\rig_\Sp\to \CAlg(\PrL_{\st,\kappa})$ by \Cref{thm:rigidification} which allows to give a quick proof of presentability, so we first deal with this special case. 

In general, we have to work harder and give a ``parametrized'' version of \cite[Theorem B]{maindbl}.
\subsection{Over $\Sp$}
In this subsection, we prove an analogue of our main theorem for rigid categories:
\begin{thm}\label{thm:rigprlst}
   Let $\CAlg\rig\subset \CAlg(\PrL_{\st})$ denote the full subcategory spanned by rigid $\Sp$-algebras. 

    $\CAlg\rig$ is presentable. 
\end{thm}
\begin{rmk}
    Unlike in the case of dualizable categories, we do not have to restrict the morphisms, because any $\Sp$-algebra map between rigid $\Sp$-algebras is an internal left adjoint, cf. \Cref{obs:morrig}. 
\end{rmk}

We recall the following classical lemma:
\begin{lm}\label{lm:corefl}
    Let $i: A\to B$ be a fully faithful embedding, where $i$ admits a right adjoint $i^R$. 
    \begin{enumerate}
    \item For any diagram $f: I\to A$ such that $i\circ f$ admits a colimit in $B$, $f$ admits a colimit too (which is automatically preserved by $i$); in particular, if $B$ is cocomplete, then so is $A$; 
        \item If $B$ is accessible and $ii^R : B\to B$ is accessible, then $A$ is accessible.
    \end{enumerate}
    In particular, if $B$ is presentable and $ii^R$ is accessible, then $A$ is presentable too. 
\end{lm}
\begin{proof}
(1) We start by noting that, as a full subcategory of $B$, $A$ consists exactly of those $b$'s such that for any $p:x\to y$ such that $i^R(p)$ is an equivalence, $\Map(b,x)\to\Map(b,y)$ is also an equivalence. 

So now let $f:I\to A$ be as in the statement, and assume that $i\circ f$ admits a colimit $C$. Let $p:x\to y$ be as above, and we then note that $\Map(C,x)\to \Map(C,y)$ is identified with $\Map(i\circ f, x)\to \Map(i\circ f, y)$, which in turn is identified with $\Map(f, i^R(x))\to \Map(f,i^R(y))$ and is thus an equivalence, as was to be shown.

(2) $A$ is the pullback of the cospan $B\to B^{\Delta^1} \leftarrow B$ where $B\to B^{\Delta^1}$ classifies the co-unit $ii^R\to \id_B$ and $B^{\Delta^1}\leftarrow B$ is the diagonal functor. 

If $ii^R$ is accessible, then so is the map $B\to B^{\Delta^1}$ classifying the co-unit. 

By \cite[Proposition 5.4.6.6.]{HTT}, it follows that $A$ is accessible. 
\end{proof}

\begin{proof}[Proof of \Cref{thm:rigprlst}]
     By \cite[Theorem 3.4]{maindbl} and \Cref{obs:morrig}, we have an inclusion $\CAlg\rig\subset \CAlg(\PrL_{\st,\omega_1})$, and this furthermore has a right adjoint $\Rig(-)$ by \Cref{thm:rigidification}. By \Cref{lm:corefl}, it suffices to prove that $\Rig(-)$, viewed as an endofunctor of $\CAlg(\PrL_{\st,\omega_1})$, is accessible. 

     Since $\CAlg(\PrL_{\st,\omega_1})\to\PrL_{\st,\omega_1}$ preserves filtered colimits and is conservative, it suffices to check accessibility after forgetting the algebra structure. To do so, we first point out that the functor $\M\mapsto \M^{\omega_1}$, with values in categories, is $\omega_1$-accessible; and therefore so is $\M\mapsto \Ind(\M^{\omega_1})$. We also point out that $\mathbb Q_{\geq 0}$ is countable.
     
     Finally, let us note that if the unit of $\M$ is $\kappa$-compact, then ``$f:x\to y$ is a trace-class map in $\M$'' can be witnessed by a $\kappa$-compact object $d\in \M$, in the notation of \Cref{defn:trcl}. It follows that on $\CAlg(\PrL_{\st,\omega_1})$, the full subfunctor of $\M\mapsto(\M^{\omega_1})^{\mathbb Q_{\geq 0}}$ spanned by those diagrams where all transition maps are trace-class is an $\omega_1$-accessible subfunctor. 
     
    Together, these facts show that the subfunctor $\Rig_{\omega_1}(\M)\subset\Ind(\M^{\omega_1})$ is also $\omega_1$-accessible. 
\end{proof}
\subsection{In general}\label{section:rigprlV}
Over $\Sp$ we had an explicit right adjoint and were thus able to prove presentability by hand. In general, we must work somewhat differently. 

Nonetheless, we still have:
\begin{thm}\label{thm:rigprl}
    Let $\CAlg\rig_\V\subset \CAlg(\Mod_\V(\PrL))$ be the full subcategory spanned by rigid $\V$-algebras. 
    $\CAlg\rig_\V$ is presentable. 
\end{thm}
\begin{warn}\label{rmk:mistake}
    An earlier version of this paper contained a similar claim for a variant of the category of locally rigid $\V$-algebras. The claim in question was mistaken, as was pointed out to the author by Jiacheng Liang, and currently we do not know a way to make any such claim both nontrivial and true. Such a result might be provable if one could prove an analogue of \Cref{prop:colimlocrig} where the functors are only required to preserve $\kappa$-compacts for some fixed $\kappa$. 
\end{warn}

We have seen in \Cref{prop:colimlocrig} that $\CAlg\rig_\V$ admits all small colimits. Thus this really is a claim about accessibility. We will prove this by witnessing the category of rigid commutative algebras as a pullback of accessible categories along accessible functors. 

The relevant result here is an immediate corollary of \cite[Theorem 10.3]{GHN}, which we spell out for the reader's convenience:
\begin{thm}
    Let $B$ be a presentable category, and $e:B\to \PrL$ be an accessible functor. The unstraightened cocartesian fibration $E\xrightarrow{p}B$ has a presentable total category $E$. 
\end{thm}
\begin{proof}
    This follows from \cite[Theorem 10.3]{GHN}, and the equivalence $(\PrR)\op\simeq\PrL$ \cite[Corollary 5.5.3.4]{HTT}.
\end{proof}
Ideally, we would like to apply this directly to $\W\mapsto \Dbl{\W}$. However, we were not able to show directly that this functor satisfies the desired assumptions. So we go on a detour through parametrized comonadicity. First we note that:
\begin{lm}
    Let $C\in\CAlg(\PrL)$. The functor $\Alg(C)\to \PrL, A\mapsto \Mod_A(C)$ is accessible (in fact, $\omega$-accessible), and hence so is its restriction to $\CAlg(C)$. 
\end{lm}
\begin{proof}
    This follows from \cite[Theorem 4.8.5.13]{HA}. 
\end{proof}

\begin{cor}\label{cor:prlfib}
    Let $\Mod^\V_\kappa \xrightarrow{p}\CAlg(\Mod_\V(\PrL_\kappa))$ be the cocartesian fibration classifying the functor $\W\mapsto \Mod_\W(\PrL_\kappa)$. 

    For $\V\in\CAlg(\PrL_\kappa)$ $\Mod^\V_\kappa$ is presentable. 
\end{cor}
We wish to prove the same for the cocartesian fibration $\Mod^\V_{\kappa,\dbl}$ classifying the functor $\W\mapsto \Dbl{\W}$. For this we note that the colimit-preserving functor $\Dbl{\W}\to \Mod_\W(\PrL_\kappa)$ induces a fiberwise left adjoint $\Mod^\V_{\kappa,\dbl}\to \Mod^\V_\kappa$ over $\CAlg(\Mod_\V(\PrL_\kappa))$. By the theory of relative adjunctions \cite[Proposition 7.3.2.6]{HA}, this functor admits a right adjoint $G$. We are thus in a position to ask whether it is comonadic. However, it is a morphism of fibrations, and so the more natural question is whether it is comonadic in a ``relative sense''. Luckily, Heine develops the theory of (co)monadicity relative to a base in \cite{heinemonad}. For us, the relevant result is:
\begin{thm}[{\cite[Theorem 5.43]{heinemonad}}]
    Let $F\dashv G: C \rightleftarrows D$ be a relative adjunction in $\Cat_{/S}$, for some category $S$. It is comonadic in $\Cat_{/S}$ if: 
    \begin{itemize}
        \item For every $s\in S$, the induced adjunction on fibers is comonadic; 
        \item For every $s\in S$, limits of $F_s$-split cosimplicial objects in $C_s$ are preserved by the inclusion $C_s\to C$.
    \end{itemize}
\end{thm}
\begin{rmk}
    The limits of $F_s$-split cosimplicial objects in $C_s$ are the ones that are required to exist in $C_s$, and to be preserved by $F_s$ for the ordinary comonadicity theorem, cf. \cite[Theorem 4.7.3.5]{HA}.
\end{rmk}
A key observation is that if $C$ is a bifibration, i.e. both a cocartesian and a cartesian fibration, the second condition is automatic:
\begin{lm}\label{lm:fwcomonadicrel}
    Let $p:C\to S$ be a cartesian fibration for which all the pullback functors preserve limits (e.g. if it is also cocartesian). In that case, the inclusion functors $C_s\to C$ preserve weakly contractible limits, for all $s\in S$. 

    In particular, if $F\dashv G: C\rightleftarrows D$ is a fiberwise comonadic relative adjunction over $S$, it is comonadic. 
\end{lm}
\begin{proof}
    This follows from \cite[Corollary 4.3.1.10]{HTT} together with the fact that a $p$-limit diagram lying above a limit diagram in $S$ is a limit diagram in $C$, and finally the fact that constant diagrams over weakly contractible categories have constant associated limit diagrams. 

    The ``In particular'' part simply follows from Heine's theorem \cite[Theorem 5.43]{heine}. 
\end{proof}
\begin{cor}\label{cor:relcomonadacc}
    The relative adjunction $\Mod^\V_{\kappa,\dbl}\rightleftarrows \Mod^\V_\kappa$ is comonadic in $\Cat_{/S}$, where $S=\CAlg(\Mod_\V(\PrL_\kappa))$. 

    The associated relative comonad is accessible, and in particular the source is accessible. 
\end{cor}
\begin{proof}
    We have explained the first part in \Cref{lm:fwcomonadicrel} as a consequence of \cite[Theorem B]{maindbl}. 

    The comonad we obtain fiberwise preserves $\max(\omega_1,\kappa)$-filtered colimits. We now explain how to prove that it is accessible on the whole of $\Mod^\V_\kappa$. By \cite[Corollary 4.3.1.11]{HTT}, $\Mod^\V_\kappa$ admits all colimits, and a diagram $f:I^\triangleright \to \Mod^\V_\kappa$ is a colimit diagram if and only if :
    \begin{enumerate}
        \item The projection $I^\triangleright \to \Mod^\V_\kappa\to \CAlg(\Mod_\V(\PrL_\kappa))$ is a colimit diagram; 
        \item Letting $\overline f: I^\triangleright \to \Dbl{\W_{f(\infty)}}$ denote the pushforward of $f$ to the fiber over the terminal object of $I^\triangleright$, $\overline f$ is a colimit diagram in $\Dbl{\W_{f(\infty)}}$.  
    \end{enumerate}
The second condition can be informally written as $\colim_I \W_\infty\otimes_{\W_i}\M_i\xrightarrow{\simeq} \M_\infty$. It follows that for $I$ filtered, under condition 1., condition 2. is equivalent to the requirement that the diagram, forgetted down to $\Mod_\V(\PrL_\kappa)$ along $(\W,\M)\mapsto \M$ be a colimit diagram. Indeed, if $\colim_I \W_i \simeq \W_\infty$, we have $\colim_I\W_\infty\otimes_{\W_i}\M_i\simeq \colim_I \M_i$ along the unit map.

    Let $f$ be such a diagram, informally denoted $(\W_i,\M_i)$. The value of the relative comonad $T$ in question is $(\W_i,\PP_{\W_i}(\M_i^\kappa))$. As $T$ is a fiberwise comonad, we only need to check condition 2. in the above, when $I$ is sufficiently filtered. Namely we need to prove that the canonical map $$\colim_I \PP_{\W_i}(\M_i^\kappa) \to \PP_{\W_\infty}(\M_\infty^\kappa)$$ is an equivalence. 

    For $I$ being $\kappa$-filtered, we have that $\colim_I \M_i^\kappa \simeq \M^\kappa_\infty$ as $\W_j^\kappa$-tensored categories, for every $j$, and thus as $\W_j$-enriched categories, for every $j$. As $\W_\infty\simeq \lim_{I\op}\W_i$ along the restriction functors, it follows that this map is also an equivalence of $\W_\infty$-enriched categories. 

    Comparing the universal properties of both sides, it follows that the map $\colim_I\PP_{\W_i}(\M_i^\kappa)\to~\PP_{\W_\infty}(\M_\infty^\kappa)$ is an equivalence of $\W_\infty$-modules, as was to be shown. Thus our relative comonad is ($\kappa$-)accessible. 

    From there, the result about accessibility of the category of comodules follows similarly as in \cite[Proposition 3.8]{maindbl}. 
\end{proof}
\begin{proof}[Proof of \Cref{thm:rigprl}]
Fix $\kappa\geq \omega_1$ so that $\V\in\CAlg(\PrL_\kappa)$. 
    By \Cref{prop:colimrig}, $\CAlg\rig_\V$ has all colimits, and so the theorem is about accessibility. 

    Any rigid $\V$-algebra is dualizable over $\V$ and hence $\kappa$-compactly generated by \cite[Theorem 3.4]{maindbl}.

    We note that we may express $\CAlg\rig_\V$ as a pullback: $$\CAlg(\Dbl{\V})\times_{(\Mod^\V_\kappa)^{\Delta^1}}(\Mod^\V_{\kappa,\dbl})^{\Delta^1}$$ where $\Mod^\V_\kappa, \Mod^\V_{\kappa,\dbl}$ are as in \Cref{cor:prlfib} and the discussion following it, and the functor $\CAlg(\Dbl{\V})\to (\Mod^\V_\kappa)^{\Delta^1}$ is the functor $\W\mapsto (\W\otimes_\V\W\to \W)$. 
    
    Now, from the description of colimits in $\Mod^\V_\kappa$ (cf. the proof of \Cref{cor:relcomonadacc}), one can check that this functor is accessible, and hence \Cref{cor:relcomonadacc} guarantees that the pullback is also accessible, as was needed. 
\end{proof}
We now point out that \Cref{thm:rigprl} allows us to prove more abstractly the existence of rigidifications, by combining the following two corollaries (which of course, we now logically do not need):
\begin{cor}
    For any cardinal $\kappa\geq \omega_1$ such that $\V\in\CAlg(\PrL_\kappa)$, the inclusion $\CAlg\rig_\V\to~\CAlg(\Mod_\V(\PrL_\kappa))$ admits a right adjoint $\Rig_{\V,\kappa}$. 
\end{cor}

As in the case of $\Sp$, we can prove abstractly that this is independent of $\kappa$:
\begin{cor}
    Let $\kappa\geq \omega_1$ and $\V\in\CAlg(\PrL_\kappa)$. Let $\lambda \geq \kappa$, and let $\W\in\CAlg(\Mod_\V(\PrL_\kappa))$. In this case, the canonical map $\Rig_{\V,\kappa}(\W)\to\Rig_{\V,\lambda}(\W)$ is an equivalence. 
\end{cor}
\begin{proof}
    If we can prove that any map $\W_0\to \W$ from a rigid $\V$-algebra lies in $\CAlg(\Mod_\V(\PrL_\kappa))$, then it will be clear that they have the same universal property. 

    Since any rigid $\V$-algebra is itself in $\CAlg(\Mod_\V(\PrL_\kappa))$ by \cite[Theorem 3.4]{maindbl}, this is really about the map $\W_0\to \W$ itself preserving $\kappa$-compacts. But note that this map is the composite: $\W_0\to \W_0\otimes_\V\W\to \W$, where the second map is an internal left adjoint by rigidity of $\W_0$ and by \Cref{prop:actintladj}, while the first map is $\W_0\otimes_\V(\V\to \W_0)$. Since the latter is in $\Mod_\V(\PrL_\kappa)$ by assumption, the tensor product is in there as well, and so we are done (this is simply a variation on \Cref{obs:morrig}).
\end{proof}
In particular, the common value of all the $\Rig_{\V,\kappa}(\W)$ for $\kappa$ large enough is the rigidification of $\W$ in the sense of \Cref{defn:rigidification}. 
\bibliographystyle{alpha}
\bibliography{Biblio.bib}
\end{document}